\documentclass{article}
\usepackage[utf8]{inputenc}
\usepackage{amsmath}
\usepackage{amssymb}
\usepackage{cite}
\usepackage{makeidx}
\usepackage{graphicx}
\usepackage{multirow}
\usepackage{booktabs}
\usepackage{appendix}
\usepackage{framed}
\usepackage{bm}
\usepackage[normalem]{ulem}
\usepackage{amsthm}
\usepackage{subfig}
\usepackage{diagbox}
\usepackage{algorithm}
\usepackage{color}
\usepackage{bigstrut}
\usepackage{algpseudocode}
\usepackage{amsfonts}
\usepackage[colorlinks,linkcolor=blue,anchorcolor=blue,citecolor=red]{hyperref}
\usepackage{enumerate}
\usepackage{cases}
\usepackage{mathrsfs}
\usepackage{grffile}
\newcommand{\UP[1]}{^{(#1)}}
\newcommand\pd[2]{\frac{\partial {#1}}{\partial {#2}}}
\newcommand\pa{\partial}
\newcommand\rd{\mathrm{d}}

\newcommand\bu{\boldsymbol{u}}
\newcommand\bg{\boldsymbol{g}}

\newcommand\bn{\boldsymbol{n}}

\newcommand\bv{\boldsymbol{v}}

\newcommand\bx{\boldsymbol{x}}

\newcommand{\bq}{\boldsymbol{q}}
\newcommand\bF{\boldsymbol{F}}
\newcommand{\He}{{\rm He}}
\newcommand\bT{\overline{T}}
\newcommand\oT{\overline{T}}
\newcommand\ou{\overline{\bu}}
\newcommand\ot{\overline{T}}
\newcommand\bbQ{\boldsymbol{Q}}

\newcommand\bbmM{\boldsymbol{\mM}}
\newcommand{\bz}{\boldsymbol{0}}

\newcommand\bbR{\mathbb{R}}
\newcommand\bbN{\mathbb{N}}
\newcommand\mB{\mathcal{B}}
\newcommand\mF{\mathcal{F}}

\newcommand\mQ{\mathcal{Q}}
\newcommand\mR{\mathbb{R}}

\newcommand\mM{\mathcal{M}}
\newcommand\Mq{\mM_q}
\newcommand\mH{\mathcal{H}}
\newcommand\mN{\mathbb{N}}

\newcommand\sv{\bv_{\ast}}
\newcommand\al{\alpha}
\newcommand\be{\beta}

\newcommand\aut{_{\al}^{\ou,\ot}}

\newcommand\but{_{\be}^{\ou,\ot}}

\newcommand{\mC}{\mathcal{C}}

\newcommand{\bbf}{{\bm{f}}}

\newcommand{\Li}{{\rm Li}}

\newtheorem*{definition}{Definition}

\newtheorem{lemma}{Lemma}

\newtheorem{remark}{Remark}

\numberwithin{equation}{section}
\graphicspath{{image/}}

\title{A highly efficient asymptotic preserving IMEX method for the quantum BGK equation}
\author{
Ruo Li\thanks{CAPT, LMAM \& School of Mathematical Sciences, Peking University, Beijing, China, 100871; Chongqing Research Institute of Big Data, Peking University, Chongqing, China, 401121, email: {\tt
      rli@math.pku.edu.cn}.},~~Yixiao Lu\thanks{HEDPS, Center for Applied Physics and Technology \& School of Mathematical
    Sciences, Peking University, Beijing, China, 100871, email: {\tt
      luyixiao@pku.edu.cn}.},~~Yanli Wang\thanks{Beijing Computational
    Science Research Center, email: {\tt ylwang@csrc.ac.cn}.}}
\begin{document}
\maketitle

\begin{abstract}
   This paper presents an asymptotic preserving (AP) implicit-explicit (IMEX) scheme for solving the quantum BGK equation using the Hermite spectral method. The distribution function is expanded in a series of Hermite polynomials, with the Gaussian function serving as the weight function. The main challenge in this numerical scheme lies in efficiently expanding the quantum Maxwellian with the Hermite basis functions. To overcome this, we simplify the problem to the calculation of polylogarithms and propose an efficient algorithm to handle it, utilizing the Gauss-Hermite quadrature. Several numerical simulations, including a 2D spatial lid-driven cavity flow, demonstrate the AP property and remarkable efficiency of this method. 
    \vspace*{4mm}

\end{abstract}


\section{Introduction}

The quantum Boltzmann equation models the evolution of a dilute quantum gas flow, which was initially derived by Uehling and Uhlenbeck in \cite{Uehling1933, Uehling1934}. It incorporates quantum effects that cannot be neglected for light molecules at low temperatures. This equation is now applied not only to low-temperature gases but also to model both bosons and fermions, potentially trapped by a confining potential.

The quantum Boltzmann equation is formulated in six-dimensional physical and phase space. The collision operator in this equation involves a five-dimensional integral, where the integrand is combined with complicated cubic terms. These complexities pose significant challenges in studying the quantum Boltzmann equation, both theoretically and numerically. Notably, the Bose-Einstein condensation is a phenomenon wherein the distribution function can exhibit finite blow-up or weak convergence towards Dirac deltas, even when the kinetic energy is conserved \cite{Mouton2023}.

In this context, we focus on the numerical methods to solve the quantum Boltzmann equation. The initial attempt, proposed in \cite{Pareschi2005}, utilizes the symmetric property to simplify the collision term. Subsequently, leveraging the convolution-like structure of the collision operator, the fast Fourier method for the quantum Boltzmann equation has been introduced in \cite{Hu2012, Filbet2012}. This method has been extended to the inhomogeneous case in \cite{HuYing2015, HuJin2015, HuWang2015, Wu2019}. Additionally, the diffusive relaxation system has been adopted in \cite{Jin2000, Klar1999} to approximate the full collision term, and a Fokker-Planck-like approximation has been proposed in \cite{Hu2012qFPL, Carrillo2020}. In \cite{Mouton2023}, the Fourier spectral method is employed for the quantum Boltzmann-Nordheim equation, particularly for describing the long-time behavior of Bose-Einstein condensation and Fermi-Dirac saturation. 

However, numerically computing the intricate collision operator can be quite expensive, making it difficult to handle high-dimensional problems with these methods.
In the classical case, the BGK collision model serves as a widely used surrogate model for the Boltzmann operator, approximating collisions through a simple relaxation mechanism. In quantum kinetic regimes, the quantum BGK model is also extensively adopted to approximate the original collision operator \cite{Reinhard2015, Nouri2008}, and has also been extended to the multi-species case \cite{Bae2021, Bae2023}. Several numerical methods have been developed to tackle the quantum Boltzmann equation with the BGK model, often referred to as the quantum BGK equation. For instance, in \cite{Yang2009}, the lattice Boltzmann method is employed for the quantum BGK equation. In addition, a macroscopic reduced model known as the 13-moment system has been derived for the quantum BGK equation using modified Hermite polynomials \cite{Yano2014, Di2017}.

In this work, we propose an asymptotic preserving (AP) scheme for solving the quantum BGK equation using the Hermite spectral method. A specially chosen expansion center is adopted in the Gauss weight function to generate the related Hermite polynomials, enhancing the approximation accuracy of the basis functions. This method has proven successful in solving the classical Boltzmann equation \cite{ZhichengHu2019}, and has been extended to address the collisional plasma scenarios \cite{FPL2020}. For the quantum BGK equation, a primary challenge of the Hermite spectral method lies in approximating the quantum equilibrium. We present a highly efficient algorithm to obtain the expansion coefficients of the equilibrium within the framework of the Hermite spectral method. The complex computations are eventually reduced to evaluating the value of the polylogarithm function, which can be further simplified into an one-dimensional integral.

In the numerical experiments, the simulations with periodic initial values are first tested, and the order of convergence validates the AP property of this numerical scheme. Subsequently, the Sod problem is implemented and the numerical results are compared with the solutions of the full quantum Boltzmann equation in \cite{Hu2015}. The excellent agreement implies that the quantum BGK model serves as a good approximation of the original collision operator. Finally, the mixing regime problem and a spatially 2-dimensional lid-driven cavity flow are conducted to further demonstrate the superiority of this Hermite spectral method. 

The rest of this paper is organized as follows. In Sec. \ref{sec:QBGK}, the quantum Boltzmann equation and the BGK collision model are introduced. Sec. \ref{sec:numerical} presents the general framework of the Hermite spectral method to solve the quantum BGK equation. A highly efficient algorithm to approximate the equilibrium is proposed in Sec. \ref{sec:int_alg} with several numerical experiments displayed in Sec. \ref{sec:experiment} to validate this Hermite spectral method. This paper concludes with some remarks in Sec. \ref{sec:conclusion} and additional content in App. \ref{sec:app}.

\section{Preliminaries}
\label{sec:QBGK}
In this section, we provide a concise overview of the quantum Boltzmann equation and discuss the quantum BGK model, which serves as a simplified collision operator for the quantum gas.

\subsection{Quantum Boltzmann equation}
The quantum Boltzmann equation governs the time evolution of the phase-space density $f(t, \bx, \bv)$, representing the probability of finding a quantum particle at time $t\geqslant 0$ in the phase-space volume $\rd \bx \rd \bv$. Here $\bx \in \Omega \subset \bbR^{D}$ is the dimensionless position variable, and $\bv \in \bbR^{D}$ is the dimensionless microscopic velocity variable. The dimensionless form of the quantum Boltzmann equation can is expressed as \cite{Uehling1933}
\begin{equation}
    \label{eq:Boltz}
    \pd{f}{t}+\bv\cdot\nabla_{\bx} f=\frac{1}{\epsilon}\mQ[f](\bv), \quad t\geqslant0, \quad \bx \in \Omega \subset \bbR^{D}, \quad \bv \in \bbR^{D},
\end{equation}
where $\epsilon$ is the Knudsen number, and $D$ is the dimension. 
$\mQ[f](\bv)$ represents the collision operator with quantum effect. The original collision term has the cubic form as follows: 
\begin{equation}
    \label{eq:operator}
       \mQ_{\rm q}[f](\bv)=\int_{\mR^{D}}\int_{\mathbb{S}^{D-1}}B(|\bg|,\sigma)\Big[f'f_{\ast}'\big(1-\theta_0f\big)\big(1-\theta_0f_{\ast}\big)-ff_{\ast}\big(1- \theta_0f'\big)\big(1-\theta_0f_{\ast}'\big)\Big]\rd\sigma \rd \sv,
   \end{equation}
where $\bg=\bv-\sv$, and $f, f_{\ast}, f'$ and $f_{\ast}'$ represent $f(t, \bx, \bv), f(t, \bx, \sv), f(t, \bx, \bv') $ and $f(t, \bx, \sv')$. $(\bv, \sv)$ and $(\bv', \sv')$ are the pre-collision and the post-collision velocity, respectively, which are determined by 
\begin{equation}
    \label{eq:post-velo}
    \left\{
    \begin{array}{l}
        \bv'=\frac{\bv+\sv}{2}+\frac{1}{2}|\bv-\sv|\sigma,\\[2mm]
        \sv'=\frac{\bv+\sv}{2}-\frac{1}{2}|\bv-\sv|\sigma,\\
    \end{array}
    \right.
\end{equation}
where $\sigma\in \mathbb{S}^{D -1}$ is the unit vector along $\bv'-\sv'$. The collision kernel $B$ is a non-negative function that depends only on $|\bg|$ and $\cos \omega$, where $\omega$ is the angle between $\sigma$ and $\bg$ \cite{Filbet2012}. The parameter $\theta_0$ indicates the type of particles \cite{Mouton2023},  which can be classified into three types:
\begin{itemize}
    \item when $\theta_0 = \hbar^{D} > 0$, the particles are the Fermi-Dirac gas, also referred to as the Fermi gas. Here, $\hbar$ represents the rescaled Planck constant \cite{Mouton2023}. For the Fermi gas, the Pauli exclusion principle gives us the inequality \cite{Filbet2012}
    \begin{equation}
        \label{eq:pauli}
        f \leqslant \frac{1}{\theta_0}.
    \end{equation}
    \item when $\theta_0 = -\hbar^{D} < 0$, the particles are the Bose-Einstein gas, or the Bose gas.
    \item when $\theta_0 = 0$, the collision model \eqref{eq:operator} reduces to the classical Boltzmann collision operator 
    \begin{equation}
        \label{eq:classical_col}
        \mQ_c[f](\bv)=\int_{\mR^{D}}\int_{\mathbb{S}^{D-1}}B(|\bg|,\sigma)\Big(f'f_{\ast}'-ff_{\ast}\Big)\rd\sigma \rd \sv,
    \end{equation}
    and the particles are the classical gases. 
\end{itemize}

In the dimensionless quantum Boltzmann equation \eqref{eq:Boltz}, the macroscopic variables such as density $\rho$, velocity $\bu$ and internal energy $e_0$ are related to the distribution function $f(t, \bx, \bv)$ through the following equations:
\begin{subequations}
    \label{eq:macro}
    \begin{align}
    \label{eq:macro1}
    &\rho(t, \bx)=\int_{\mR^{D}}f(t, \bx,\bv)\rd\bv , \\
    \label{eq:macro2}
    &\bu(t, \bx)=\frac{1}{\rho}\int_{\mR^{D}}\bv f(t, \bx,\bv)\rd\bv , \\
    \label{eq:macro3}
    &e_0(t, \bx)=\frac{1}{2\rho}\int_{\mR^{D}}|\bv-\bu|^2f(t, \bx,\bv)\rd\bv .
    \end{align}
\end{subequations}
Additionally, the stress tensor $\mathbb{P}$ and the heat flux $\bq$ are defined as 
\begin{equation}
    \label{eq:P_q}
    \mathbb{P} = \int_{\mR^{D}} \left[(\bv-\bu)\otimes(\bv-\bu)-\frac13|\bv-\bu|^2I\right] f \rd \bv, \qquad \bq = \frac{1}{2} \int_{\mR^{D}}(\bv - \bu)|\bv -\bu|^2 f \rd \bv. 
\end{equation}

Compared to the classical Boltzmann equation \eqref{eq:classical_col}, the quantum Boltzmann operator \eqref{eq:operator} exhibits cubic dependence on the distribution density $f$, and involves more nonlinearity. This makes the theoretical and numerical study of the quantum Boltzmann equation much more challenging \cite{Hu2015, Filbet2012, Mouton2023}. 

\subsection{The quantum BGK model}
Similarly to the classical kinetic theory, a BGK-type model \cite{CCer1988, Nouri2008} is introduced in the quantum case to facilitate the study in the near continuous fluid regime. This simplified model approximates the complex quantum collision model \eqref{eq:operator} with the relaxation form as follows:
\begin{equation}
\label{eq:BGK}
    \mQ_{\rm qBGK}[f](\bv) = \mM_q - f,
\end{equation}
which is referred to as the quantum BGK model. Substituting $\mQ[f]=\mQ_{\rm qBGK}[f]$ into \eqref{eq:Boltz} yields the quantum BGK equation. In \eqref{eq:BGK}, $\mM_q$ represents the local equilibrium, also known as the quantum Maxwellian:
\begin{equation}
    \label{eq:Maxwellian}
   \mM_q(t, \bx, \bv)\triangleq \frac{1}{|\theta_0|} \frac{1}{(z|\theta_0|)^{-1}\exp\left(\frac{(\bv-\bu)^2}{2T}\right)+ {\rm sign}(\theta_0)}=\frac{1}{z^{-1}\exp\left(\frac{(\bv-\bu)^2}{2T}\right)+\theta_0},
\end{equation}
where $z|\theta_0| > 0$ represents the fugacity, and $T>0$ is the temperature. Determining $z$ and $T$ will be discussed later in \eqref{eq:z_T}. $\Mq$ also satisfies $\mQ_q[\Mq] = 0$. 

For the Bose gas ($\theta_0 < 0$), ensuring the non-negativity of $\mM_q$ in \eqref{eq:Maxwellian} requires:
\begin{equation}
    \label{eq:theta_BE}
    z \theta_0 \in [-1, 0).
\end{equation}
In particular, when $z \theta_0 = -1$, Bose-Einstein condensation occurs, and the steady state differs from \eqref{eq:Maxwellian}, taking the form \cite{Hu2015, Mouton2023}
\begin{equation}
    \label{eq:condensation}
    \tilde{\mM}_q(t, \bx, \bv)=m_0\delta(\bv-\bu)+\frac{1}{|\theta_0|}\frac{1}{\exp\left(\frac{(\bv-\bu)^2}{2T}\right)-1},
\end{equation}
where $m_0$ is the critical mass, and $\delta(\cdot)$ is the Dirac delta function. For the Fermi gas ($\theta_0>0$), no additional constraint on $z$ is required to obtain a quantum Maxwellian $\Mq$. If $\theta_0=0$, $\mM_q$ reduces to the classical Maxwellian with macroscopic velocity $\bu$ and temperature $T$: 
\begin{equation}
    \label{eq:c_Maxwellian}
    \mM^{\bu, T}_c(\bv)=\frac{\rho}{ (2 \pi T)^{D/2}}\exp\left(-\frac{|\bv-\bu|^2}{2T}\right). 
\end{equation}

When $|\theta_0|$ is small, $\mM_q$ is close to $\mM^{\bu, T}_c$, and the quantum BGK model resembles the classical BGK model. For large $|\theta_0|$, $\mM_q$ behaves quite differently from $\mM^{\bu, T}_c$, and the quantum effect becomes significant. These phenomena are illustrated in \cite[Figs. 1 and 2]{Filbet2012}.

There are several important properties of the collision operator. Firstly, it conserves the total mass, momentum, and energy as \cite{Nouri2008, Hu2015}
\begin{equation}
    \label{eq:conserv}
    \int_{\bbR^{D}}\mQ[f](\bv)\left(
    \begin{array}{c}
        1 \\
        \bv \\
        |\bv|^2
    \end{array}
    \right)\rd\bv=0, \qquad \mQ[f](\bv) = \mQ_q[f](\bv),\; \mQ_{\rm qBGK}[f](\bv). 
\end{equation}
Moreover, letting
$$\varphi(\bv)=\ln\left(\frac{f(\bv)}{1-\theta_0f(\bv)}\right),$$
one can derive the H-theorem of the quantum Boltzmann equation as 
\begin{equation}
    \label{eq:H-thm}
    \int_{\bbR^{D}}\mQ[f](\bv)\varphi(\bv)\rd\bv\leqslant 0, \qquad \mQ[f](\bv) = \mQ_q[f](\bv),\; \mQ_{\rm q BGK}[f](\bv),
\end{equation}
and this equality holds if and only if $f$ attains the quantum Maxwellian \cite{Nouri2008, Hu2015}. In \eqref{eq:Maxwellian}, the parameters $z$ and $T$ can be obtained through the nonlinear system
\begin{equation}
    \label{eq:z_T}
    \begin{aligned}
     \int_{\mR^{D}} \mM_q(t, \bx, \bv) \rd \bv  & = \int_{\mR^{D}}f(t, \bx,\bv)\rd\bv = \rho, \\
     \int_{\mR^{D}}|\bv-\bu|^2\mM_q(t, \bx,\bv)\rd\bv & = \int_{\mR^{D}}|\bv-\bu|^2f(t, \bx,\bv)\rd\bv = 2\rho e_0.
    \end{aligned}
\end{equation}
For the degenerate Bose-Einstein case \eqref{eq:condensation}, the values of $m_0$ and $T$ can be computed as 
\begin{equation}
    \label{eq:BE_deg_T}
    T=\frac{2\zeta(\frac{D}{2})}{D\zeta\left(\frac{D+2}{2}\right)}e_0, \qquad m_0=\rho-\frac{(2 \pi T)^{D/2}}{|\theta_0|}\zeta\left(\frac{D}{2}\right),
\end{equation}
where $\zeta(s):=\Li_s(1)$ represents the Riemann zeta function. Furthermore, for the Bose-Einstein condensation steady state \eqref{eq:condensation}, the conservation of macroscopic variables \eqref{eq:conserv} and the H-theorem \eqref{eq:H-thm} persist.

In the following, we will propose a numerical scheme for the quantum BGK equation. Specifically, an Implicit-Explicit (IMEX) method will be introduced within the framework of the Hermite spectral method. Additionally, an efficient algorithm will be presented for the calculation of polylogarithms, which is crucial for deriving the expansion coefficients of the distribution function and solving the nonlinear system \eqref{eq:z_T}.



\section{Hermite spectral method for the quantum BGK equation}

\label{sec:numerical}
This section introduces the Hermite spectral method for the quantum BGK equation. We begin by discussing the approximation of the distribution function and deriving the moment system in Sec. \ref{sec:Hermite}. Subsequently, the numerical scheme with complete discretization is presented in Sec. \ref{sec:discretize}.

\subsection{Series expansion of the distribution function and the moment system}
\label{sec:Hermite}
To seek a polynomial spectral method for solving the quantum BGK equation, a natural approach is to consider $\Mq$ as the weight function. However, as shown in \cite{Yano2014}, the orthogonal polynomials with respect to $\Mq$ are quite complicated. 

It is observed that when $|\theta_0|$ is small, the classical Maxwellian $\mM_c^{\bu, T}$ serves as a good approximation to $\mM_q$. Therefore, the classical Maxwellian $\mM_c^{\bu, T}$ defined in \eqref{eq:c_Maxwellian} is chosen as the weight function, and the resulting orthogonal polynomials are the Hermite polynomials. These polynomials are defined as follows:
\begin{definition}[Hermite polynomials]
\label{def:Her}
For $\alpha=(\alpha_1,\alpha_2, \cdots)\in\bbN^{D}$, with $\ou \in \bbR^{D}$ and $\oT \in \bbR^+$, the three-dimensional 
Hermite polynomial $H_{\alpha}^{\ou,\oT}(\bv)$ is defined as
\begin{equation}
    \label{eq:Hermite}
    H_{\alpha}^{\ou,\oT}(\bv)=\frac{(-1)^{|\alpha|}\oT^{\frac{|\alpha|}{2}}}
    {\mM_c^{\ou,\oT}(\bv)}
    \dfrac{\partial^{|\alpha|}}{\partial \bv^{\alpha}}
    \mM_c^{\ou,\oT}(\bv), 
\end{equation}
with $|\alpha|=\sum_{d = 1}^{D} \alpha_d$, $\partial \bv^{\alpha} = \Pi_{d = 1}^{D}\pa v_{d}^{\al_d}$ and $\mM_c^{\ou,\oT}(\bv)$ defined in \eqref{eq:c_Maxwellian}. Here, $[\ou, \oT]$ is the expansion center, typically determined by a rough average over the entire spatial space.
\end{definition}
Then the distribution function $f$ can be expanded as 
\begin{equation}
    \label{eq:Her-expan}
    f(t,\bx,\bv)=\sum_{\alpha\in\mN^{D}}f_{\alpha}(t,\bx)
		\mH\aut(\bv),
\end{equation}
where $\mH\aut(\bv)=H\aut(\bv)\mM_c^{\ou,\oT}(\bv)$ are the Hermite basis functions. By truncating the expansion in \eqref{eq:Her-expan}, a finite approximation to the distribution function is obtained:
\begin{equation}
    \label{eq:trun-expan}
    f(t,\bx,\bv)\approx f_M(t,\bx,\bv)\triangleq\sum_{|\alpha|\leqslant M}f_{\alpha}(t,\bx)	\mH\aut(\bv),
\end{equation}
where $M$ is the expansion order. Similarly, the quantum Maxwellian \eqref{eq:Maxwellian} is approximated as 
\begin{equation}
    \label{eq:expan-BGK}
    \mM_q(t, \bx, \bv)\approx\sum_{|\alpha|\leqslant M}\mM_{q, \alpha}(t, \bx) \mH\aut(\bv).
\end{equation}
With the orthogonality of the Hermite polynomials 
\begin{equation}
    \label{eq:orth}
    \int_{\mR^3}H\aut(\bv)H\but(\bv)\mM_c^{\ou,\oT}(\bv)\rd\bv= \prod_{d=1}^{D}\alpha_d!
    \delta_{\alpha_d,\beta_d},
\end{equation}
the expansion coefficients $f_{\alpha}(t, \bx)$ and $\mM_{q, \alpha}(t, \bx)$ are calculated as 
\begin{align}
    \label{eq:falpha}
    f_{\alpha}(t, \bx) &= \frac{1}{\alpha!}\int_{\mR^{D}}f(t, \bx, \bv) H\aut(\bv)
    \rd\bv, \\
    \label{eq:Malpha}
    \mM_{q, \alpha}(t, \bx) &= \frac{1}{\alpha!}\int_{\mR^{D}}\mM_q(t, \bx, \bv) H\aut(\bv)
    \rd\bv.
\end{align}
With the Hermite expansion, the macroscopic variables defined in \eqref{eq:macro} and \eqref{eq:P_q} can be expressed using the expansion coefficients as 
\begin{equation}
    \label{eq:exp-macro}
    \begin{split}
        &\rho = f_{\bz},\quad u_k= \overline{u}_k+\frac{\sqrt{\oT}}{\rho}f_{e_k}, 
    \quad e_0=\frac{\oT}{\rho}\sum_{k=1}^{D}f_{2e_k}+\frac{D}{2}\oT-\frac{1}{2\rho}|\bu-\ou|^2, \\
    &p_{kl}=(1+\delta_{kl})\oT f_{e_i+e_j}+\delta_{kl}\rho\left(\oT - \theta\right)-
    \rho\left(\overline{u}_k-u_k\right)\left(\overline{u}_l-u_l\right), \\
    &q_k=2\oT^{\frac{3}{2}}f_{3e_k}+(\overline{u}_k-u_k)\oT f_{2e_k}+|\ou-\bu|^2\sqrt{\oT} f_{e_k} +\\ 
    & \qquad \qquad\sum_{d=1}^{D}\left[\oT^{\frac{3}{2}}f_{2e_d+e_k}+\left(\overline{u}_d-u_d\right)\oT f_{e_d+e_k}+  
     \left(\overline{u}_k-u_k\right)\oT f_{2e_d}\right], \qquad k, l = 1, 2, \cdots D.
    \end{split}
\end{equation}
Here, $e_d$ represents the unit vector. For example, when $D = 3$, $e_1=(1,0,0), e_2=(0,1,0), e_3=(0,0,1)$.  
Additionally, by substituting the Hermite expansion of $f$ and $\mM_{q}$ into \eqref{eq:Boltz}, we can obtain the moment system as
\begin{equation}
    \label{eq:moment}
    \pd{}{t}f_{\alpha}+\sum_{d=1}^{D}\pd{}{x_d}\left((\alpha_d+1)\sqrt{\ot}f_{\alpha+e_d}+\overline{u}_df_{\alpha}+\sqrt{\ot}f_{\alpha-e_d}\right)=\frac{1}{\epsilon}Q_{\alpha}, \qquad |\alpha|\leqslant M,
\end{equation}
where the collision term $Q_{\alpha}$ arises from the quantum BGK model \eqref{eq:BGK} and is expressed as
\begin{equation}
    \label{eq:exp_Q}
    Q_{\alpha} = \mM_{q, \alpha} - f_{\alpha}.
\end{equation}
Besides, the convection term is simplified with the recurrence relationship of Hermite polynomials as 
\begin{equation}
    \label{eq:recur}
    \begin{split}
    v_dH^{\ou,\oT}_{\alpha}=\sqrt{\oT}H^{\ou,\oT}_{\alpha+e_d}+
    u_dH^{\ou,\oT}_{\alpha}+\al_d\sqrt{\oT}H^{\ou,\oT}_{\alpha-e_d}.
    \end{split}
\end{equation}
The system \eqref{eq:moment} is closed with the constraint
\begin{equation}
    \label{eq:closure}
    f_{\alpha + e_d} = 0, \qquad |\alpha| = M.
\end{equation}

Let $\bbf = (f_{\bz}, f_{e_1}, f_{e_2}, f_{e_3}, \cdots)$ represent the vector of expansion coefficients of the distribution function $f$. \eqref{eq:moment} can be expressed in matrix form as
\begin{equation}
    \label{eq:moment1}
    \pd{\bbf}{t} + \sum_{d = 1}^{D}{\mathbf A_d} \pd{\bbf}{x_d} = \frac{1}{\epsilon} \bbQ, \qquad \bbQ = \bbmM_q - \bbf, 
\end{equation}
where $\bbQ = (Q_{\bz}, Q_{e_1}, Q_{e_2}, Q_{e_3}, \cdots)$ and $\bbmM_q = (\mM_{q, \bz}, \mM_{q, e_1}, \mM_{q, e_2}, \cdots)$. ${\mathbf A_d}$ is a matrix whose entries are decided by the convection coefficients in \eqref{eq:moment}.

The Hermite spectral method has been successfully employed to solve the classical Boltzmann equation \cite{ZhichengHu2019}, and has been extended to the plasma kinetic models \cite{FPL2020}. Following similar procedures, we will complete the full discretization of the moment system in Sec. \ref{sec:discretize}.

\subsection{Temporal and spatial discretization}
\label{sec:discretize}
In this section, we focus on the numerical scheme to discretize the moment system \eqref{eq:moment1} of the quantum BGK equation. We start with the temporal discretization, employing the implicit-explicit (IMEX) scheme to handle the stiff collision term.
\paragraph{Temporal discretization}
Assuming the numerical solution at time step $t^n$ is $\bbf^n$, then the temporal discretization for the first-order IMEX scheme takes the form
\begin{equation}
    \label{eq:AP_time}
        \frac{\bbf^{n+1}-\bbf^n}{\Delta t}+ \sum_{d = 1}^{D} {\mathbf A}_d \pd{\bbf^n}{x_d} = \frac{1}{\epsilon}\bbQ^{n+1}.
\end{equation}
In the simulation, \eqref{eq:AP_time} is split into 
\begin{itemize}
    \item convection step:
    \begin{equation}
        \label{eq:con_step}
        \frac{\bbf^{n+1, \ast}-\bbf^n}{\Delta t}+ \sum_{d = 1}^{D} {\mathbf A}_d \pd{\bbf^n}{x_d} = 0,
    \end{equation}
    \item collision step:
    \begin{equation}
        \label{eq:col_step}
        \frac{\bbf^{n+1}-\bbf^{n+1, \ast}}{\Delta t}=\frac{1}{\epsilon}\bbQ^{n+1}, \qquad \bbQ^{n+1} = \bbmM_q^{n+1} - \bbf^{n+1}. 
    \end{equation}
\end{itemize}
Since the collision conserves the total mass, momentum, and energy \eqref{eq:conserv}, it can be derived that $\bbmM_q^{n+1} = \bbmM_q^{n+1, \ast}$. Therefore, the convection step is first solved, and then $\bbmM^{n+1}_q$ is obtained based on $\bbf^{n+1, \ast}$. Finally, the collision step is solved with the computational cost of an explicit scheme. 

This first-order IMEX scheme can be easily extended into the high-order scheme, and we only present the second-order scheme below: 
\begin{subnumcases}
    {\label{eq:second}}
         \frac{\bbf^{n+1/2}-\bbf^n}{\Delta t/2}+ \sum_{d = 1}^{D} {\mathbf A}_d \pd{\bbf^n}{x_d} = \frac{1}{\epsilon}\bbQ^{n+1/2}, \\
         \frac{\bbf^{n+1}-\bbf^n}{\Delta t}+ \sum_{d = 1}^{D} {\mathbf A}_d \pd{\bbf^{n+1/2}}{x_d} = \frac{1}{2\epsilon}(\bbQ^{n+1}+\bbQ^n).
\end{subnumcases}
The same splitting method can also be applied to this second-order scheme. More high-order IMEX schemes can be referred to \cite{IMEX}. 

The time step length is chosen to satisfy the CFL condition
\begin{equation}
    \label{eq:CFL}
    {\rm CFL}\triangleq \frac{\Delta t}{\Delta x} \max_{d}\lambda({\mathbf{A}_d}) < 1,
\end{equation}
where $\lambda({\mathbf{A}_d})$ represents the spectral radius (i.e. the maximum absolute value of all the eigenvalues) of matrix ${\mathbf{A}_d}$. For further discussions about the eigenvalues of ${\mathbf{A}_d}$, we refer the readers to \cite{ZhichengHu2019}.

\paragraph{Spatial discretization}
For spatial discretization, the finite volume method is adopted for the moment system \eqref{eq:moment1}. Let the spatial domain $\Omega \subset \bbR^{D}$ be discretized by a uniform grid with cell size $(\Delta x_1, \Delta x_2, \cdots) \in \bbR^{D}$ and cell centers $\bx_k = (x_{k_1}, x_{k_2}, \cdots) \in \bbR^{D}$. Denoting $\bbf_{k}^n$ as the approximation of the average of $\bbf$ over the $k$-th grid cell at time $t^n$, the finite volume method for the convection step has the form
\begin{equation}
    \label{eq:FVM}
    \bbf_{k}^{n+1, \ast}=\bbf_{k}^n-\sum_{d=1}^{D}\frac{\Delta t}{\Delta x_d}\left(\bF_{k+\frac12 e_d}^n-\bF_{k-\frac12 e_d}^n\right),
\end{equation}
where $\bF_{k+\frac12 e_d}^n$ is the numerical flux computed by the HLL scheme \cite{HLL} with spatial reconstruction utilized. Detailed expressions can be found in \cite[Sec. 4.2]{multi2022}. 

With this spatial discretization, the numerical scheme to solve the collision step is given by
\begin{equation}
    \label{eq:col_fin}
     \frac{\bbf^{n+1}_k-\bbf^{n+1, \ast}_k}{\Delta t}=\frac{1}{\epsilon}\left( \bbmM_{q, k}^{n+1} - \bbf_k^{n+1}\right).   
\end{equation}

\begin{algorithm}[htbp]
    \caption{IMEX Hermite spectral method to solve the quantum BGK equation}
    \label{algo:algBGK}
    \begin{algorithmic}[1]
        \item {\bf Initialize:} Set $n = 0$, and choose an expansion center $[\ou, \oT]$ for the convection step. Calculate the initial value of $\bbf_{k}^{0}$ and $\bbmM_{q,k}^0$. 
        \item {\bf Step forward in one single time step:}
        \begin{algorithmic}[1]
            \item Determine the time step length $\Delta t^n$ satisfying the CFL condition \eqref{eq:CFL}. 
            \item Solve the convection step \eqref{eq:FVM} to obtain $\bbf^{n+1, \ast}_k$.
            \item Obtain the density $\rho^{n+1}_k$, macroscopic velocity $\bu_k^{n+1}$ and internal energy $e_{0, k}^{n+1}$ through \eqref{eq:macro}. 
            \item Obtain the parameters $z^{n+1}$ and $T^{n+1}$ through \eqref{eq:z_T}, and derive the quantum Maxwellian $\mM^{n+1}_{q, k}$. \label{step:zT}
            \item Calculate the expansion coefficients $\bbmM_{q, k}^{n+1}$ through \eqref{eq:Malpha}. \label{step:Maxwellian}
            \item Solve the collision step \eqref{eq:col_fin} to obtain $\bbf^{n+1}_k$. 
        \end{algorithmic}
    \end{algorithmic}
\end{algorithm}

So far, we have presented the complete discretization of the quantum BGK equation, and the entire algorithm is outlined in Alg. \ref{algo:algBGK}. However, it is important to note that compared with the classical case \cite{ZhichengHu2019}, obtaining the expansion coefficients of the quantum Maxwellian in Step \ref{step:Maxwellian} poses greater challenges. Additionally, one has to obtain the parameter $z$ and temperature $T$ through the nonlinear system \eqref{eq:z_T} in Step \ref{step:zT}. These two problems will be addressed in the following Sec. \ref{sec:int_alg}. 

\section{Expansion of the quantum Maxwellian}
\label{sec:int_alg}
In this section, we will delve into the strategies for accomplishing Steps \ref{step:zT} and \ref{step:Maxwellian} in Alg. \ref{algo:algBGK}. Specifically, the algorithm for obtaining the expansion coefficients is presented in Sec. \ref{sec:polylog}, and the approach to solving the nonlinear system \eqref{eq:z_T} is discussed in Sec. \ref{sec:solve}. 

\subsection{Algorithm to obtain \texorpdfstring{$\mM_{q, \alpha}$}{Mq}}
\label{sec:polylog}

To compute the expansion coefficients $\mM_{q, \alpha}(t, \bx)$ in \eqref{eq:Malpha}, we begin with the exact expansion of Hermite polynomials. From the definition of the one-dimensional Hermite polynomials, when $[\ou, \oT] = [\bz, 1]$, it takes the form below:
\begin{equation}
\label{eq:He_1d}
  \He_n(x)=(-1)^n\exp\left(\frac{x^2}2\right) \frac{\rd^n}{\rd x^n}\exp\left(-\frac{x^2}2\right),
\end{equation}
which can be precisely expressed as
\begin{equation}
    \label{eq:coef_Her}
    \begin{aligned}
        \He_{2n}(x) & =\sum_{k=0}^n\frac{(2n-1)!!}{(2n-2k-1)!!}
        (-1)^kC_n^kx^{2n-2k}, \\
        \He_{2n-1}(x)& =\sum_{k=0}^{n-1}\frac{(2n-1)!!}{(2n-2k-1)!!}
        (-1)^kC_{n-1}^kx^{2n-2k-1},
    \end{aligned}
\end{equation}
where the combination number $C_n^k$ is defined as
\begin{equation}
    \label{eq:comb_number}
    C_n^k=\frac{n!}{k!(n-k)!}.
\end{equation}
With the transitivity
\begin{equation}
    \label{eq:tran}
    H\aut(\bv)=H_{\alpha}^{\bz,1}\left(\sqrt{\frac{1}{{\;\bT\;}}}(\bv-\ou)\right),
\end{equation}
it holds that 
\begin{equation}
    \label{eq:form_Hermite}
    H\aut(\bv)=\sum_{\beta\in \bbN^{D}:\; \beta_i\leqslant \alpha_i}\mC^{[\ou,\oT]}(\alpha,\beta) \bv^{\beta},
\end{equation}
where $\mC^{[\ou,\oT]}(\alpha,\beta)$ are constants that can be directly calculated by $[\ou, \oT]$. Therefore, to obtain the expansion coefficients $\mM_{q, \alpha}$, we only need to compute the coefficients 
\begin{equation}
    \label{eq:mom_M_0}
  \mM_{\alpha} =   \int_{\bbR^{D}}\mM_q(\bv)\bv^{\alpha}\rd\bv, \qquad \bv^{\alpha} = \prod_{d=1}^{D}v_d^{\alpha_d}, \qquad |\alpha|  \leqslant M. 
\end{equation}
Then $\mM_{q, \alpha}$ is calculated as
\begin{equation}
    \label{eq:M_q_q}
    \mM_{q, \alpha} = \sum_{\beta \in \bbN^{D}:\; \beta_i\leqslant \alpha_i}\mC^{[\ou,\oT]}(\alpha,\beta)\mM_{\beta}, \qquad |\alpha| \leqslant M.
\end{equation}
Without loss of generality, we assume the macroscopic velocity $\bu=\bz$.
In this case, $\Mq$ is an even function of $\bv$, and it holds that 
\begin{equation}
    \label{eq:mom_M_1}
    \int_{\bbR^{D}}\mM_q(\bv)\bv^{\alpha}\rd\bv=0,
\end{equation}
if any entry of $\alpha$ is odd. 
When all entries of $\alpha$ are even, the expression of $\mM_{\alpha}$ is given by
\begin{equation}
    \label{eq:mom_M_2}
    \mM_{\alpha} = \int_{\bbR^{D}}\mM_q(\bv)\bv^{\alpha}\rd \bv
    = \int_{\bbR^{D}}z\exp\left(-\frac{\bv^2}{2T}\right)\frac{\bv^{\alpha}}{1+z\theta_0\exp\left(-\frac{\bv^2}{2T}\right)}\rd \bv.
\end{equation}
When $|z \theta_0| < 1$ and $\theta_0 \neq 0$, it follows that 
\begin{equation}
    \label{eq:mom_M_3}
    \frac{1}{1+z\theta_0\exp\left(-\frac{\bv^2}{2T}\right)} = \sum_{n=0}^{+\infty}\left[-z\theta_0\exp\left(-\frac{\bv^2}{2T}\right)\right]^n.
\end{equation}
By substituting \eqref{eq:mom_M_3} into \eqref{eq:mom_M_2}, the expression of $\mM_{\alpha}$ becomes 
\begin{equation}
    \label{eq:mom_M_4}
    \mM_{\alpha} = -\frac{\overline{\Gamma}(\frac{\alpha+1}{2})}{\theta_0}(2T)^S \sum_{n = 1}^{+\infty} \frac{(-z\theta_0)^n}{n^{S}}, \qquad S = \frac{|\alpha|+D}{2},
\end{equation}
where $\overline{\Gamma}\left(\frac{\alpha+1}{2}\right) = \prod_{l = 1}^{D}\Gamma(\frac{\alpha_l+1}{2})$ and $\Gamma(\cdot)$ denotes the Gamma function. To further simplify \eqref{eq:mom_M_4}, we introduce the polylogarithm function as follows \cite{Roughan2020}:
\begin{definition}[The polylogarithm function]
The polylogarithm function $\Li_s(y)$ is defined by a power series of $y$, which is also a Dirichlet series of $s$:
\begin{equation}
    \label{eq:log}
    \Li_s(y)=\sum_{k=1}^{\infty}\frac{y^k}{k^s},\quad |y|<1.
\end{equation}
\eqref{eq:log} is valid for arbitrary complex order $s$ and for all complex variables $y$ with $|y|< 1$. It can be extended to $|y| \geqslant 1$ through analytic continuation. 
\end{definition}

Substituting \eqref{eq:log} into \eqref{eq:mom_M_4}, we have 
\begin{equation}
    \label{eq:mom_M}
    \mM_{\alpha} = -\frac{{\overline{\Gamma}\left(\frac{\alpha+1}{2}\right)}}{\theta_0}(2T)^{\frac{|\alpha|+D}2}\Li_{\frac{|\alpha|+D}2}(-z\theta_0),\qquad z\theta_0 \in[-1, 0) \cup (0, +\infty).
\end{equation}

\begin{remark}
When $\theta_0=0$, it is reduced into the classical case, and one can derive that
\begin{equation}
    \lim_{\theta_0\rightarrow 0} \frac{\Li_{s}(-z\theta_0)}{\theta_0} = -z, \quad \forall s>0. 
\end{equation}
In this case, \eqref{eq:mom_M} still holds, and $\mM_{\alpha}$ is calculated as 
\begin{equation}
    \label{eq:c_M_alpha}
    \mM_{\alpha} = \overline{\Gamma}\left(\frac{\alpha+1}{2}\right)(2T)^{\frac{|\alpha|+D}2}z, \qquad z = \frac{\rho}{(2\pi T)^{\frac{D}{2}}}.
\end{equation}
\end{remark}

From Alg. \ref{algo:algBGK}, it can be observed that $\mM_{\alpha}$ needs to be computed at each spatial position in each time step. Thus, an efficient algorithm is required to calculate the polylogarithm $\Li_s(y)$.

\paragraph{Calculations of the polylogarithm}
\label{sec:integral}
Several algorithms have been proposed to evaluate the polylogarithm $\Li_s(y)$. While the function \textbf{polylog} in MATLAB can be used for this purpose, the low efficiency restricts its applications in large-scale numerical simulations. Some efforts have been made to numerically compute polylogarithms for integer $s$ \cite{Vollinga2005, Duhr2019}, but they are inadequate for simulations involving quantum kinetic problems.

In fact, for the Bose and Fermi gas, based on \eqref{eq:Maxwellian} and \eqref{eq:mom_M_4}, the domain for $s$ and $y$ is given by
\begin{equation}
    \label{eq:region_s_y} 
    s = \frac{2n +D}{2}, \qquad n \in \bbN, \qquad y \in (-\infty, 1],
\end{equation}
and a method to compute $\Li_s(y)$ in this region would meet our demands. Inspired by the derivation of \eqref{eq:mom_M}, we transform the polylogarithm into an one-dimensional integral, expressed as
\begin{equation}
    \label{eq:rela-Li}
    \begin{split}
        \int_{\bbR}\frac{|x|^n}{\exp\left({\frac{x^2}{2}}\right)-y} \rd x&=\frac1y\int_{\bbR}\sum_{k=1}^{\infty}\left(y\exp\left(-\frac{x^2}{2}\right)\right)^k|x|^n\rd x \\
        & =2^{\frac{n+1}2}\Gamma\left(\frac{n+1}2\right)\frac{\Li_{\frac{n+1}2}(y)}{y}, \qquad y \in (-1, 0)\cup(0, 1). 
    \end{split}
\end{equation}
By employing the analytical continuation with respect to $y$, the polylogarithm can be computed as
\begin{equation}
    \label{eq:algo-Li}
    \Li_{s}(y)=\frac{2^{-s}y}{\Gamma\left(s\right)} \int_{\bbR}\frac{x^{2n+2}}{\exp\left({\frac{x^2}{2}}\right)-y} \rd x, \qquad  s = \frac{2n+D}{2}, \quad n \in \bbN,   \quad 
    y \in(-\infty, 1],
\end{equation}
and the integral on the right-hand side can be approximated using a Gauss-type quadrature. The rescaled Gauss-Hermite quadrature
\begin{equation}
    \label{eq:rescale-GH}
    \int_{\bbR}g(x)\exp\left(-\frac{x^2}{\beta}\right)\rd x=\sqrt{\beta}\int_{\bbR}g(\sqrt{\beta}x)\exp\left(-{x^2}\right)\rd x\approx \sqrt{\beta}\sum_{k=1}^{N_{\rm int}} \omega_kf(\sqrt{\beta}x_k)
\end{equation}
is adopted here to evaluate this integral. Here, $N_{\rm int}$ represents the number of integral points, $x_k$ are the roots for the Hermite polynomial of degree $(N_{\rm int}+1)$, and $\omega_k$ are the integral weights. For further details on the Gauss-Hermite quadrature, readers may refer to \cite{Glaser2007}.

To enhance the efficiency of this integral approximation, we treat the scaling factor $\beta$ as a function of $y$. The integral in \eqref{eq:algo-Li} is then approximated by
\begin{equation}
    \label{eq:GH-Li}
    \int_{\bbR}\frac{1}{\exp\left({\frac{x^2}{2}}\right)-y}|x|^{2n+2}\rd x=\int_{\bbR}\exp\left(-\frac{x^2}{\beta(y)}\right)G(x,y)\rd x=\sqrt{\beta(y)}\sum_{k=1}^{N_{\rm int}} \omega_k G(\sqrt{\beta(y)}x_k,y),
\end{equation}
where 
\begin{equation}
    \label{eq:def-G}
    G(x,y)=\frac{\exp\left(\frac{x^2}{\beta(y)}\right)}{\exp\left({\frac{x^2}{2}}\right)-y}|x|^{2n+2}, \qquad n \in \bbN, \qquad  y \in(-\infty,0) \cup(0, 1].
\end{equation}

\begin{remark}
    The integral most commonly used to analyze $\Li_s(y)$ has the form
    \begin{equation}
        \label{eq:polylog_GL}
         \Li_{s}(y)=\frac{y}{\Gamma(s)}\int_{0}^{\infty}\frac{x^{s-1}}{\exp\left(x\right)-y} \rd x,  \quad y \in(-\infty, 1].    
    \end{equation}
For the common case $D=3$, the numerator is a function of $x$ with a half-integer index. However, the Gauss-type quadrature will be more accurate when the integrand behaves closely to a polynomial \cite{Burden2011numerical}. Therefore, we opt for the integral \eqref{eq:algo-Li} over \eqref{eq:polylog_GL} to compute the polylogarithm. Additionally, the choice of the scaling factor $\beta(y)$ in \eqref{eq:def-G} aims to make the integrand $G(x, y)$ more polynomial-like.
\end{remark}

In the numerical experiments, $\beta(y)$ is empirically chosen as 
\begin{equation}
    \label{eq:beta_y}
    \beta(y)=\left\{\begin{array}{ll}
        2-1.8y, & y\in(0, 1], \\
        1+\exp(y), &  y\in(-\infty, 0).
    \end{array}
    \right.
\end{equation}
It is always challenging to accurately calculate $x_k$ and $\omega_k$ when $N_{\rm int}$ is too large, so we set it as $N_{\rm int}=70$ in all subsequent tests. In App. \ref{app:compare}, this method is compared with the MATLAB algorithm \textbf{polylog}, which verifies its excellent efficiency and accuracy. This algorithm will also be employed to solve the nonlinear system in Sec. \ref{sec:solve}.

\subsection{Solving the nonlinear system}
\label{sec:solve}
In this section, we present the method for obtaining $z$ and $T$ in the quantum Maxwellian $\Mq$. Since $\rho$ and $e_0$ can be easily derived from the distribution function $f$, $z$ and $T$ can be solved from the nonlinear system \eqref{eq:z_T}. 

With the relationship \eqref{eq:mom_M}, the nonlinear system \eqref{eq:z_T} can be simplified as 
\begin{equation}
    \label{eq:newsystem}
    \begin{split}
        \rho&=-\frac{1}{\theta_0}(2\pi)^{\frac{D}{2}}T^{\frac{D}{2}}\Li_{\frac{D}{2}}(-z\theta_0), \\
        \rho e_0&=-\frac{D}{2\theta_0}(2\pi)^{\frac{D}{2}}T^{\frac{2+D}{2}}\Li_{\frac{2+D}{2}}(-z\theta_0).
    \end{split}
\end{equation}
Without loss of generality, we set $D = 3$ in this section, and the algorithm can be easily extended to other $D$. By eliminating $T$ in \eqref{eq:newsystem}, the system is reduced to a nonlinear equation of $z$ as 
\begin{equation}
    \label{eq:equation-z}
    \frac{\left|\Li_{\frac32}(-z\theta_0)\right|^{\frac52}}{\left|\Li_{\frac52}(-z\theta_0)\right|^{\frac32}}=|\theta_0|\rho\left(\frac{3}{4\pi e_0}\right)^{\frac{3}2}.
\end{equation}
We will first discuss the existence of the solution for the Bose-Einstein and Fermi-Dirac gas separately before introducing the algorithm to solve \eqref{eq:equation-z}.
\paragraph{Bose-Einstein gas ($\theta_0<0$)}
As stated in \eqref{eq:theta_BE}, $-z \theta_0$ is restricted in $(0, 1]$ to ensure $\Mq$ is positive. Define 
\begin{equation}
    \label{eq:def-B}
    \mB(y)=
        \frac{\big(\Li_{\frac32}(y)\big)^{\frac52}}{\big(\Li_{\frac52}(y)\big)^{\frac32}}, \qquad  0<y\leqslant 1.
\end{equation}
Then \eqref{eq:equation-z} is reduced into 
\begin{equation}
    \label{eq:BEeq}
    \mB(y)=|\theta_0|\rho\left(\frac{3}{4\pi e_0}\right)^{\frac32},\qquad  y = -z \theta_0, \qquad 0<y\leqslant 1.
\end{equation}
As observed in \cite{Mouton2023}, $\mB(y)$ is continuous and non-decreasing on $(0,1]$. When $$|\theta_0|\rho\left(\frac{3}{4\pi e_0}\right)^{\frac32}\geqslant \mB(1),$$ the Bose-Einstein condensation occurs. In this case, $z \theta_0 = -1$, and the quantum Maxwellian is reduced into \eqref{eq:condensation} with the related parameters derived in \eqref{eq:BE_deg_T}.
Otherwise, a solution for $y$ exists in $(0,1]$.

\paragraph{Fermi-Dirac gas ($\theta_0>0$)} Unlike the Bose-Einstein gas, $z$ can be arbitrary large in the Fermi-Dirac distribution. To distinguish from the Bose-Einstein gas, we define 
\begin{equation}
    \label{eq:def-F}
    \mF(y) = \frac{\big(-\Li_{\frac32}(y)\big)^{\frac52}}{\big(-\Li_{\frac52}(y)\big)^{\frac32}}, \qquad y<0,
\end{equation}
and then \eqref{eq:equation-z} is reduced into 
\begin{equation}
    \label{eq:FDeq}
    \mF(y)=|\theta_0|\rho\left(\frac{3}{4\pi e_0}\right)^{\frac32}, \qquad y = -z \theta_0, \qquad  y\in (-\infty, 0). 
\end{equation}
As observed in \cite{Mouton2023}, $\mF(y)$ is continuous and non-increasing on $(-\infty, 0)$ and satisfies 
\begin{equation}
    \label{eq:limF}
    \lim_{y\to-\infty}\mF(y)=\frac53\sqrt{\frac{10}{\pi}}.
\end{equation}
The following lemma ensures that there exists a solution for $y$ in \eqref{eq:FDeq}.
\begin{lemma}
    \label{lemma:FD}
    Under the Pauli exclusion principle $f\leqslant \frac{1}{\theta_0}$, it holds that 
    \begin{equation}
    \label{eq:lemma-FD}
        \theta_0\rho\left(\frac{3}{4\pi e_0}\right)^{\frac32}\leqslant\frac53\sqrt{\frac{10}{\pi}}.
    \end{equation}
\end{lemma}
\begin{proof}[Proof of Lemma \ref{lemma:FD}]
    Without loss of generality, we assume the macroscopic velocity $\bu=\bz$. Let $R > 0$ satisfy 
    \begin{equation}
        \label{eq:R}
        \frac{4\pi}3 R^3 = \rho \theta_0,
    \end{equation}
    and define the auxiliary function $f_0$ as 
    \begin{equation}
        \label{eq:FD-sat}
        f_0(\bv)=\frac{1}{\theta_0}\chi_{B(\bz,R)}=\left\{
        \begin{array}{ll}
            \frac{1}{\theta_0}, & |\bv|\leqslant R,  \\
            0, & |\bv|>R,
        \end{array}
        \right.
    \end{equation}
which is also known as the Fermi-Dirac saturation, and represents the critical state of the Fermi gas. It follows for $f_0$ that 
    \begin{equation}
        \label{eq:mom-f0}
            \int_{\bbR^3}f_0(\bv)\rd\bv= \frac{4\pi}{3\theta_0} R^3=\rho, \qquad 
            \frac12\int_{\bbR^3}f_0(\bv)|\bv|^2\rd\bv=\frac{2\pi}{5 \theta_0} R^5.
    \end{equation}
Using the relation of the macroscopic variables and the distribution function $f$ as in \eqref{eq:macro}, we can derive that 
    \begin{equation}
        \label{eq:compare-e0}
        \begin{split}
           \rho e_0 - \frac{2\pi}{5 \theta_0} R^5 &=   \rho e_0-\frac12\int_{\bbR^3}f_0(\bv)|\bv|^2\rd\bv 
             = \frac12\int_{\bbR^3}\big(f(\bv)-f_0(\bv)\big)|\bv|^2\rd\bv \geqslant 0.
        \end{split}
    \end{equation}
    The proof is completed by combining \eqref{eq:compare-e0} and \eqref{eq:R}. 
\end{proof}

The derivative of the polylogarithm function can be expressed as
\begin{equation}
    \label{eq:deri_poly}
    \frac{\rd}{\rd y}\Li_s(y)=\frac{1}{y}\Li_{s-1}(y).
\end{equation}
Thus, if a solution exists in \eqref{eq:equation-z}, it can be numerically obtained through the Newton iteration method, with the stopping criterion set as
\begin{equation}
    \label{eq:stop_criterion}
    \left| \frac{\big|\Li_{\frac32}(y)\big|^{\frac52}}{\big|\Li_{\frac52}(y)\big|^{\frac32}}-|\theta_0|\rho\left(\frac{3}{4\pi e_0}\right)^{\frac32}\right|\leqslant 10^{-12},
\end{equation}
and the numerical solution at the last time step is utilized as the initial solution for the iteration. 

In practical numerical simulations, the iteration count remains quite low, typically less than $5$ for most cases. To illustrate the efficiency of this Newton iteration in detail, an example is presented in App. \ref{app:Newton}. Once $z$ is obtained, the temperature $T$ can be directly solved from \eqref{eq:newsystem}, concluding the algorithm for this nonlinear system.



\section{Numerical experiments}
\label{sec:experiment}
In this section, several numerical examples are conducted to validate this numerical scheme for the quantum BGK equation. First, the asymptotic-preserving (AP) property of this Hermite spectral method is tested with a periodic flow. Subsequently, we examine its performance in the spatially one-dimensional cases, including the Sod and mixing regime problems. Finally, a spatially two-dimensional lid-driven cavity flow problem is simulated to further validate the accuracy and efficiency of this Hermite spectral method. 

\subsection{Test of the AP property}
\label{sec:exp_AP}
\begin{figure}[!ht]
    \centering
    \subfloat[density $\rho$, $\theta_0 = 9$]
    {\includegraphics[width=0.33\textwidth,
      clip]{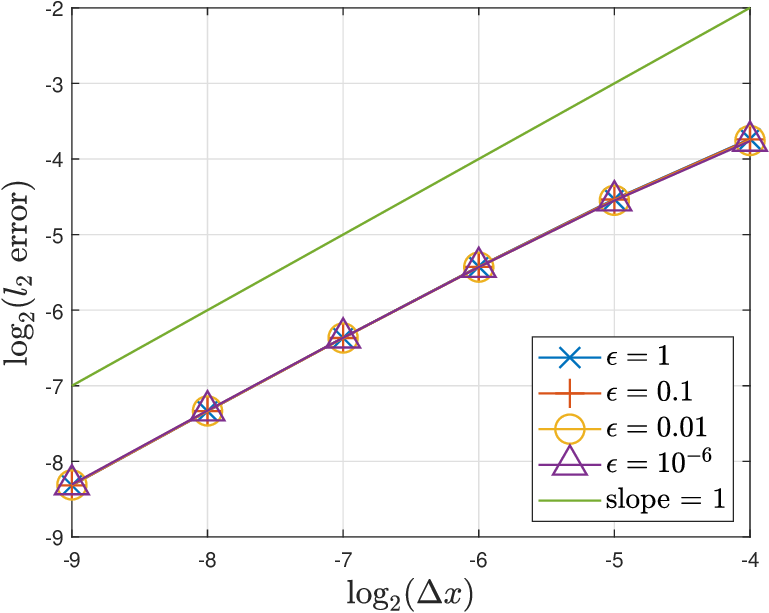}}\hfill
      \subfloat[temperature $T$, $\theta_0 = 9$]
    {\includegraphics[width=0.33\textwidth,
      clip]{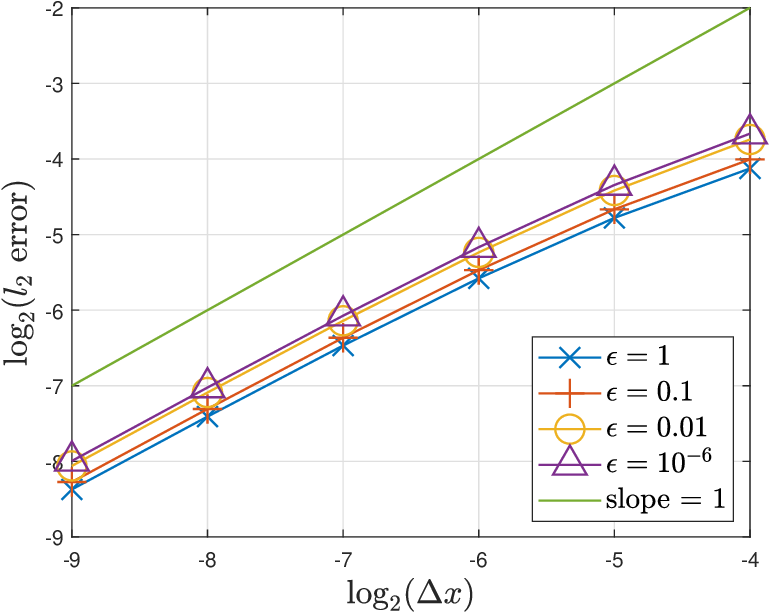}}\hfill
 \subfloat[fugacity $z|\theta_0|$, $\theta_0=9$]
    {\includegraphics[width=0.33\textwidth,
      clip]{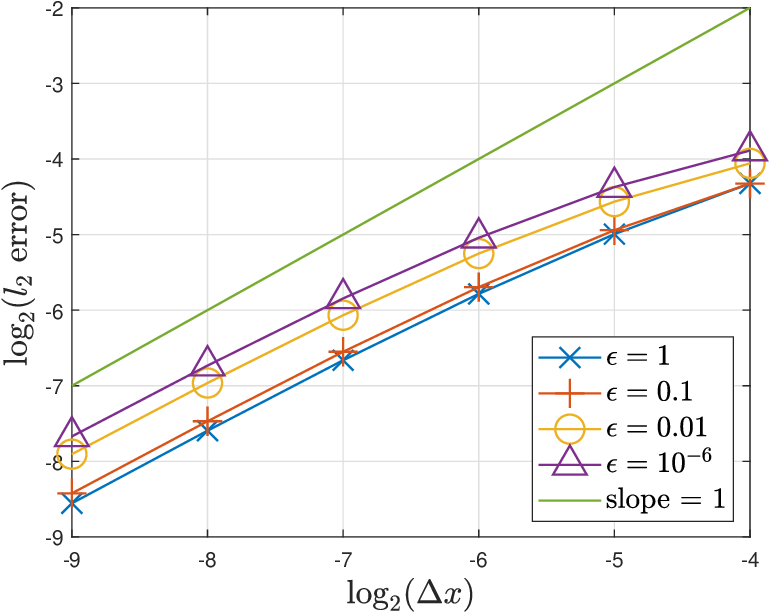}}\\
    \subfloat[density $\rho$, $\theta_0 = -9$]
    {\includegraphics[width=0.33\textwidth,
      clip]{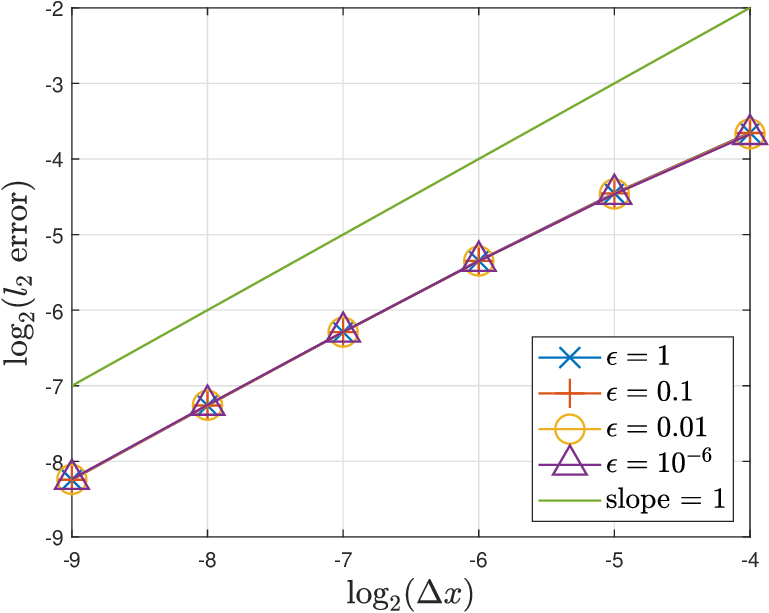}}\hfill
      \subfloat[temperature $T$, $\theta_0=-9$]
    {\includegraphics[width=0.33\textwidth,
      clip]{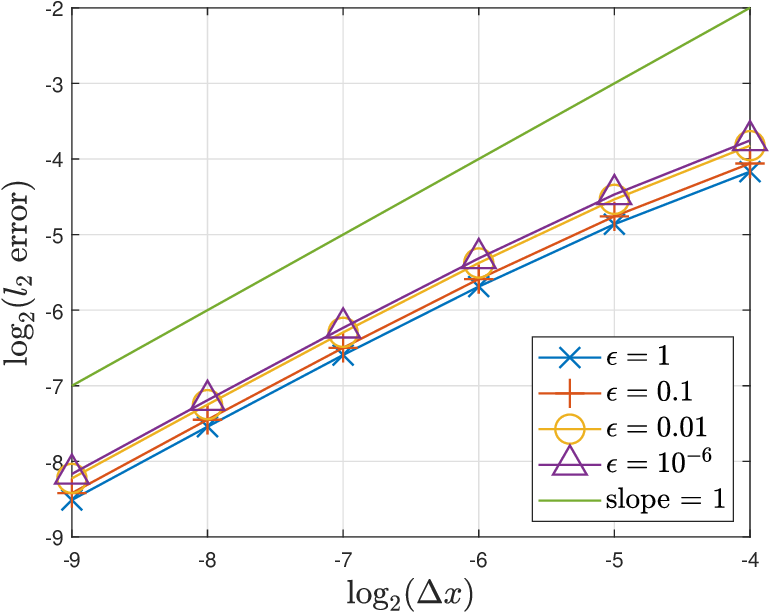}}\hfill
     \subfloat[fugacity $z|\theta_0|$, $\theta_0=-9$]
    {\includegraphics[width=0.33\textwidth,
      clip]{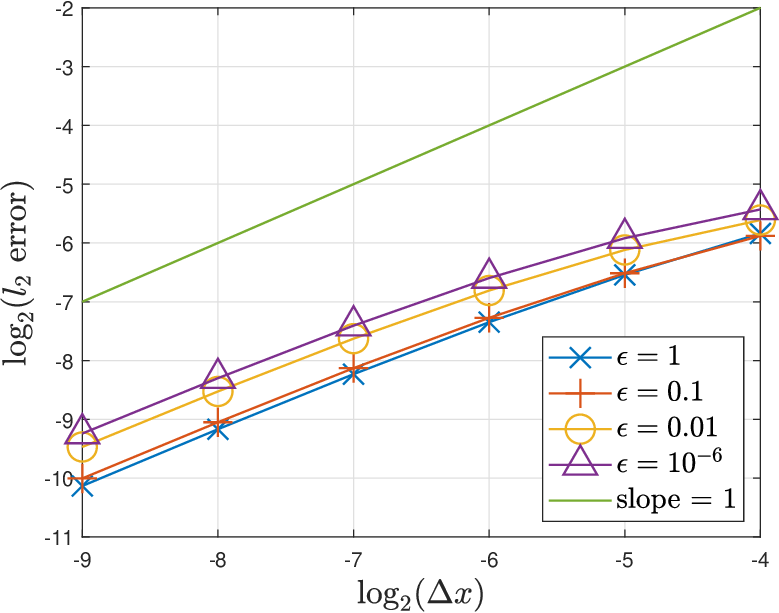}}      
    \caption{(Test of the AP property in Sec. \ref{sec:exp_AP}) The $l_2$ error of the density $\rho$, temperature $T$, and fugacity $z |\theta_0|$ at $t=0.1$ between the numerical solutions with the first-order scheme \eqref{eq:AP_time} and the reference solutions. The first row presents the Fermi gas with $\theta_0 = 9$, and the second row presents the Bose gas with $\theta_0 = -9$. }
    \label{fig:AP1}
\end{figure}

\begin{figure}[!ht]
    \centering
    \subfloat[density $\rho$, $\theta_0 = 9$]
    {\includegraphics[width=0.33\textwidth,
      clip]{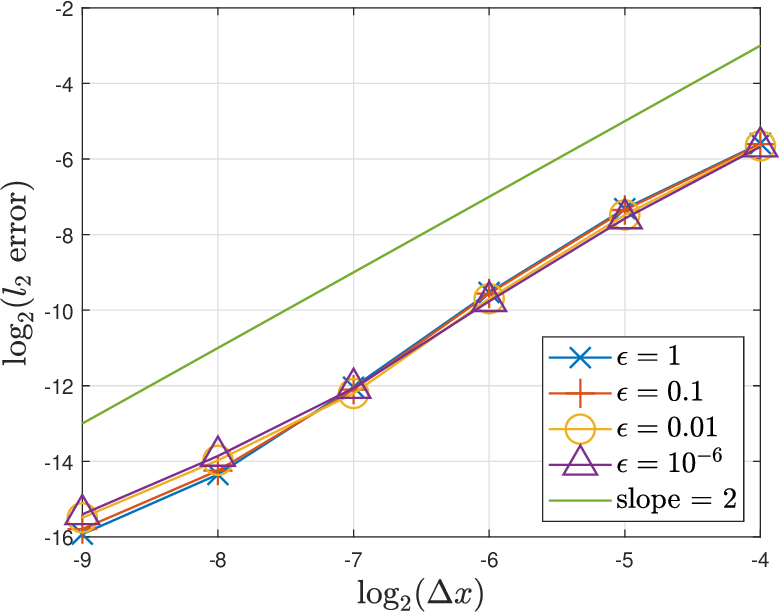}}\hfill
      \subfloat[temperature $T$, $\theta_0 = 9$]
    {\includegraphics[width=0.33\textwidth,
      clip]{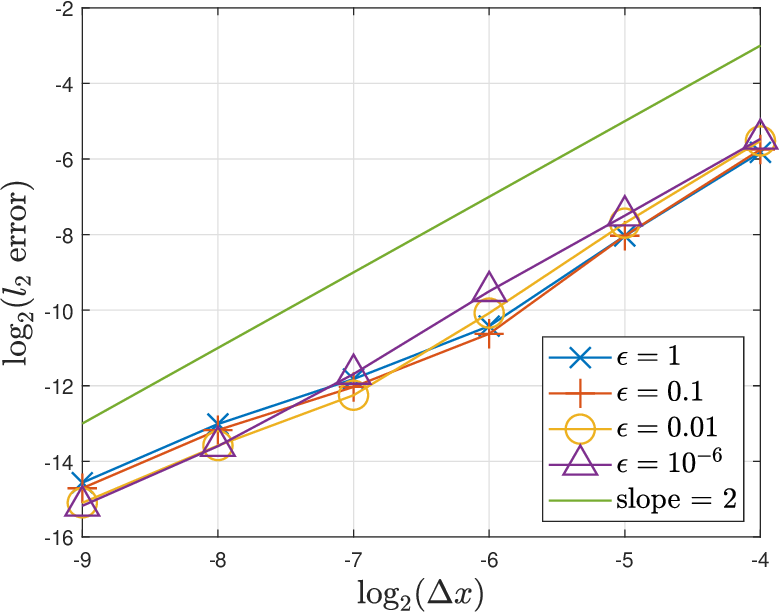}}\hfill
   \subfloat[fugacity $z|\theta_0|$, $\theta_0=9$]
    {\includegraphics[width=0.33\textwidth,
      clip]{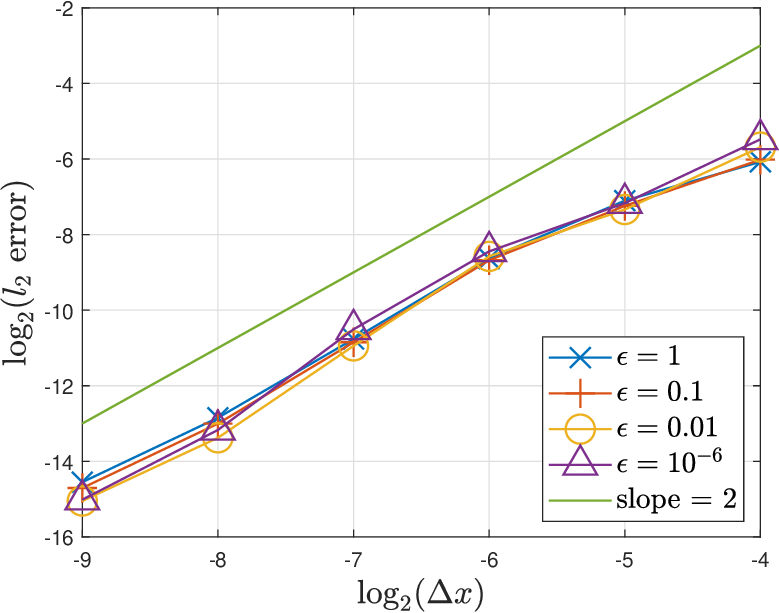}}\\
    \subfloat[density $\rho$, $\theta_0 = -9$]
    {\includegraphics[width=0.33\textwidth,
      clip]{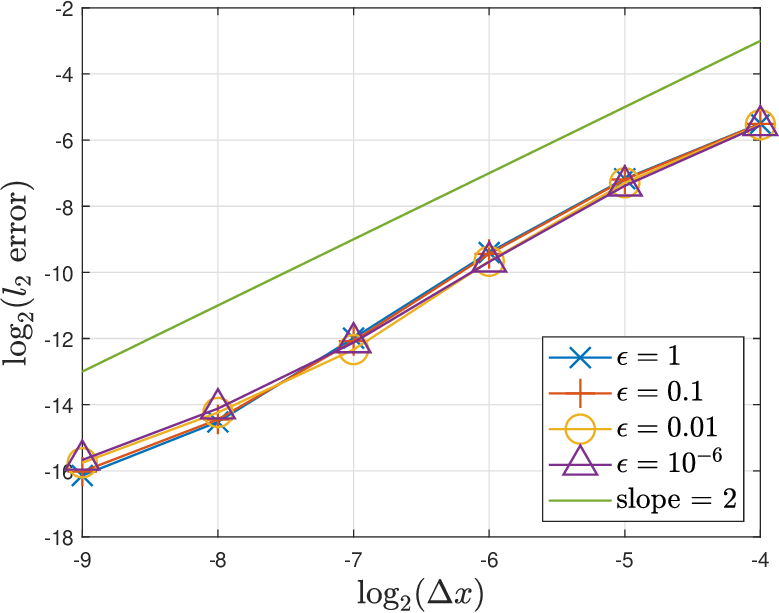}}\hfill
      \subfloat[temperature $T$, $\theta_0=-9$]
    {\includegraphics[width=0.33\textwidth,
      clip]{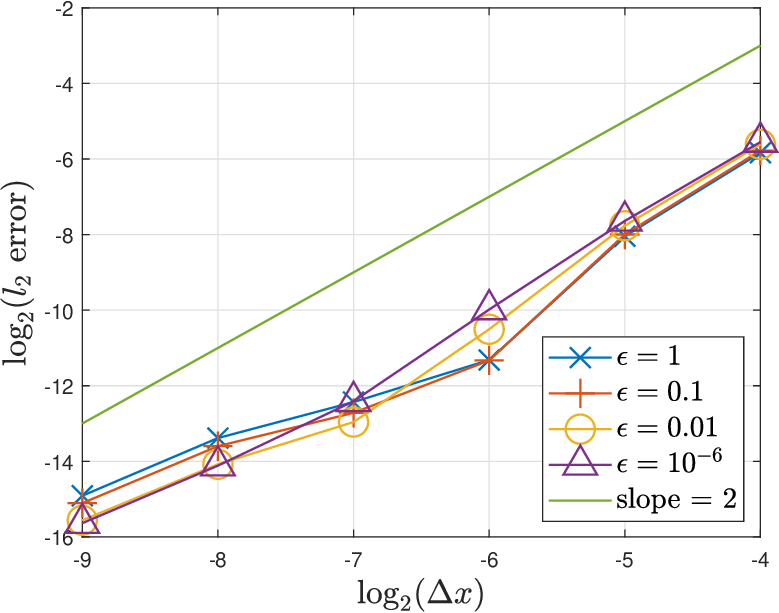}}\hfill
     \subfloat[fugacity $z|\theta_0|$, $\theta_0=-9$]
    {\includegraphics[width=0.33\textwidth,
      clip]{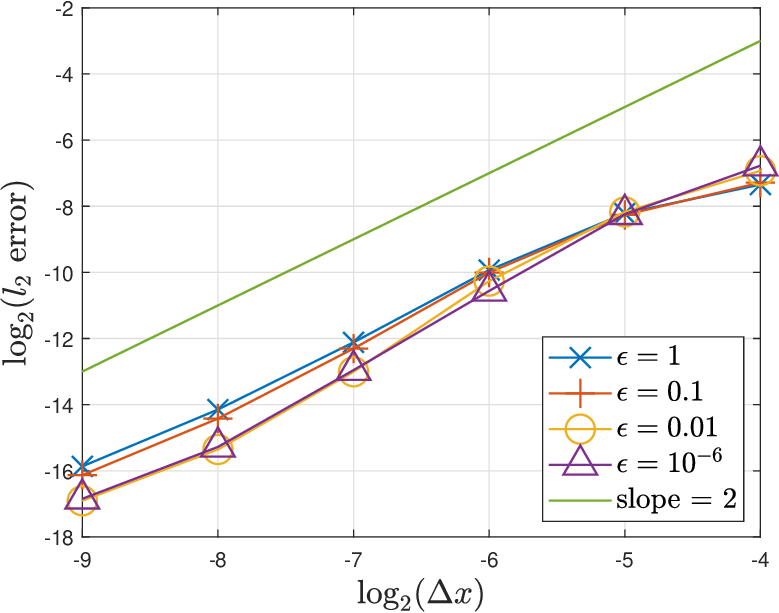}}      
    \caption{(Test of the AP property in  Sec. \ref{sec:exp_AP}) The $l_2$ error of the density $\rho$, temperature $T$, and fugacity $z |\theta_0|$ at $t=0.1$ between the numerical solutions with the IMEX2 scheme \eqref{eq:second} and the reference solutions. The first row presents the Fermi gas with $\theta_0 = 9$, and the second row presents the Bose gas with $\theta_0 = -9$.}
    \label{fig:AP2}
\end{figure}

In this section, the AP property of this Hermite spectral method is assessed. The spatial and microscopic velocity dimensions are set to $1$ and $3$, respectively. The initial condition is the equilibrium with macroscopic variables given by
\begin{equation}
    \label{eq:Ini-AP}
        \rho=\frac{2+\sin(2\pi x)}{3},  \qquad 
        \bu = 0, \qquad
        T=\frac{3+\sin(2\pi x)}{4}.
\end{equation}
The computational domain is $[0, 1]$, with periodic boundary conditions applied in the spatial space. The expansion order is set as $M = 10$, and the CFL number is set as ${\rm CFL} = 0.2$. We perform simulations with grid numbers $N = 16, 32, 64, 128, 256$ and $512$, respectively for Knudsen numbers $\epsilon = 1, 0.1, 0.01$ and $10^{-6}$, and the parameter $\theta_0 = \pm 9$.

The first-order IMEX scheme \eqref{eq:AP_time} is initially tested, with the expansion center set as the spatial average of the macroscopic velocity and temperature, i.e., $[\ou, \oT] = [0, \frac34]$, and no reconstruction is applied in the spatial space. Reference solutions are obtained by the second-order scheme \eqref{eq:second} with the WENO reconstruction, utilizing a grid size of $N = 1024$. The computation time is $t = 0.1$. The $l_2$ error of the density $\rho$, temperature $T$, and fugacity $z|\theta_0|$ between numerical solutions and reference solutions for $\theta_0 = \pm 9$ is presented in Fig. \ref{fig:AP1}. The results indicate that for both Bose and Fermi gas with different Knudsen numbers, the error uniformly converges with first-order, confirming the AP property of this first-order scheme \eqref{eq:AP_time}.

Next, we investigate the AP property of the IMEX2 scheme \eqref{eq:second}, utilizing the same computational parameters as the first-order scheme. The WENO reconstruction is utilized in the spatial space. The identical reference solutions as in the previous test are employed. The $l_2$ error of the density $\rho$, temperature $T$, and the fugacity $z|\theta_0|$ is displayed in Fig. \ref{fig:AP2}, illustrating that the error uniformly converges with second-order for both $\theta_0$ and different Knudsen numbers. This demonstrates the AP property of the IMEX2 scheme \eqref{eq:second}.

\subsection{Sod problem}
\label{sec:sod}
In this section, we examine the spatially 1D and microscopically 2D quantum Sod problem, which has also been investigated in previous studies \cite{Filbet2012, Hu2015}. The identical initial condition as in \cite[Sec. 5.1]{Hu2015} is used here: 
\begin{equation}
    \label{eq:Ini-Sod}
    \left\{
    \begin{array}{ll}
        \rho_l=1,\quad \bu_l=\bz, \quad T_l=1, & 0<x<0.5;  \\
        \rho_r=0.125,\quad \bu_r=\bz, \quad T_r=0.25, & 0.5\leqslant x<1. 
    \end{array}
    \right.
\end{equation}
Given the discontinuity in the initial condition, this problem poses a significant challenge in numerical computations.

First, we set the Knudsen number as $\epsilon = 0.01$, and $\theta_0=\pm 9, -0.01$, which are the same parameters as in \cite[Sec. 5.1]{Hu2015}. The computational region is $[0, 1]$. The IMEX2 scheme \eqref{eq:second} with the linear reconstruction is employed. To handle the discontinuity in the initial condition, the Minmod limiter is applied in the reconstruction. The CFL number in \eqref{eq:CFL} is set to ${\rm CFL} = 0.3$, and the mesh size in the spatial space is chosen as $N = 256$. The expansion order and center are selected as $M = 15$ and $[\ou, \oT]=[\bz, 1]$, respectively. Numerical solutions for the density $\rho$, the internal energy $e_0$ and the fugacity $z|\theta_0|$ at $t = 0.2$ by this Hermite spectral method (HSM) are provided in Fig. \ref{fig:Sod1}, where the reference solutions are obtained from solving the full quantum Boltzmann equation \eqref{eq:Boltz} with Exp-RK2 method in \cite[Fig. 3]{Hu2015}. It is evident that the numerical solutions closely match the reference solutions, indicating that the quantum BGK model \eqref{eq:BGK} serves as a good approximation of the full quantum collision operator \eqref{eq:operator}. Notably, this approximation allows for significantly reduced computational costs.

\begin{figure}[!ht]
    \centering
    \subfloat[density $\rho$, $\theta_0=-0.01$]
    {\includegraphics[width=0.33\textwidth, clip]{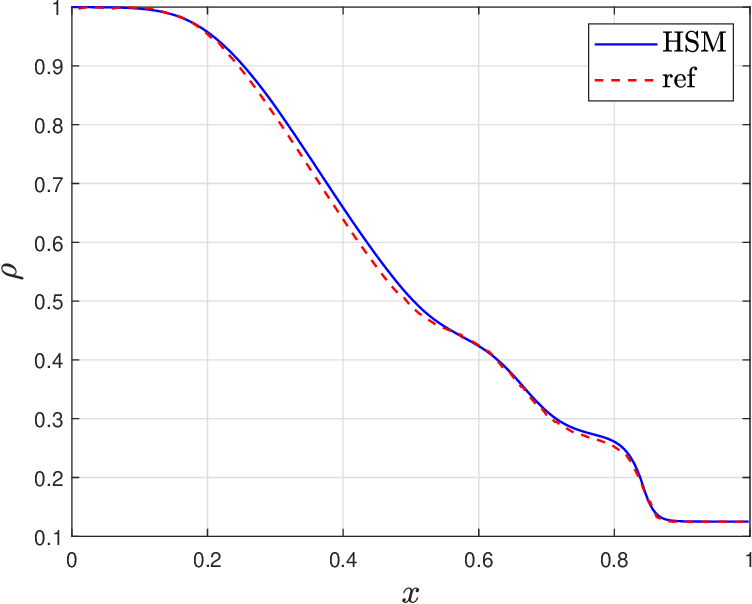}}\hfill
    \subfloat[density $\rho$, $\theta_0=9$]
    {\includegraphics[width=0.33\textwidth, clip]{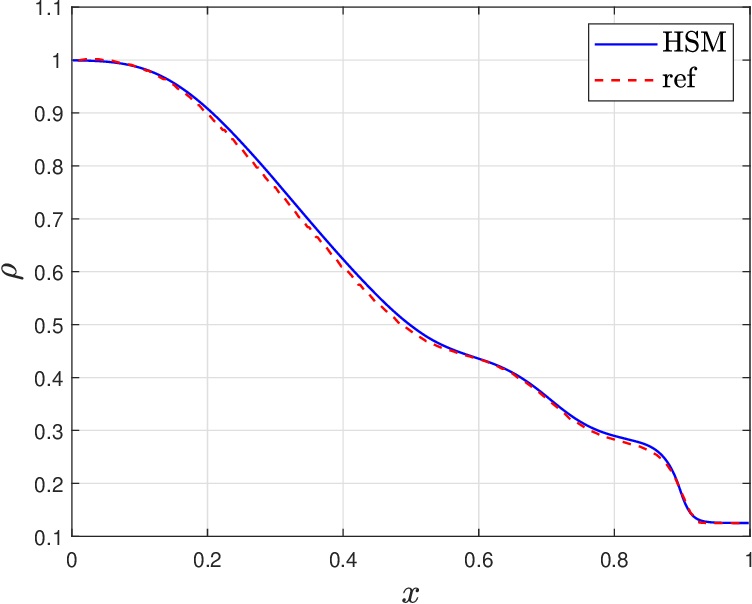}}\hfill
    \subfloat[density $\rho$, $\theta_0=-9$]
    {\includegraphics[width=0.33\textwidth, clip]{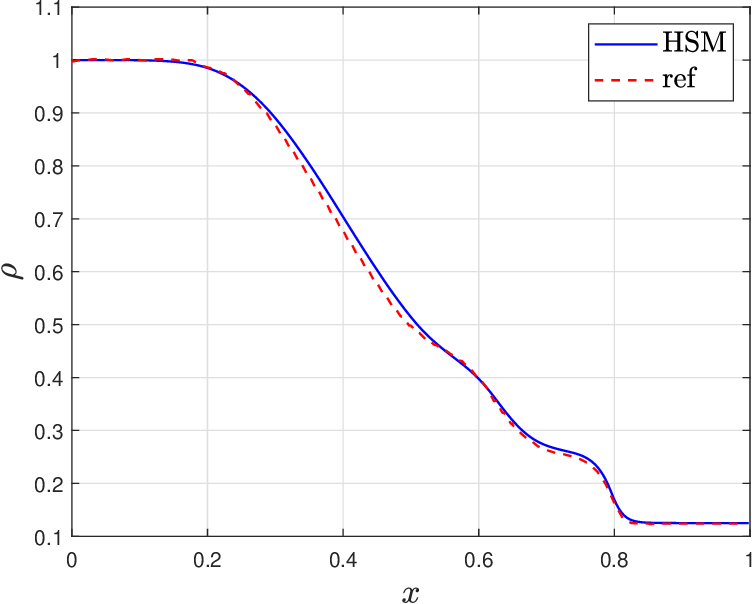}}\\
    \subfloat[internal energy $e_0$, $\theta_0=-0.01$]
    {\includegraphics[width=0.33\textwidth, clip]{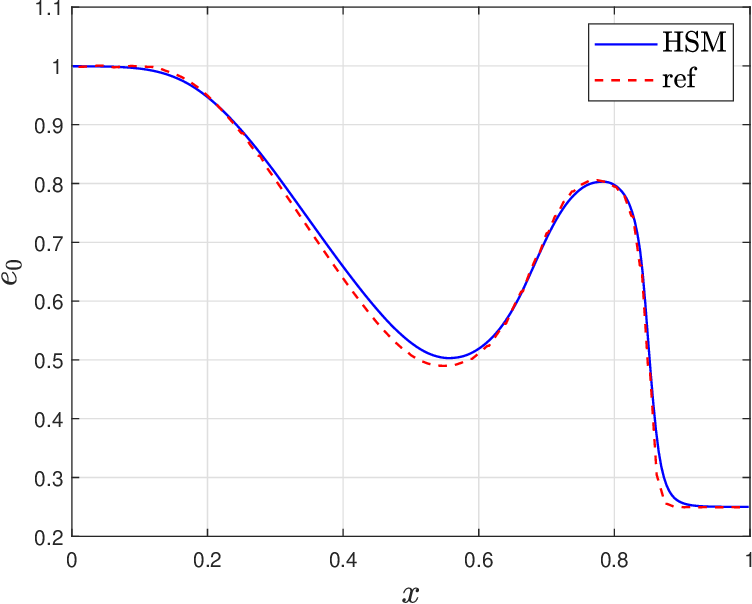}}\hfill
    \subfloat[internal energy $e_0$, $\theta_0=9$]
    {\includegraphics[width=0.33\textwidth, clip]{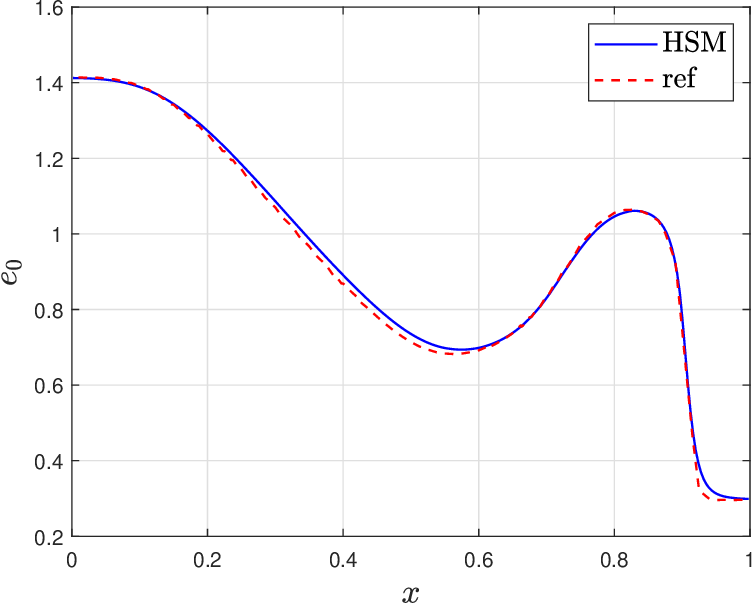}}\hfill
    \subfloat[internal energy $e_0$, $\theta_0=-9$]
    {\includegraphics[width=0.33\textwidth, clip]{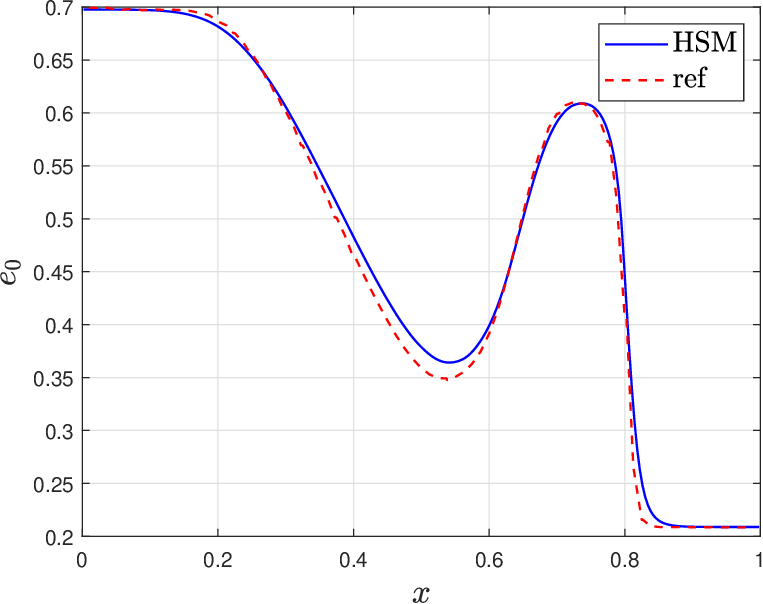}}\\
    \subfloat[fugacity $z|\theta_0|$, $\theta_0=-0.01$]
    {\includegraphics[width=0.33\textwidth, clip]{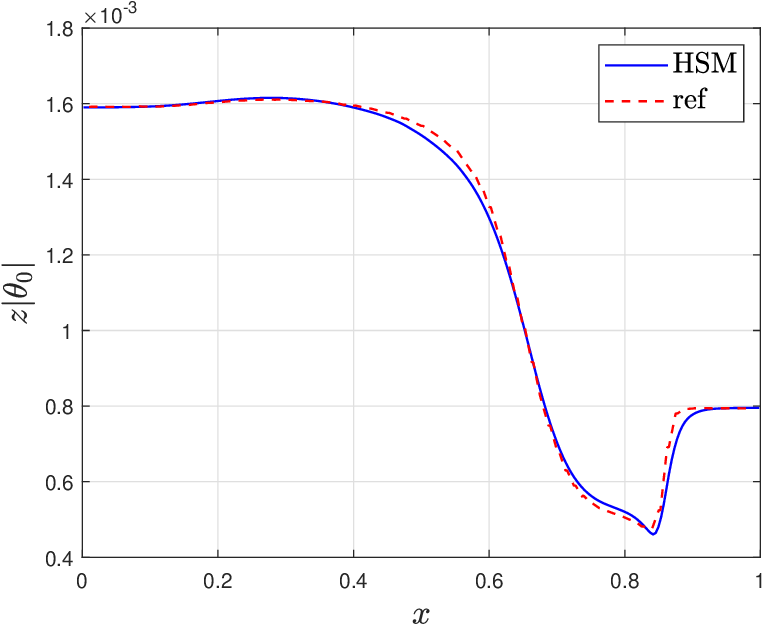}}\hfill
    \subfloat[fugacity $z|\theta_0|$, $\theta_0=9$]
    {\includegraphics[width=0.33\textwidth, clip]{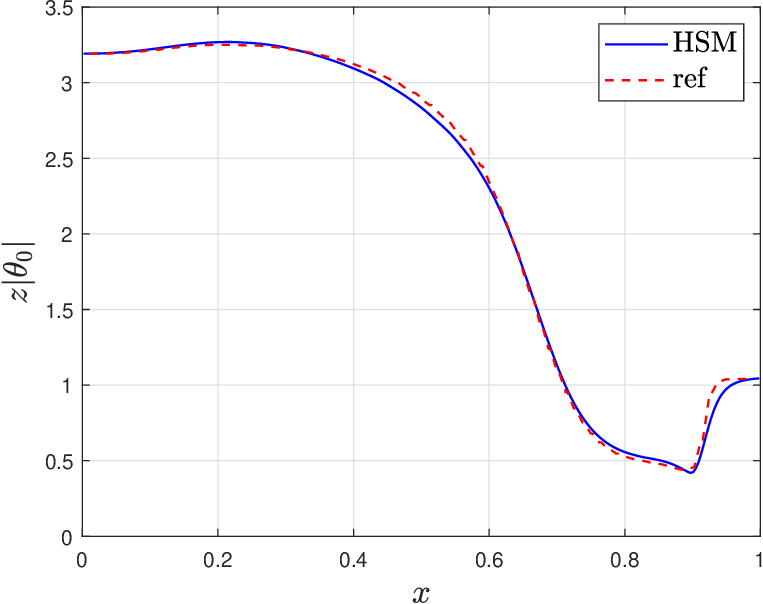}}\hfill
    \subfloat[fugacity $z|\theta_0|$, $\theta_0=-9$]
    {\includegraphics[width=0.33\textwidth, clip]{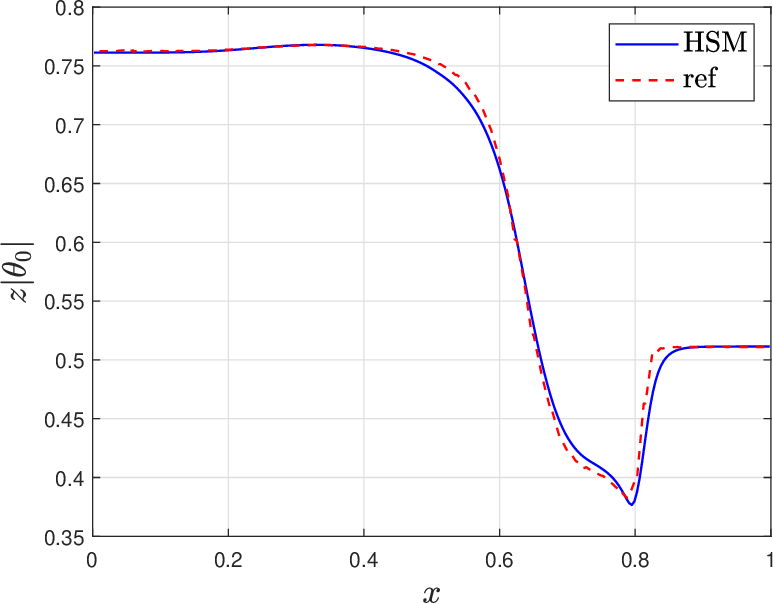}}\hfill
    \caption{(The Sod problem in Sec. \ref{sec:sod}) The numerical solutions of the density $\rho$, internal energy $e_0$, and fugacity $z|\theta_0|$ at $t = 0.2$ for $\epsilon = 0.01$. The blue lines represent the numerical solutions obtained using the Hermite spectral method (HSM), while the red dashed lines represent the reference solutions from \cite[Fig. 3]{Hu2015}.  The three columns present the near classical regime $\theta_0 = -0.01$, the Fermi gas $\theta_0 = 9$ in the quantum regime, and the Bose gas $\theta_0 = -9$ in the quantum regime, respectively.}    
    \label{fig:Sod1}
\end{figure}

To further validate the computational capability of this Hermite spectral method, we increase the Knudsen number to $\epsilon = 1$, which introduces greater challenges associated with rarefaction. The computational domain is extended to $[-0.5, 1.5]$. This simulation uses the same IMEX2 scheme \eqref{eq:second} and linear reconstruction as the previous test. We increase the expansion order to $M=30$, modify the CFL number to ${\rm CFL}=0.2$, and retain the expansion center at $[\ou, \oT] = [\bz, 1]$. The mesh size remains as $N = 256$. Numerical solutions of the density $\rho$, the internal energy $e_0$ and the fugacity $z|\theta_0|$ for $\theta = \pm 9$ and $-0.01$ at $t = 0.2$ are displayed in Fig. \ref{fig:Sod2}, along with the reference solutions obtained from the discrete velocity method (DVM). The excellent agreement between the numerical solutions and the reference solutions verifies the capability of this Hermite spectral method to accurately describe rarefied scenarios. 

\begin{figure}[!ht]
    \centering
    \subfloat[density $\rho$, $\theta_0=-0.01$]
    {\includegraphics[width=0.33\textwidth, clip]{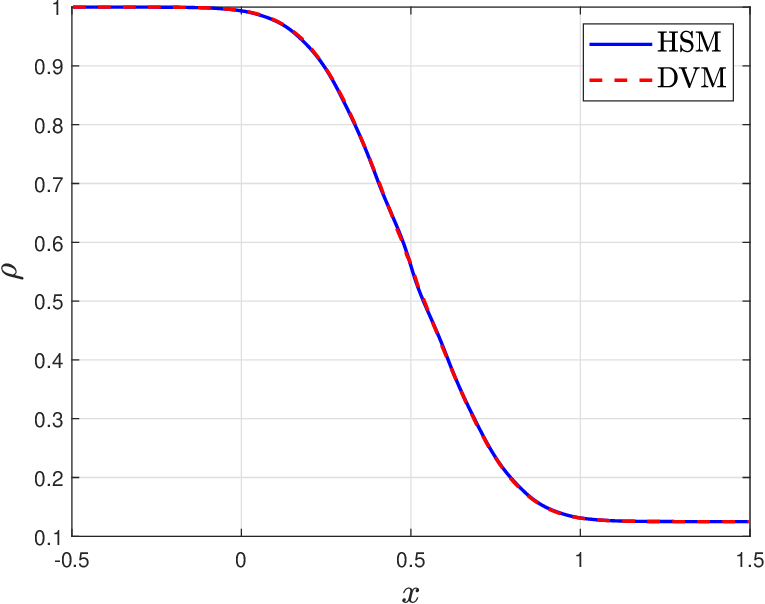}}\hfill
    \subfloat[density $\rho$, $\theta_0=9$]
    {\includegraphics[width=0.33\textwidth, clip]{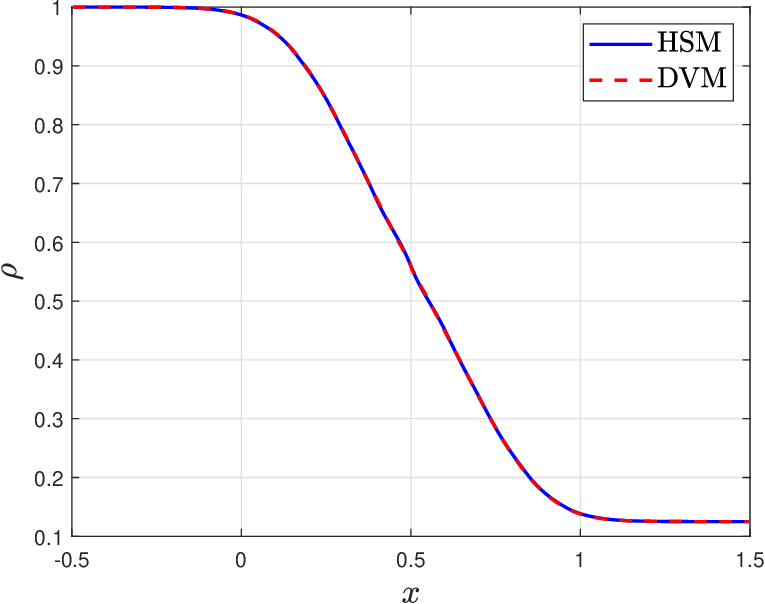}}\hfill
    \subfloat[density $\rho$, $\theta_0=-9$]
    {\includegraphics[width=0.33\textwidth, clip]{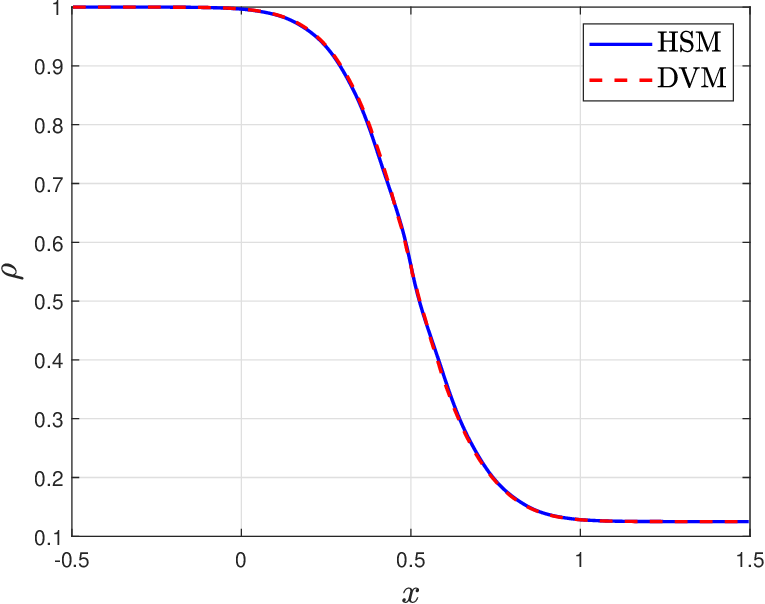}}\\
    \subfloat[internal energy $e_0$, $\theta_0=-0.01$]
    {\includegraphics[width=0.33\textwidth, clip]{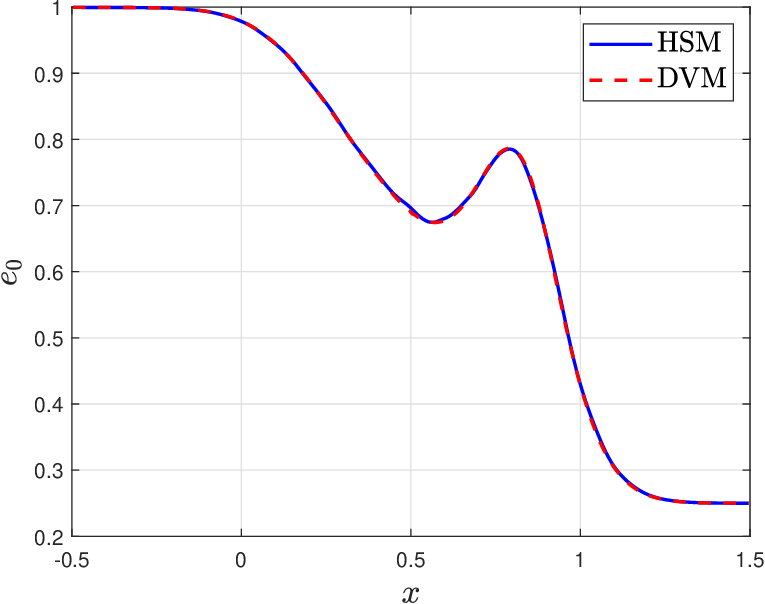}}\hfill
    \subfloat[internal energy $e_0$, $\theta_0=9$]
    {\includegraphics[width=0.33\textwidth, clip]{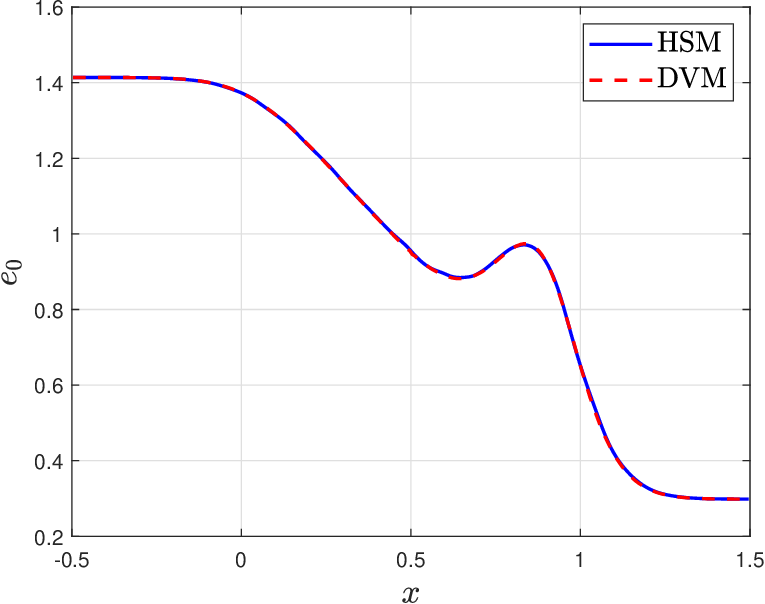}}\hfill
    \subfloat[internal energy $e_0$, $\theta_0=-9$]
    {\includegraphics[width=0.33\textwidth, clip]{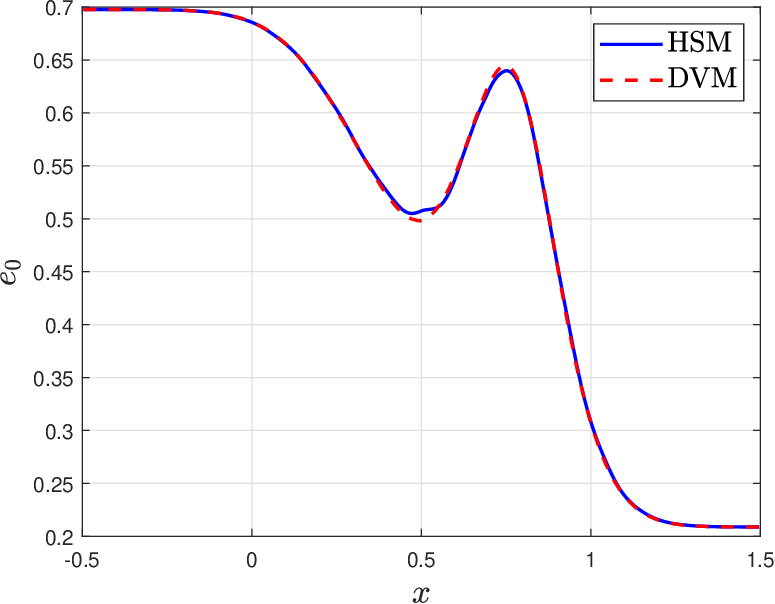}}\\
    \subfloat[fugacity $z|\theta_0|$, $\theta_0=-0.01$]
    {\includegraphics[width=0.33\textwidth, clip]{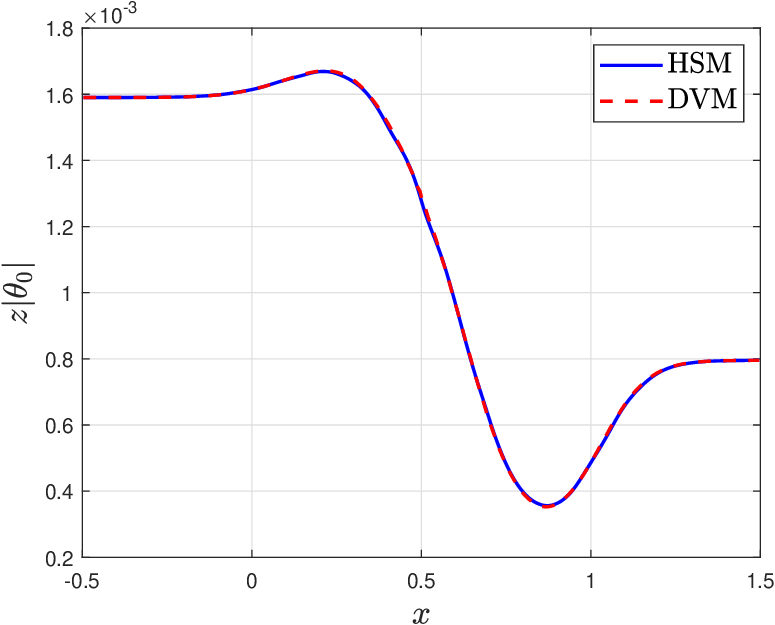}}\hfill
    \subfloat[fugacity $z|\theta_0|$, $\theta_0=9$]
    {\includegraphics[width=0.33\textwidth, clip]{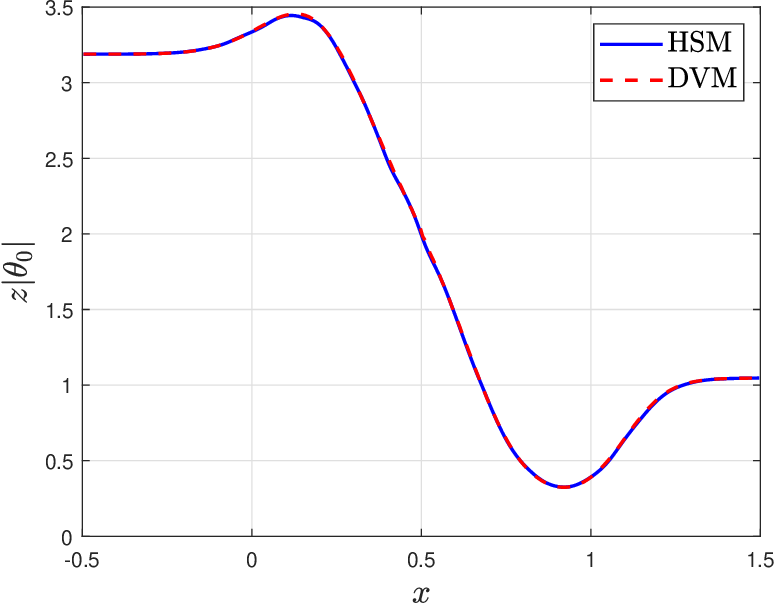}}\hfill
    \subfloat[fugacity $z|\theta_0|$, $\theta_0=-9$]
    {\includegraphics[width=0.33\textwidth, clip]{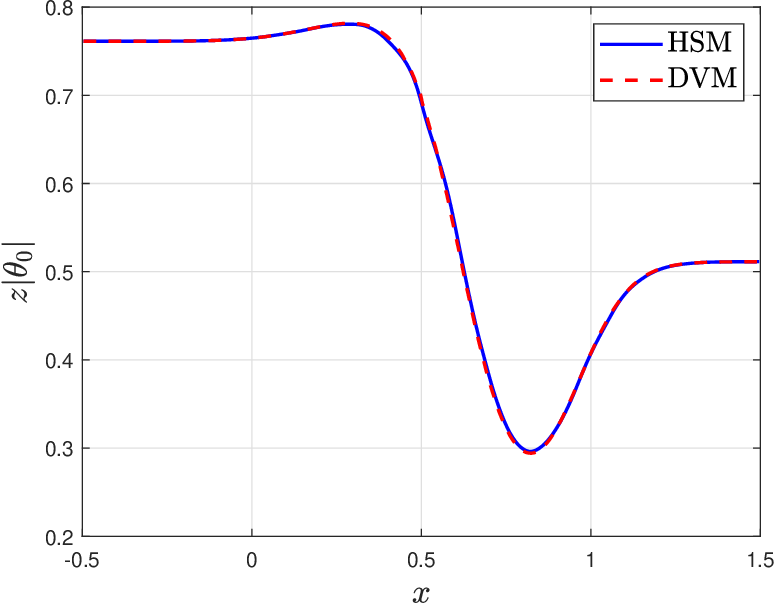}}\hfill
    \caption{(The Sod problem in Sec. \ref{sec:sod}) The numerical solutions of the density $\rho$, internal energy $e_0$, and fugacity $z|\theta_0|$ at $t = 0.2$ for $\epsilon = 1$. The blue lines represent the numerical solutions obtained using the Hermite spectral method (HSM), while the red dashed lines represent the reference solutions from the DVM. The three columns present the near classical regime $\theta_0 = -0.01$, the Fermi gas $\theta_0 = 9$ in the quantum regime, and the Bose gas $\theta_0 = -9$ in the quantum regime, respectively. }
    \label{fig:Sod2}
\end{figure}

\subsection{Mixing regime problem}
\label{sec:mix}
In this section, we address the spatially 1D and microscopically 3D problem with the mixing regime. Similar simulations have also been tested in \cite{Hu2015, Hu2012qFPL}. The initial condition 
\begin{equation}
    \label{eq:ini_Mix}
    \rho=1+\frac12 \sin(2\pi x),\qquad \bu=0, \qquad T=1+\frac14\sin(2\pi x)
\end{equation}
is utilized in this test. The Knudsen number, varying across kinetic and fluid regimes, is expressed as
\begin{equation}
    \label{eq:mix_Kn}
    \epsilon(x)=\epsilon_0+0.005(\exp(3x)-1), \qquad \epsilon_0 = 0.001,
\end{equation}
with the profile depicted in Fig. \ref{fig:Mix_eps}.

\begin{figure}[!ht]
    \centering
    \includegraphics[width=0.35\textwidth]{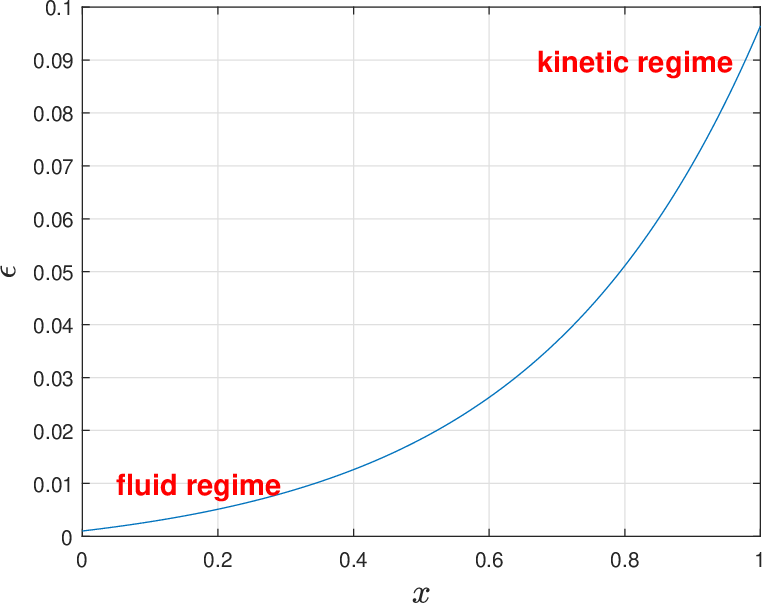}
    \caption{(Mixing regime problem in Sec. \ref{sec:mix}) The profile of the Knudsen number $\epsilon(x)$.}
    \label{fig:Mix_eps}
\end{figure}

\begin{figure}[!ht]
    \centering
    \subfloat[density $\rho$]
    {\includegraphics[width=0.33\textwidth, height=0.26\textwidth]{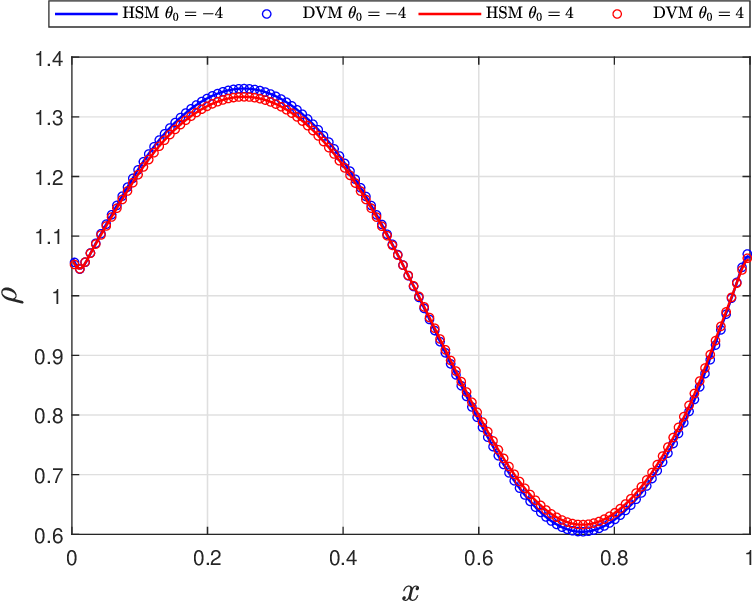}}\hfill
        \subfloat[$x$-direction velocity $u_1$]
    {\includegraphics[width=0.33\textwidth, height=0.26\textwidth]{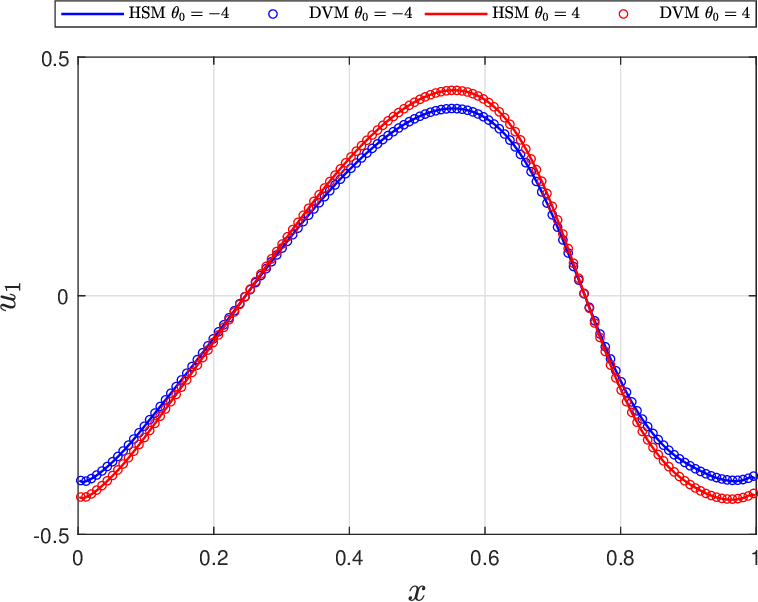}}\hfill
    \subfloat[internal energy $e_0$]
    {\includegraphics[width=0.33\textwidth, height=0.26\textwidth]{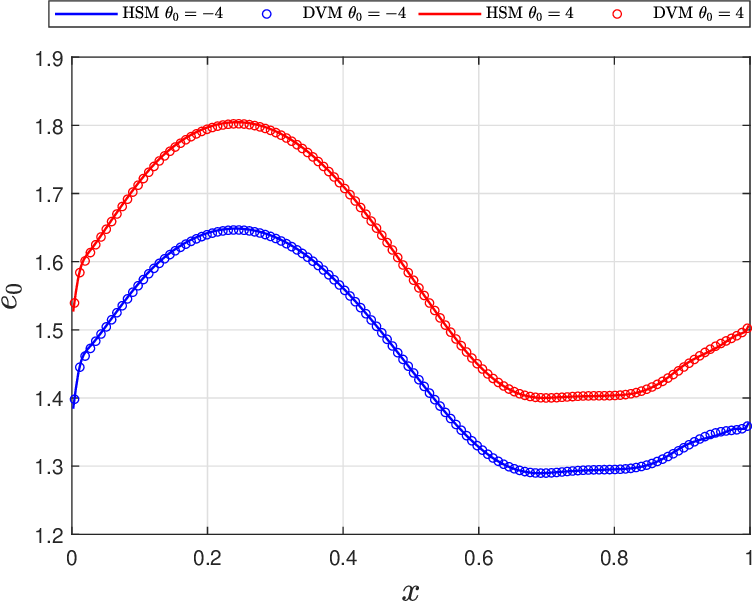}}\\
    \subfloat[fugacity $z|\theta_0|$]
    {\includegraphics[width=0.33\textwidth, height=0.26\textwidth]{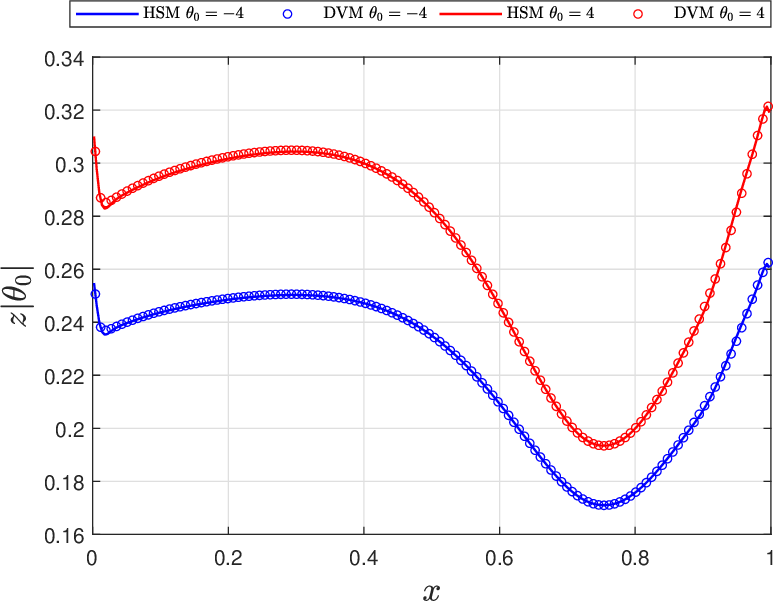}}\hfill  
      \subfloat[shear stress $p_{11}$]
    {\includegraphics[width=0.33\textwidth, height=0.26\textwidth]{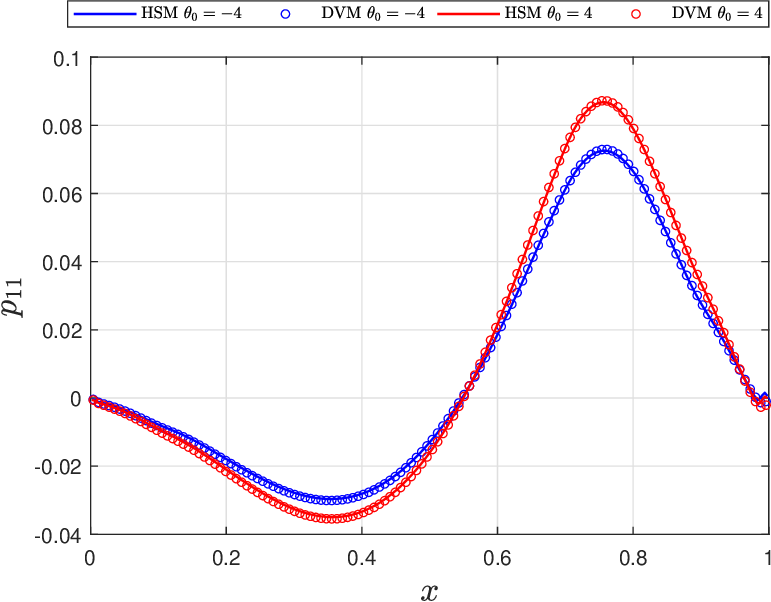}}\hfill 
    \subfloat[heat flux $q_1$]
    {\includegraphics[width=0.33\textwidth, height=0.26\textwidth]{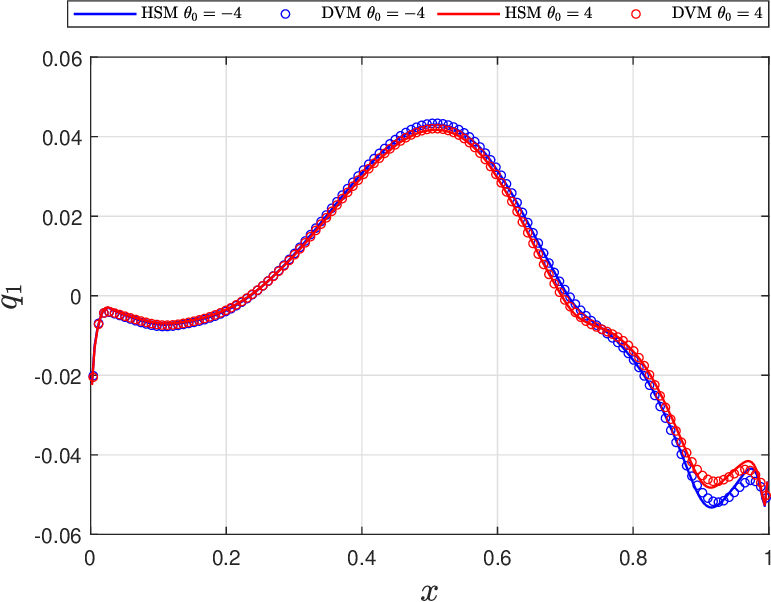}}
    \caption{(Mixing regime problem in Sec. \ref{sec:mix}) The numerical solutions of the mixing regime problem at $t=0.1$ for $\theta_0=\pm4$. The solid lines depict the numerical solutions, and the dots depict the reference solutions by DVM. The blue color represents those for $\theta_0 = -4$, and red color represents those for $\theta_0 = 4$. }
    \label{fig:Mix_DVM}
\end{figure}
In the simulation, periodic boundary conditions are employed, with the CFL number set to ${\rm CFL} = 0.3$. The expansion order and the grid size are chosen to be $M = 10$ and $N = 256$, respectively. The IMEX2 scheme \eqref{eq:second} is adopted with the WENO reconstruction. Besides, the expansion center is set to $[\ou, \oT]=[\bz, 1]$.

For $\theta_0 = \pm 4$, numerical solutions for the density $\rho$, macroscopic velocity in the $x$-direction $u_1$, internal energy $e_0$, fugacity $z|\theta_0|$, shear stress in the $x$-direction $p_{11}$, and heat flux in the $x$-direction $q_1$ at $t = 0.1$ are displayed in Fig. \ref{fig:Mix_DVM}. Reference solutions, computed using DVM, are also included for comparison. The numerical solutions exhibit good agreement with the reference solutions for both Bose and Fermi gases. Notably, due to the non-periodic nature of $\epsilon(x)$ in the spatial space, oscillations arise near the boundary in the numerical solutions, as observed in Fig. \ref{fig:Mix_DVM}, particularly for fugacity $z|\theta_0|$ and heat flux $q_1$. These oscillations are also well-captured by this Hermite spectral method.

To compare the efficiency of this Hermite spectral method (HSM) with DVM, the running time of these two methods is summarized in Tab. \ref{table:Mix_time}. All simulations are conducted on the CPU model Intel Xeon E5-2697A V4 @ 2.6GHz with 8 threads. For the DVM, the velocity space is discretized within $[-10, 10]$ by $80$ points in the $x-$direction and within $[-5, 5]$ by 20 points in the other two directions. Temporal and spatial discretizations in the DVM are kept consistent with the Hermite spectral method. Tab. \ref{table:Mix_time} reveals that for both $\theta_0$, the computational cost of HSM is significantly lower than that of DVM, which demonstrates the high efficiency of this Hermite spectral method. 

\begin{table}[!hptb]
\centering
\def\arraystretch{1.5}
{\footnotesize
\begin{tabular}{l|llll}
\hline
$\theta_0$ &  4  & -4        & 4       & -4        \\
\hline
Method & HSM & HSM & DVM & DVM \\
CFL number & 0.3 & 0.3 & 0.49 & 0.49 \\
time step & $2.26\times 10^{-4}$ & $2.26\times 10^{-4}$ & $1.81\times 10^{-4}$ & $1.81\times 10^{-4}$ \\
Grid number & 256 & 256 & 256 & 256 \\
\hline
Total CPU time $T_{\rm CPU}$ (s): & 281 & 287 & 3849 & 3393 \\ 
Elapsed time (Wall time) $T_{\rm Wall}$ (s): & 79.45 & 76.43 & 741.60 & 639.66\\
CPU time per time step (s): & 0.634 & 0.648 & 7.360 & 6.488 \\
CPU time per grid (s) & $2.48\times10^{-3}$ & $2.53\times10^{-3}$ & $2.88\times10^{-2}$ & $2.53\times10^{-2}$ \\
\hline
\end{tabular}
}
\caption{(Mixing regime problem in Sec. \ref{sec:mix}) Run-time data of the mixing regime problem for the Hermite spectral method and DVM.}
\label{table:Mix_time}
\end{table}

To investigate how the parameter $\theta_0$ influences the numerical solutions, we compute results for $\theta_0=-10, -4, -1, 0, 1, 4, 10$, respectively, while keeping other parameters consistent with the previous test. The numerical solutions of the macroscopic variables are illustrated in Fig. \ref{fig:Mix_theta}. It can be observed that for different $\theta_0$, the numerical solutions behave differently. When $|\theta_0|$ gets larger, the distinction between numerical solutions and the classical case (i.e. $\theta_0 = 0$) becomes more evident. Furthermore, it can be inferred that the numerical solutions exhibit continuous variation with $\theta_0$.


\begin{figure}[!ht]
    \centering
    \subfloat[density $\rho$]
    {\includegraphics[width=0.33\textwidth, clip]{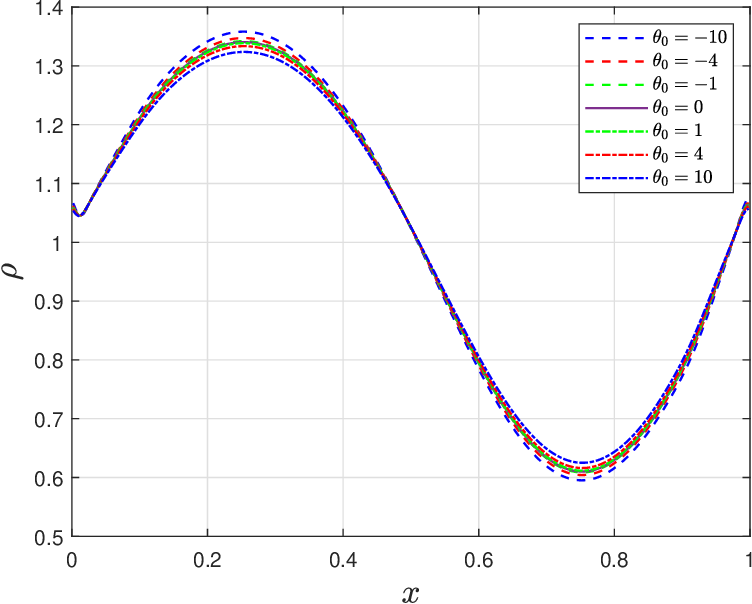}}\hfill
        \subfloat[$x$-direction velocity $u_1$]
    {\includegraphics[width=0.33\textwidth, clip]{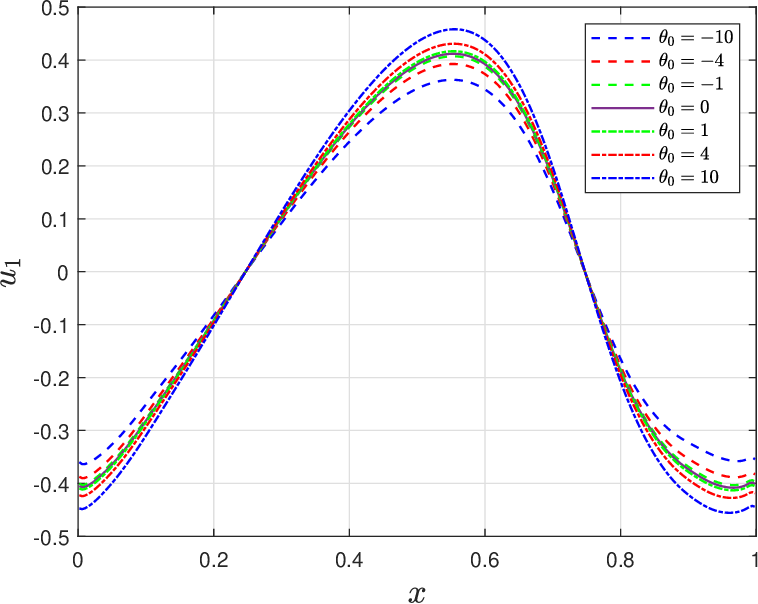}}\hfill
    \subfloat[internal energy $e_0$]
    {\includegraphics[width=0.33\textwidth, clip]{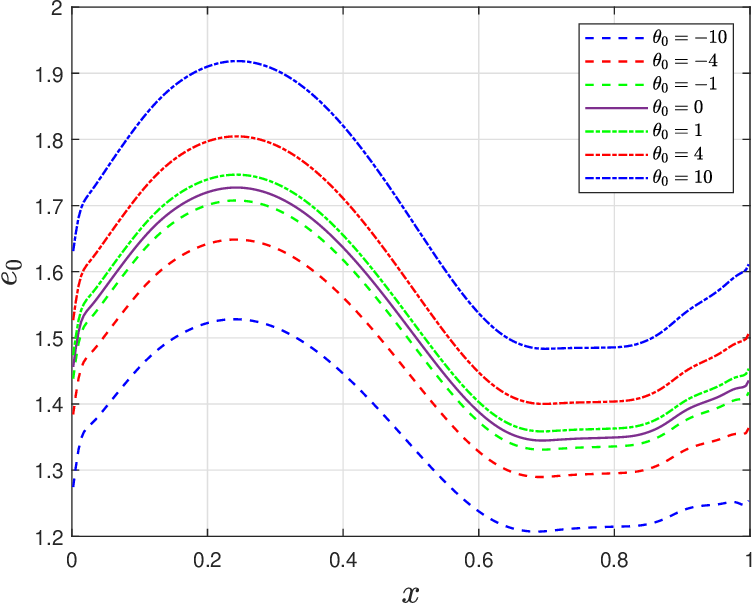}}\\
     \subfloat[fugacity $z|\theta_0|$]
    {\includegraphics[width=0.33\textwidth, clip]{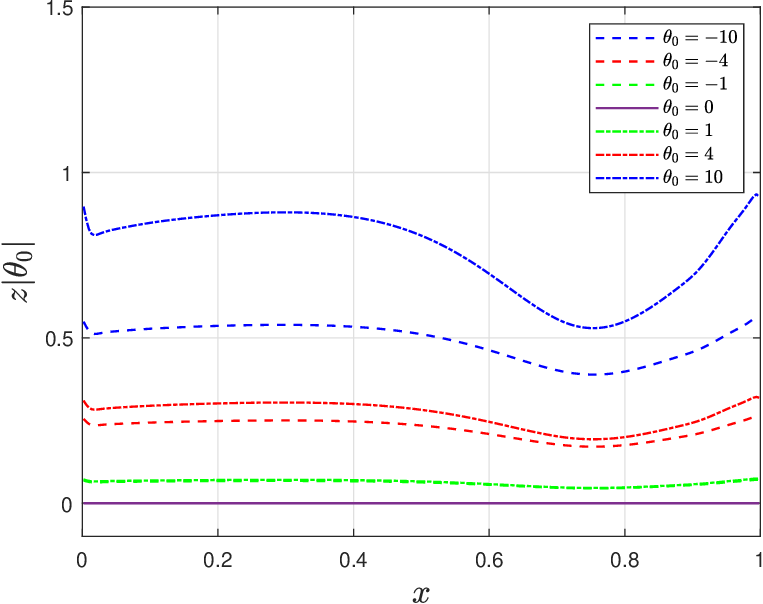}}\hfill 
      \subfloat[shear stress $p_{11}$]
    {\includegraphics[width=0.33\textwidth, clip]{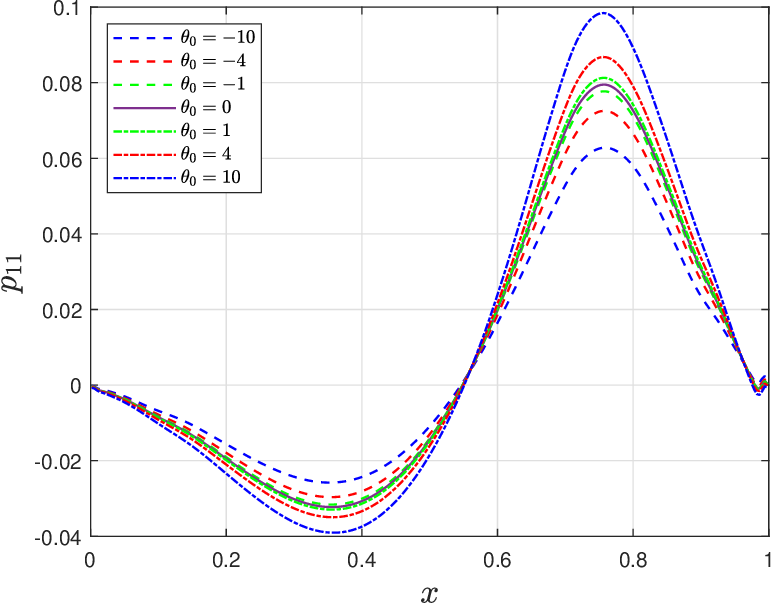}}\hfill 
    \subfloat[heat flux $q_1$]
    {\includegraphics[width=0.33\textwidth, clip]{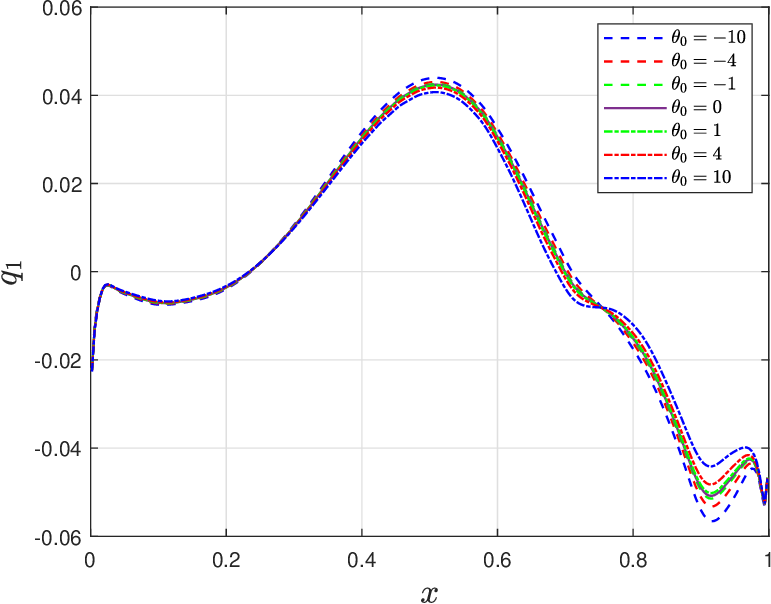}}
    \caption{(Mixing regime problem in Sec. \ref{sec:mix}) The numerical solutions of the mixing regime problem at $t=0.1$ for different $\theta_0$. }
    \label{fig:Mix_theta}
\end{figure}

\subsection{2D lid-driven cavity flow}
\label{sec:cavity}
\begin{figure}[!ht]
    \centering
    \subfloat[fugacity $z|\theta_0|$, $\theta_0=4$]
    {\includegraphics[width=0.33\textwidth, clip]{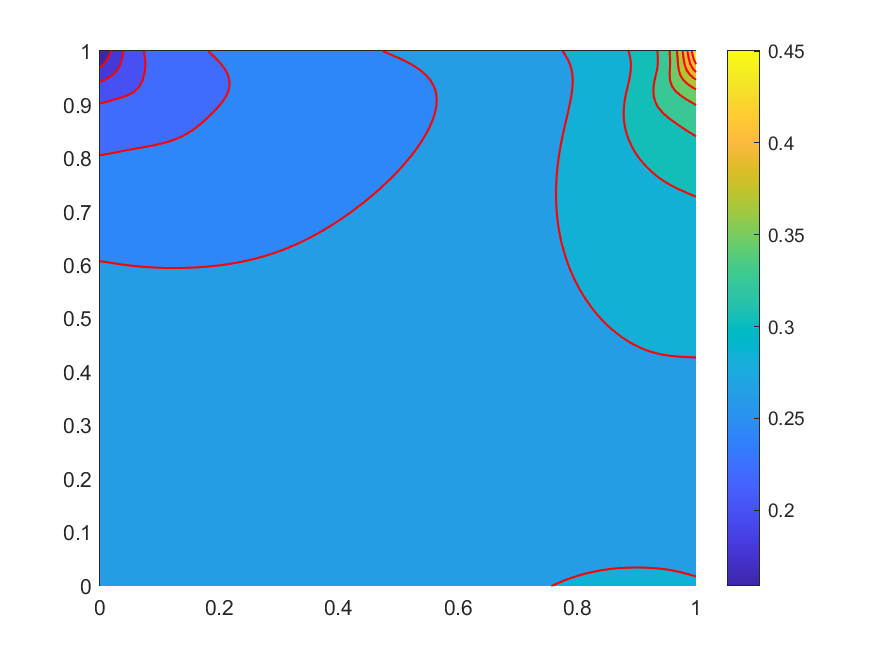}}\hfill
    \subfloat[fugacity $z|\theta_0|$, $\theta_0=0.01$]
    {\includegraphics[width=0.33\textwidth, clip]{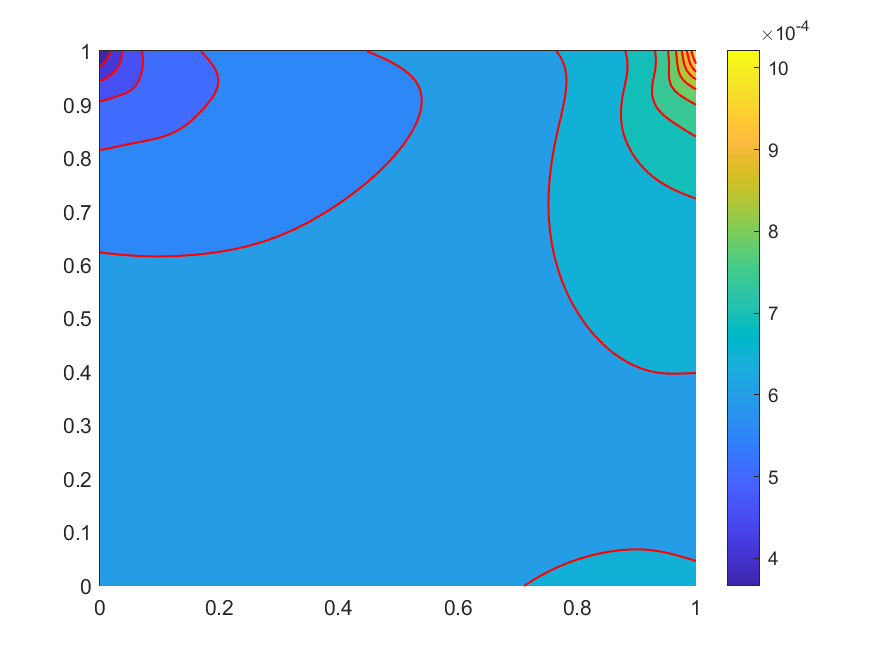}}\hfill
    \subfloat[fugacity $z|\theta_0|$, $\theta_0=-4$]
    {\includegraphics[width=0.33\textwidth, clip]{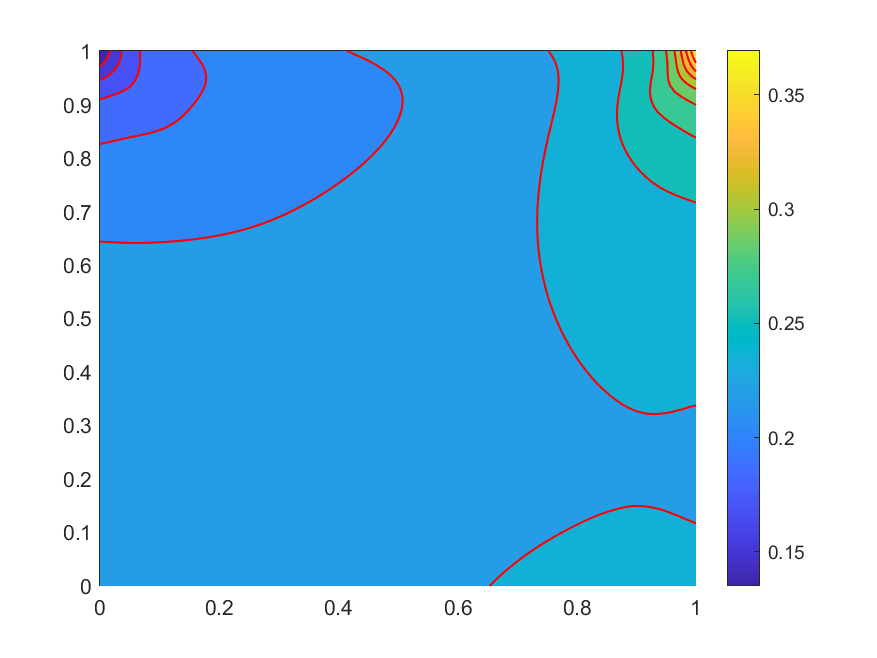}}\\
    \subfloat[Temperature $T$, $\theta_0=4$]
    {\includegraphics[width=0.33\textwidth, clip]{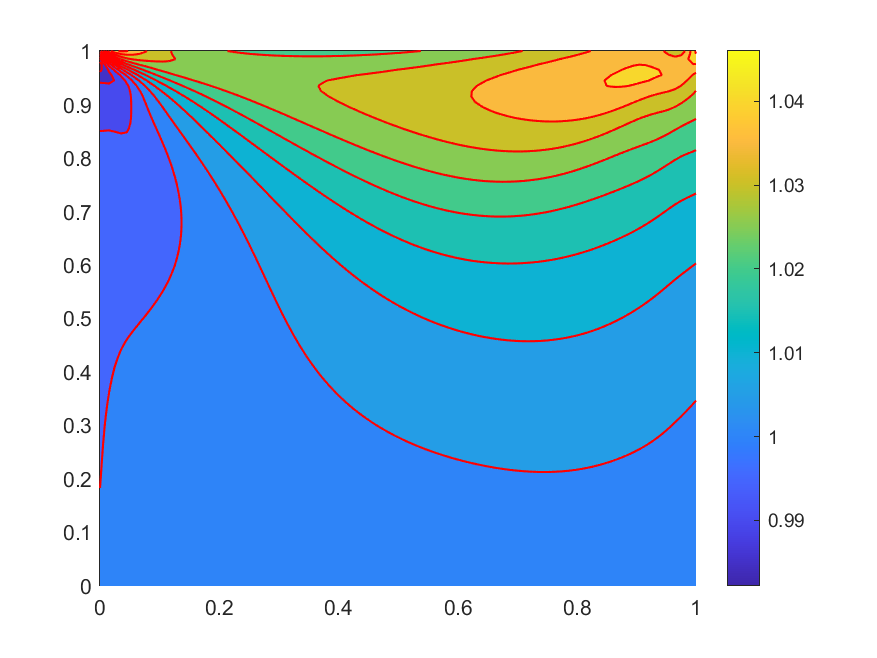}}\hfill
    \subfloat[Temperature $T$, $\theta_0=0.01$]
    {\includegraphics[width=0.33\textwidth, clip]{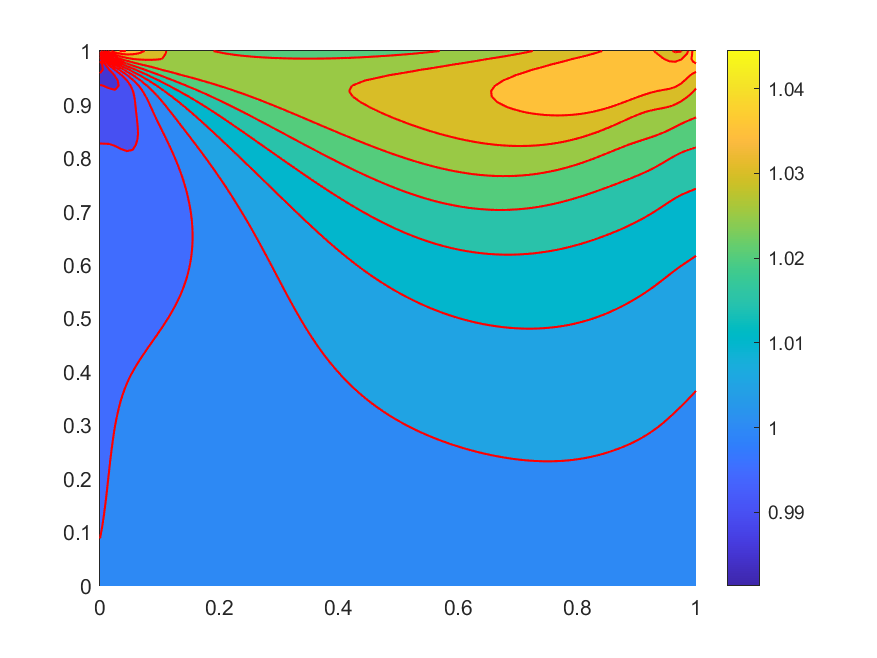}}\hfill
    \subfloat[Temperature $T$, $\theta_0=-4$]
    {\includegraphics[width=0.33\textwidth, clip]{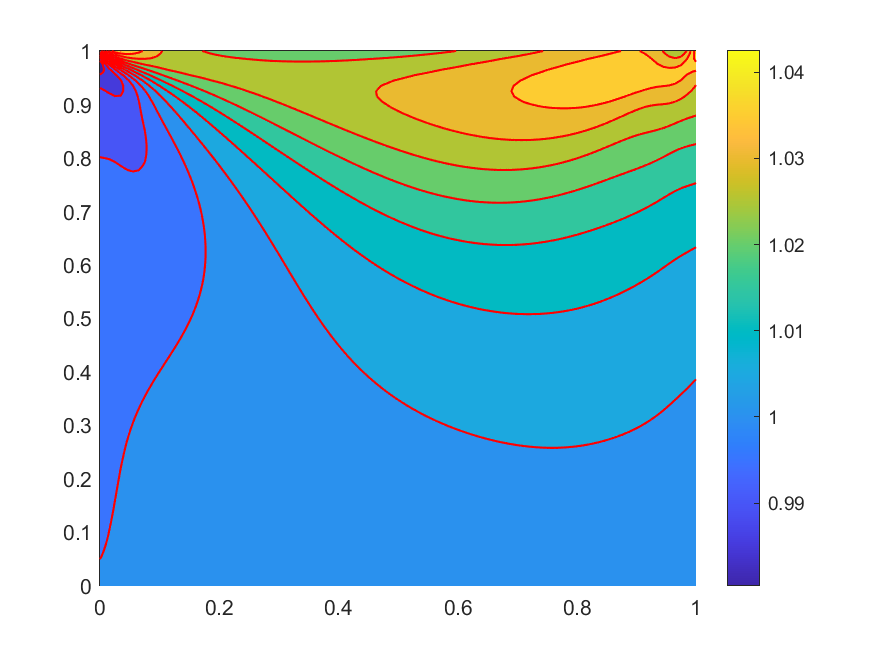}}\\
    \subfloat[shear stress $p_{12}$, $\theta_0=4$]
    {\includegraphics[width=0.33\textwidth, clip]{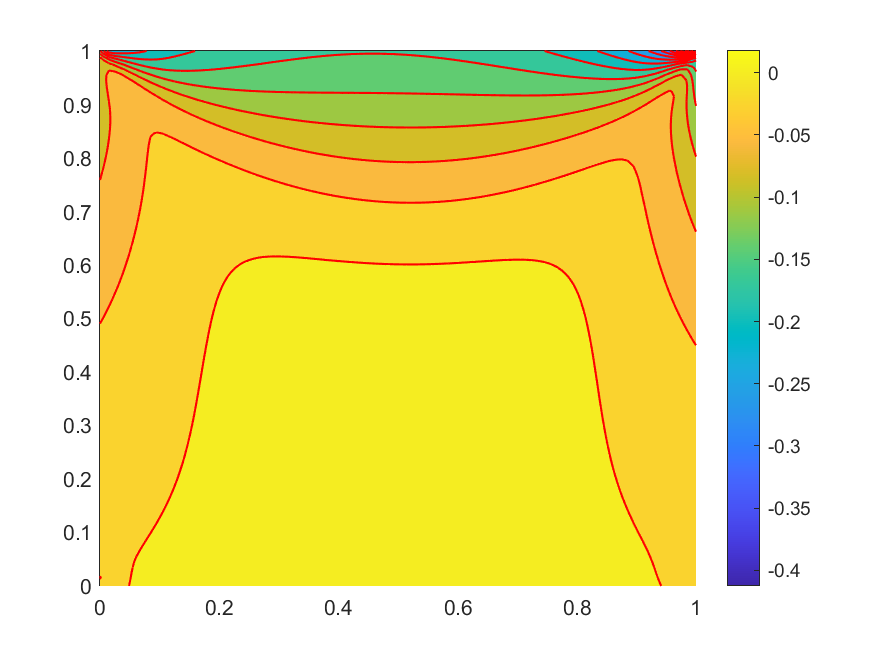}}\hfill
    \subfloat[shear stress $p_{12}$, $\theta_0=0.01$]
    {\includegraphics[width=0.33\textwidth, clip]{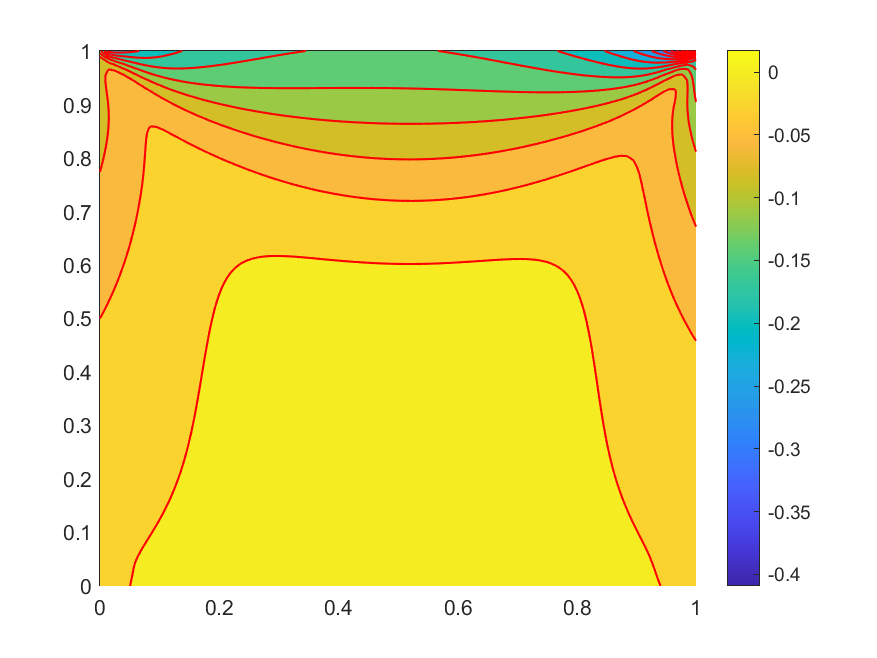}}\hfill
    \subfloat[shear stress $p_{12}$, $\theta_0=-4$]
    {\includegraphics[width=0.33\textwidth, clip]{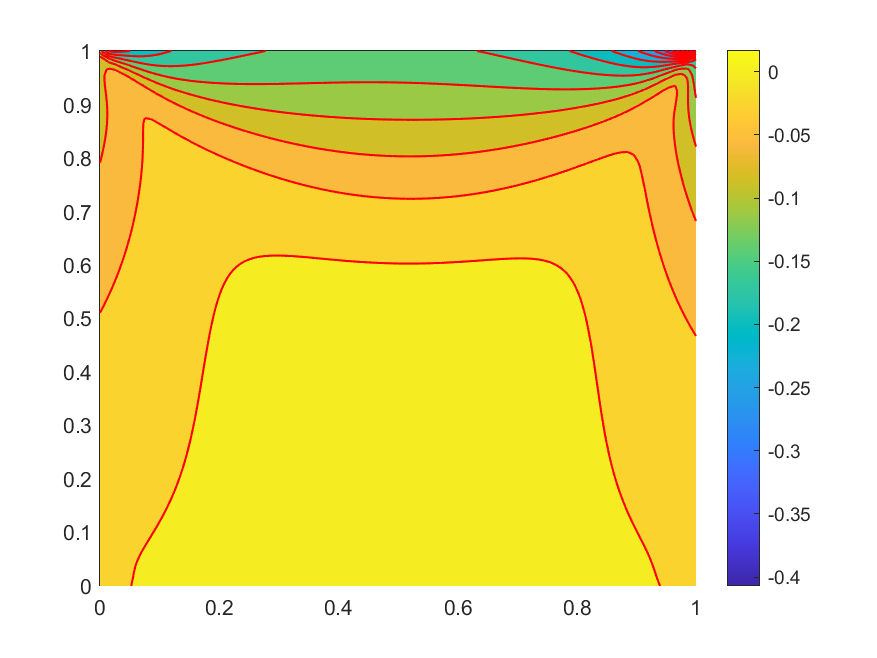}}\hfill
    \caption{(The lid-driven cavity flow in Sec. \ref{sec:cavity}) The numerical solutions of the fugacity $z|\theta_0|$, temperature $T$ and shear stress $p_{12}$ for $\epsilon = 0.1$. The left column presents the Fermi gas in the quantum regime for $\theta_0 = 4$, the middle column presents the near classical regime for $\theta_0 = 0.01$, and the right column presents the Bose gas in the quantum regime for $\theta_0 = -4$.}
    \label{fig:Cavity1}
\end{figure}

In this section, the lid-driven cavity flow in a spatially 2D and microscopically 3D setting is considered. The classical case of this problem has been widely studied, as seen in works like \cite{Liu2020, ZhichengHu2019}. In this scenario, the quantum gas is confined in a square cavity, where the top lid moves to the right at a constant speed, while all the other three walls remain stationary. All walls maintain the same temperature. Over an extended period, the gas reaches a steady state, which is the condition of particular interest. Due to the high dimensionality of this problem, capturing the steady state results in a considerable computational challenge.

For the classical problem, the Maxwell boundary condition \cite{Maxwell1878} is adopted. A similar boundary condition is utilized for the quantum gas. Specifically, assuming the velocity and temperature of the boundary wall are $\bu^w$ and $T^w$, respectively, the wall boundary condition at $\bx_0$ is given by 
\begin{equation}
    \label{eq:Boundary}
    f^w(t, \bx_0, \bv)=\left\{\begin{array}{ll}
         \mM^w_q(t,\bx_0,\bv)\triangleq \frac{1}{(z^w)^{-1}\exp\left(\frac{(\bv-\bu^w)^2}{2T^w}\right)+\theta_0}, & (\bv-\bu^w)\cdot \bn_0 \leqslant 0,  \\
       f(t, \bx_0, \bv),  &  (\bv- \bu^w) \cdot \bn_0 > 0,
    \end{array}
    \right.
\end{equation}
where $\bn_0$ represents the outer unit normal vector at $\bx_0$. Here, $z^w$ is determined such that the normal mass flux on the boundary is zero: 
\begin{equation}
    \label{eq:zeroflux}
    \int_{\bbR^3}\left[(\bv-\bu^w)\cdot \bn_0\right] f^w(t, \bx_0, \bv)\rd \bv=0.
\end{equation}
However, when $\theta_0 < 0$, there are cases where no $z^w\in [0,-1/\theta_0]$ satisfies \eqref{eq:zeroflux}. In such instances, the condition \eqref{eq:theta_BE} is not met. Under these circumstances, the boundary condition is adjusted as 
    \begin{equation}
        \label{eq:ad_BC}
        f^w(t, \bx_0, \bv)=C^w\mM^w_q(t,\bx_0,\bv),\qquad z^w = -\frac{1}{\theta_0},  \qquad (\bv - \bu^w) \cdot \bn_0 \geqslant 0,
    \end{equation}
where $C^w$ is a positive constant determined by $\eqref{eq:zeroflux}$. For detailed implementation of this boundary condition with the Hermite spectral method, readers may refer to \cite{ZhichengHu2019}.

In the simulation, the velocity of the top lid is set to $\bu^w = (0.5, 0, 0)$ with a uniform temperature $T^w = 1.0$ for all walls. A uniform grid mesh with $N = N_x = N_y = 100$ is employed for spatial discretization. The first-order scheme \eqref{eq:AP_time} is adopted with linear reconstruction in the spatial space. The CFL number is set as ${\rm CFL} = 0.2$, and the expansion center is chosen as $[\ou, \oT] = [\bz, 1]$. 

Firstly, the case with $\epsilon = 0.1$ is tested, employing an expansion order of $M = 10$. Numerical solutions of the fugacity $z|\theta_0|$, temperature $T$, and shear stress $p_{12}$ for $\theta_0 = 4, 0.01, -4$ are depicted in Fig. \ref{fig:Cavity1}, illustrating distinct behaviors for different $\theta_0$.

To examine the convergence of the numerical solutions, results along $x = 0.5$ and $y = 0.5$ are displayed in Fig. \ref{fig:Cavity_comp1} and \ref{fig:Cavity_comp2}, respectively, for grid sizes $N = 25, 50, 100$, and $200$. Both figures demonstrate that the numerical solutions for $z|\theta_0|$, $T$, and $p_{12}$ converge with increasing grid numbers for all $\theta_0$. This affirms the reliability of the numerical results.

For further validation, the case with $\epsilon = 1.0$ is investigated. The expansion number is increased to $M = 15$, while other settings are maintained from the test of $\epsilon = 0.1$. The numerical solutions for $\epsilon = 1.0$ are presented in Fig. \ref{fig:Cavity2}, exhibiting trends similar to those for $\epsilon = 0.1$.

All the experiments are conducted on the CPU model Intel Xeon E5-2697A V4 @ 2.6GHz with 28 threads. The simulations take approximately 7 hours with $M=10$ and around 30 hours with $M=15$, both for a spatial mesh of $N=100$. These results indicate that achieving a steady state simulation for this two-dimensional spatial and three-dimensional microscopic velocity problem is feasible with reasonable computational cost, which highlights the efficiency of the Hermite spectral method.

\begin{figure}[!ht]
    \centering
    \subfloat[fugacity $z|\theta_0|$, $\theta_0=4$]
    {\includegraphics[width=0.33\textwidth, height=0.26\textwidth, clip]{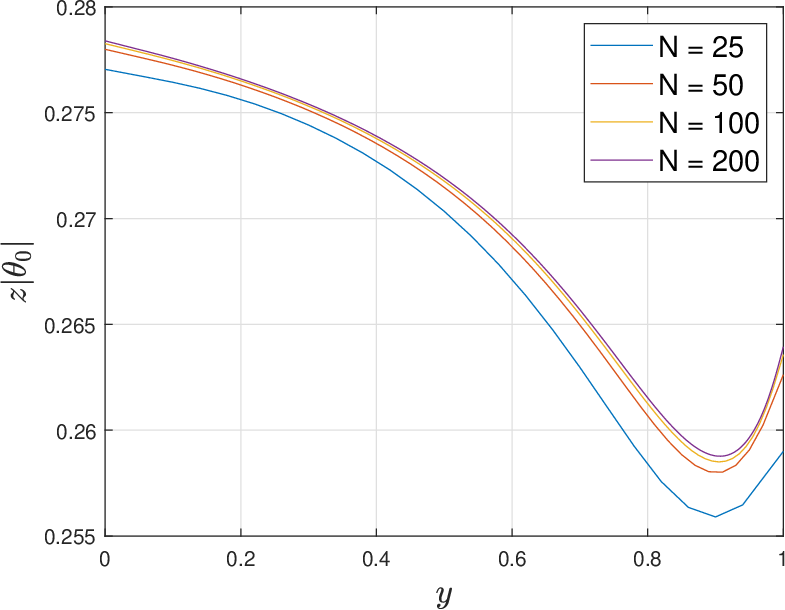}}\hfill
    \subfloat[fugacity $z|\theta_0|$, $\theta_0=0.01$]
    {\includegraphics[width=0.33\textwidth, height=0.27\textwidth, clip]{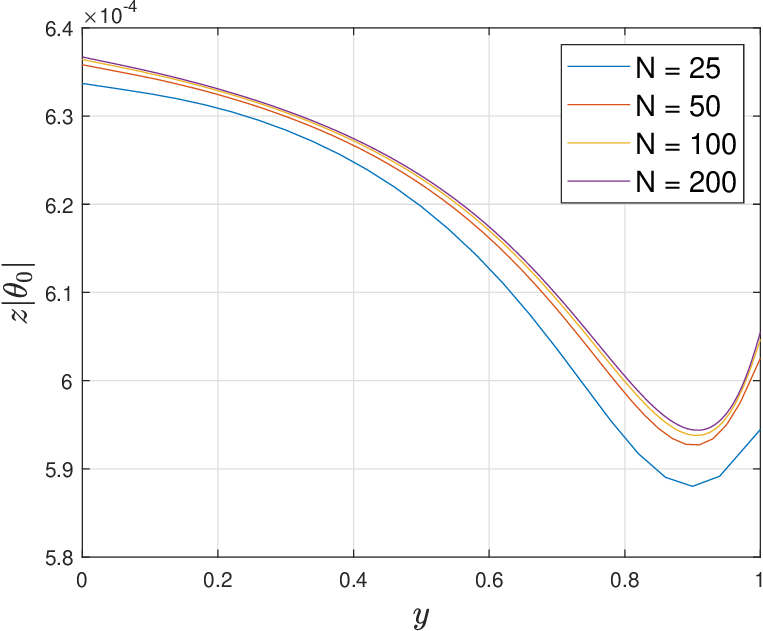}}\hfill
    \subfloat[fugacity $z|\theta_0|$, $\theta_0=-4$]
    {\includegraphics[width=0.33\textwidth, height=0.26\textwidth, clip]{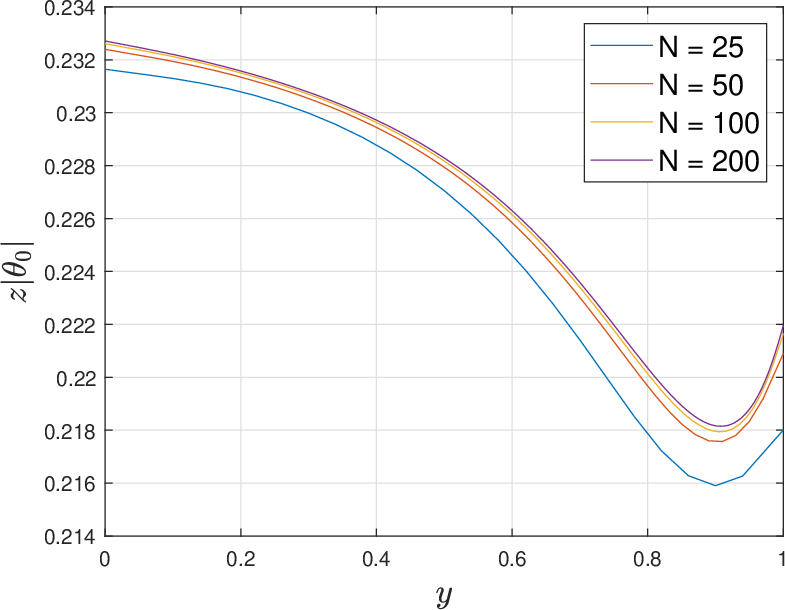}}\\
    \subfloat[Temperature $T$, $\theta_0=4$]
    {\includegraphics[width=0.33\textwidth, clip]{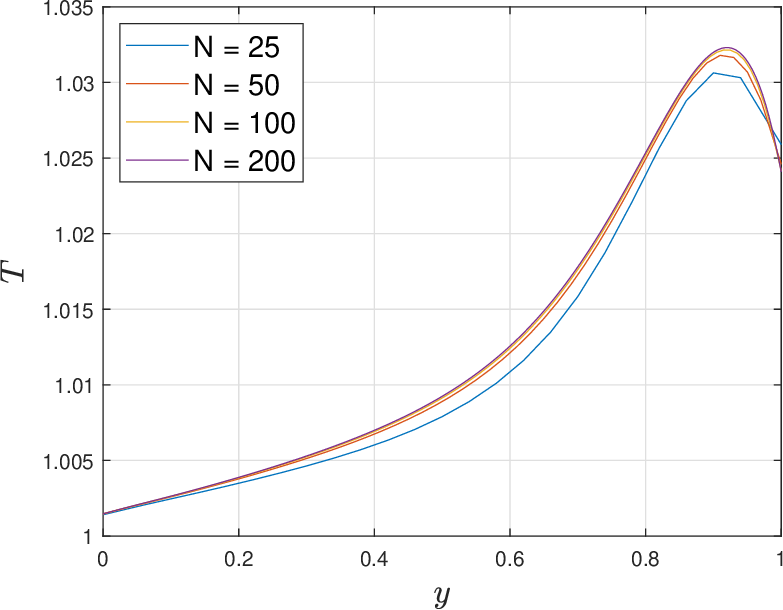}}\hfill
    \subfloat[Temperature $T$, $\theta_0=0.01$]
    {\includegraphics[width=0.33\textwidth, clip]{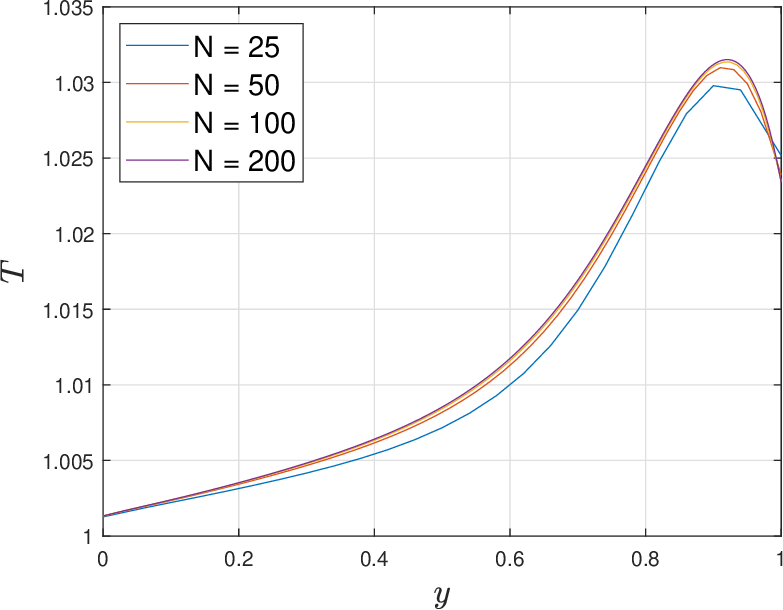}}\hfill
    \subfloat[Temperature $T$, $\theta_0=-4$]
    {\includegraphics[width=0.33\textwidth, clip]{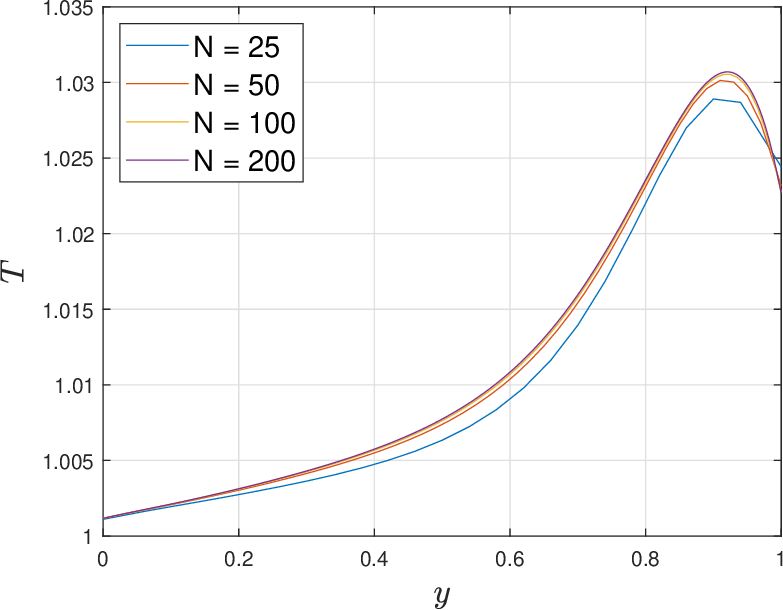}}\\
    \subfloat[shear stress $p_{12}$, $\theta_0=4$]
    {\includegraphics[width=0.33\textwidth, clip]{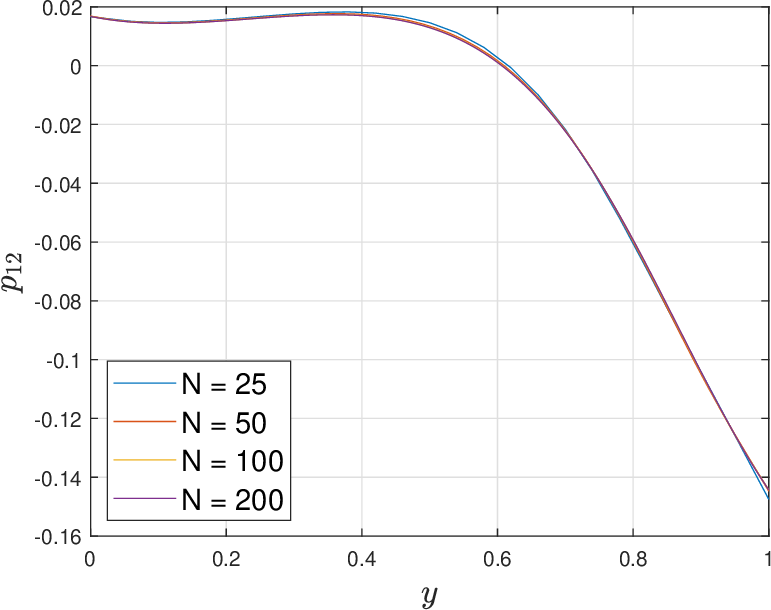}}\hfill
    \subfloat[shear stress $p_{12}$, $\theta_0=0.01$]
    {\includegraphics[width=0.33\textwidth, clip]{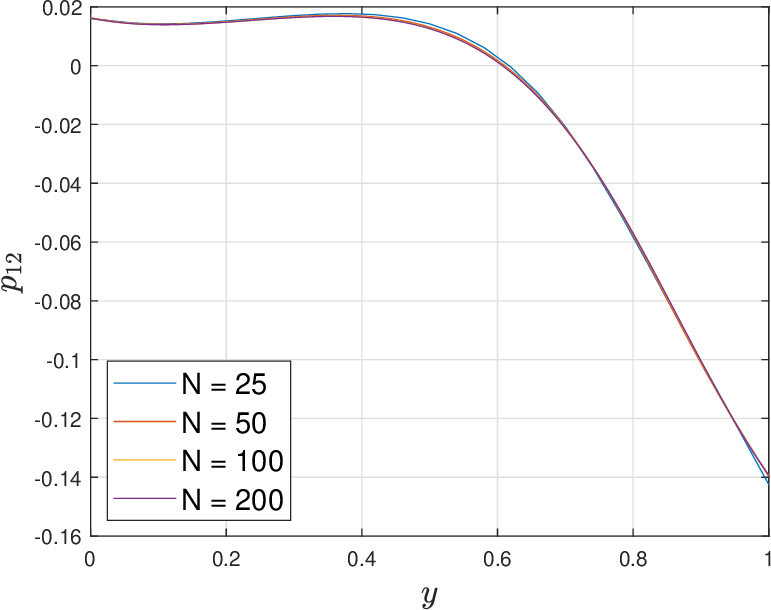}}\hfill
    \subfloat[shear stress $p_{12}$, $\theta_0=-4$]
    {\includegraphics[width=0.33\textwidth, clip]{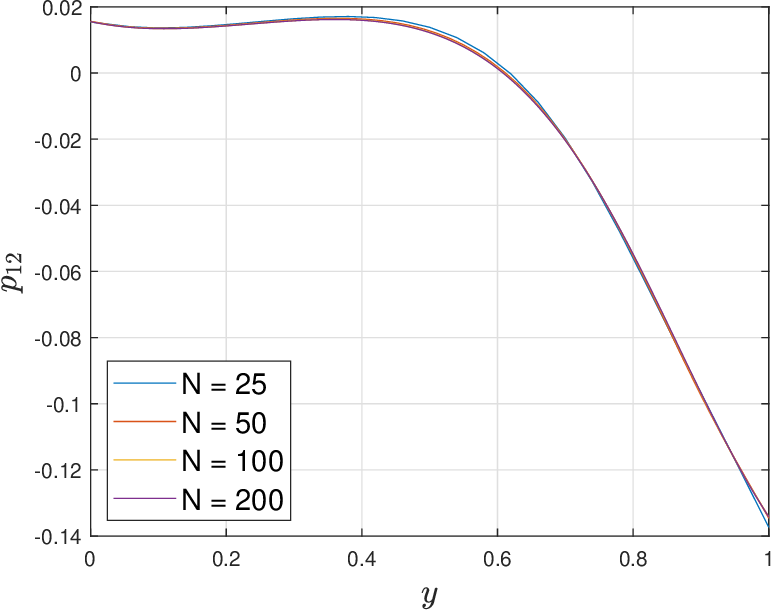}}\hfill
    \caption{(The lid-driven cavity flow in Sec. \ref{sec:cavity}) The numerical solutions of the fugacity $z|\theta_0|$, temperature $T$ and shear stress $p_{12}$ for $\epsilon = 0.1$ with different grid sizes along $x = 0.5$. The left column presents for the Fermi gas in the quantum regime for $\theta_0 = 4$, the middle column presents the near classical regime for $\theta_0 = 0.01$, and the right column presents the Bose gas in the quantum regime for $\theta_0 = -4$.}
    \label{fig:Cavity_comp1}
\end{figure}

\begin{figure}[!ht]
    \centering
    \subfloat[fugacity $z|\theta_0|$, $\theta_0=4$]
    {\includegraphics[width=0.33\textwidth, clip]{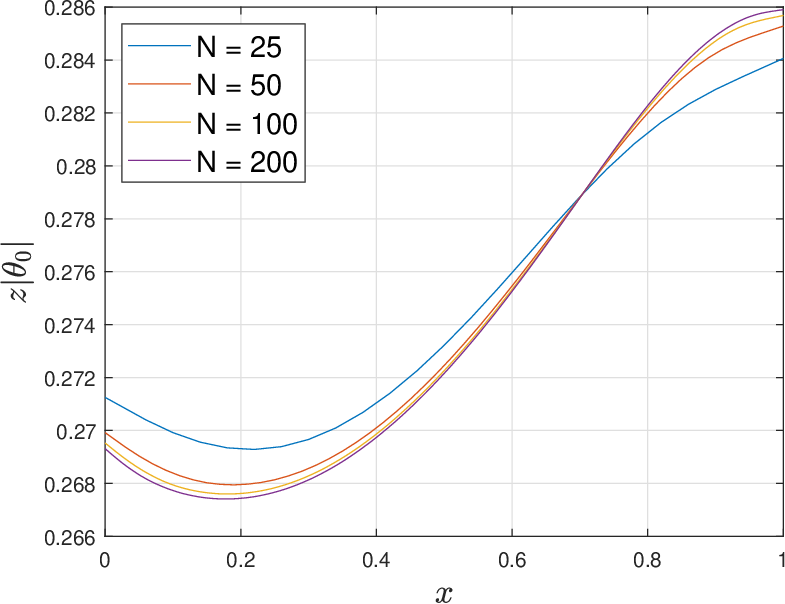}}\hfill
    \subfloat[fugacity $z|\theta_0|$, $\theta_0=0.01$]
    {\includegraphics[width=0.33\textwidth, clip]{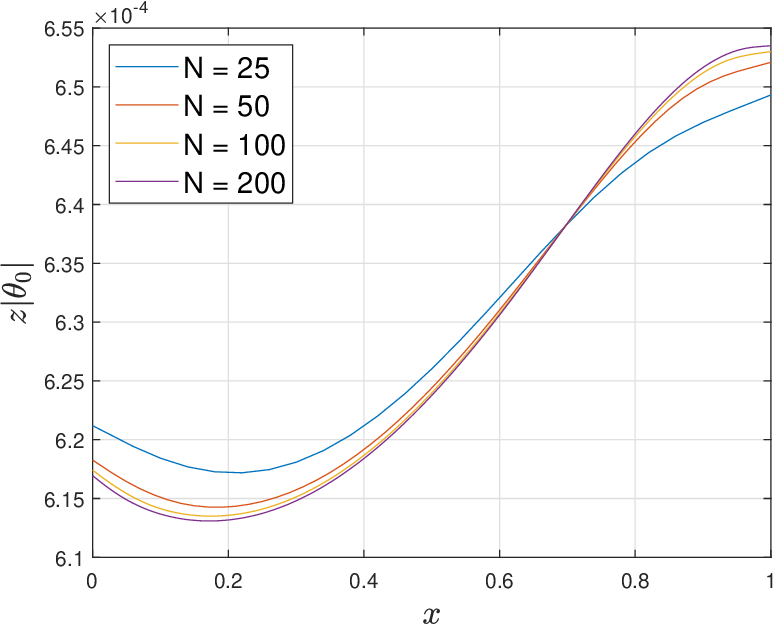}}\hfill
    \subfloat[fugacity $z|\theta_0|$, $\theta_0=-4$]
    {\includegraphics[width=0.33\textwidth, clip]{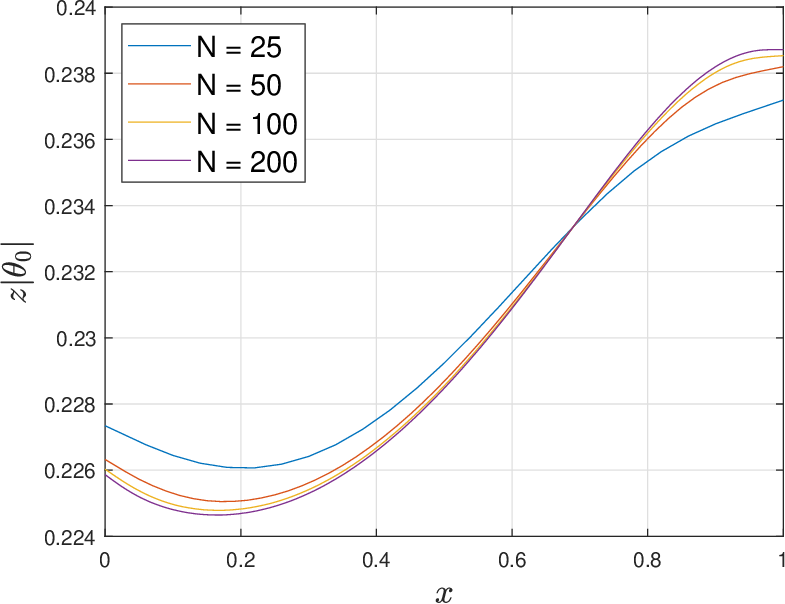}}\\
    \subfloat[Temperature $T$, $\theta_0=4$]
    {\includegraphics[width=0.33\textwidth, clip]{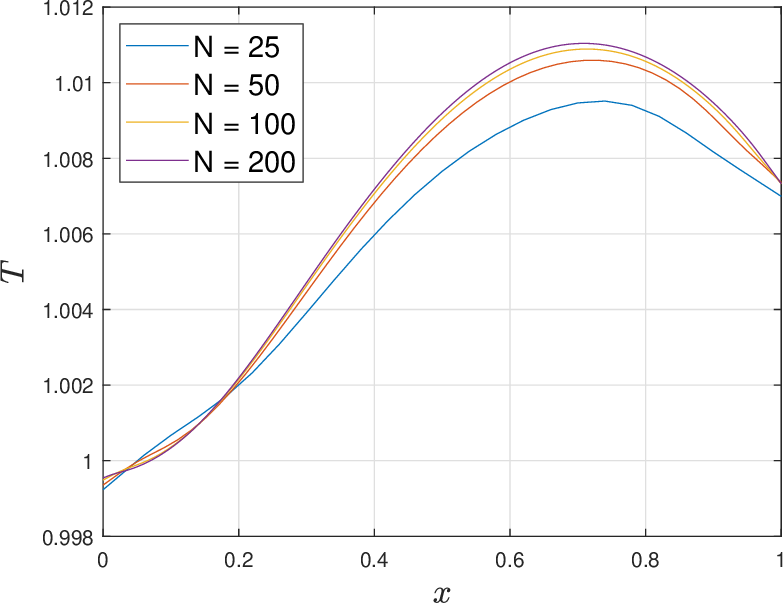}}\hfill
    \subfloat[Temperature $T$, $\theta_0=0.01$]
    {\includegraphics[width=0.33\textwidth, clip]{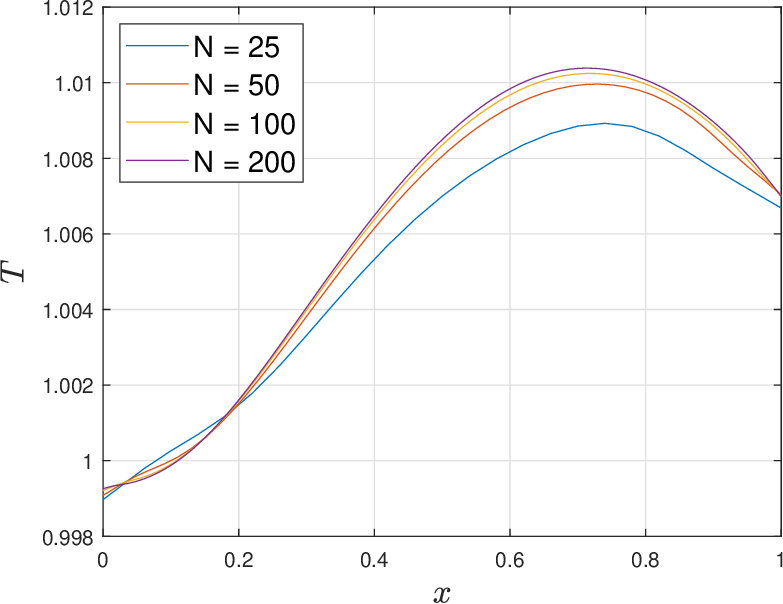}}\hfill
    \subfloat[Temperature $T$, $\theta_0=-4$]
    {\includegraphics[width=0.33\textwidth, clip]{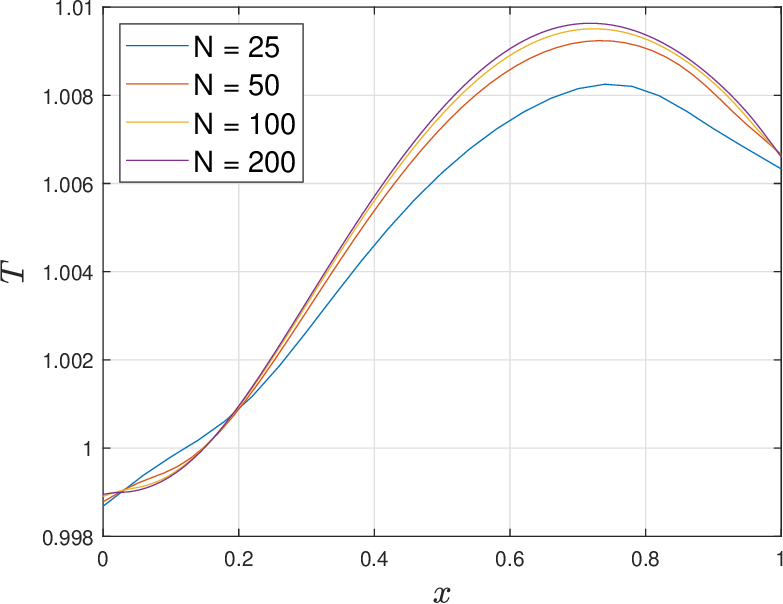}}\\
    \subfloat[shear stress $p_{12}$, $\theta_0=4$]
    {\includegraphics[width=0.33\textwidth, clip]{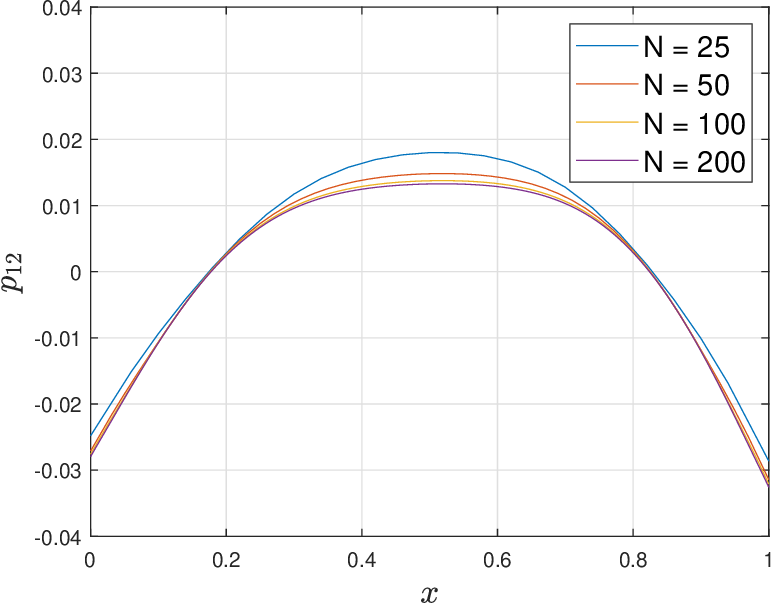}}\hfill
    \subfloat[shear stress $p_{12}$, $\theta_0=0.01$]
    {\includegraphics[width=0.33\textwidth, clip]{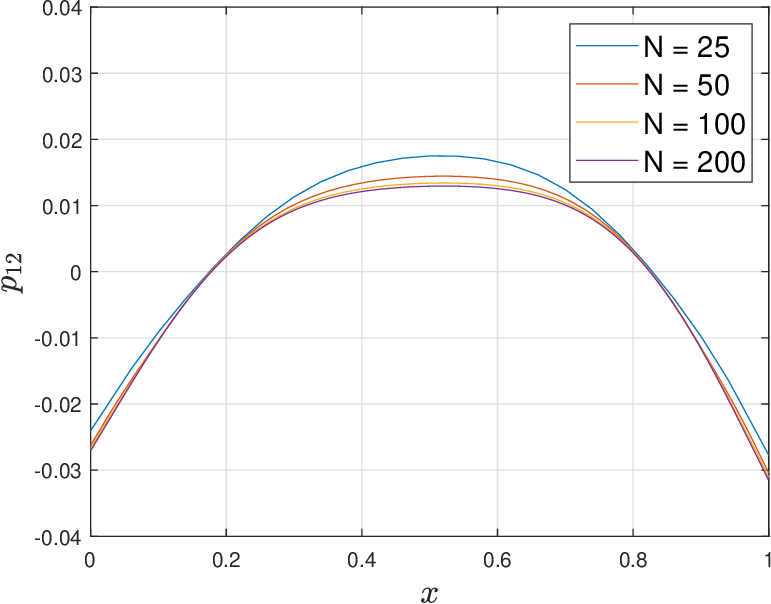}}\hfill
    \subfloat[shear stress $p_{12}$, $\theta_0=-4$]
    {\includegraphics[width=0.33\textwidth, clip]{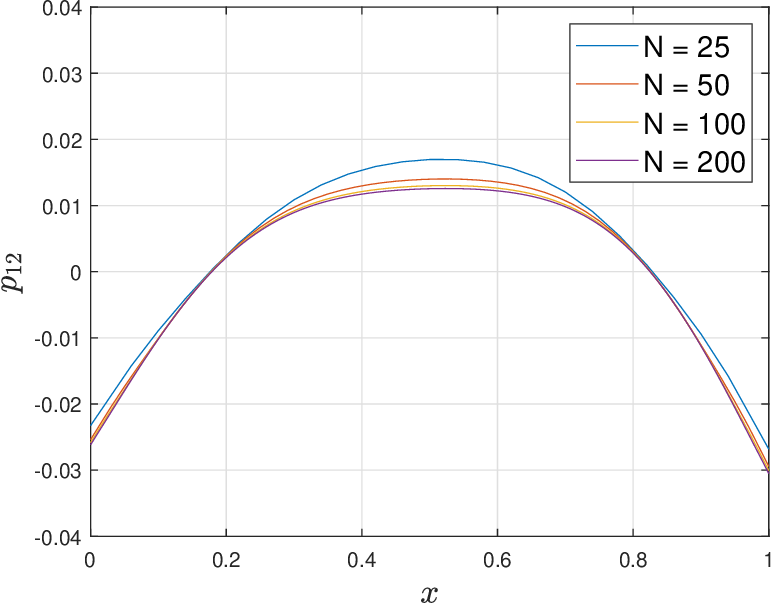}}\hfill
    \caption{(The lid-driven cavity flow in Sec. \ref{sec:cavity}) The numerical solutions of the fugacity $z|\theta_0|$, temperature $T$ and shear stress $p_{12}$ for $\epsilon = 0.1$ with different grid sizes along $y = 0.5$. The left column presents for the Fermi gas in the quantum regime for $\theta_0 = 4$, the middle column presents the near classical regime for $\theta_0 = 0.01$, and the right column presents the Bose gas in the quantum regime for $\theta_0 = -4$. }
    \label{fig:Cavity_comp2}
\end{figure}

\begin{figure}[!ht]
    \centering
    \subfloat[fugacity $z|\theta_0|$, $\theta_0=4$]
    {\includegraphics[width=0.33\textwidth, clip]{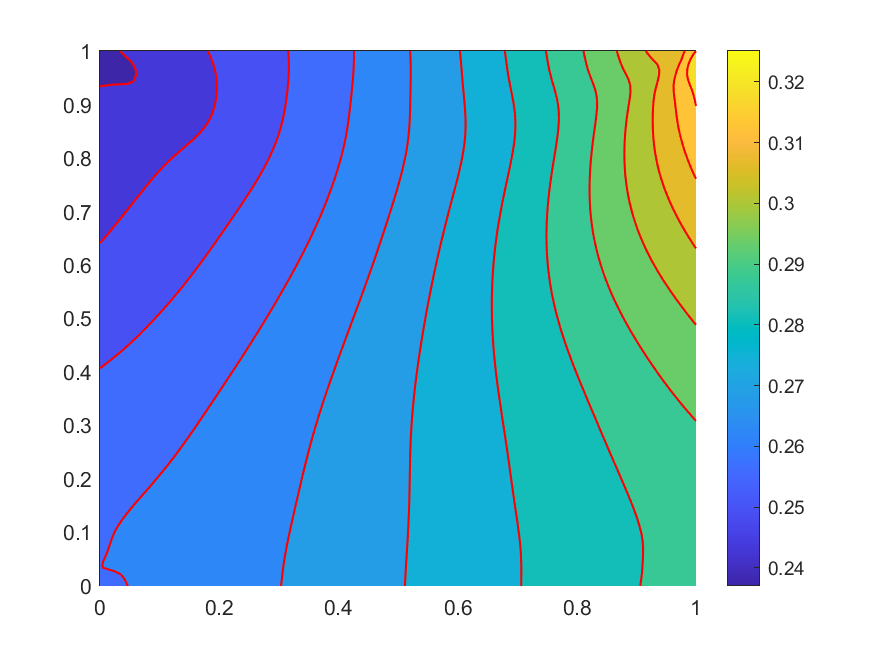}}\hfill
    \subfloat[fugacity $z|\theta_0|$, $\theta_0=0.01$]
    {\includegraphics[width=0.33\textwidth, clip]{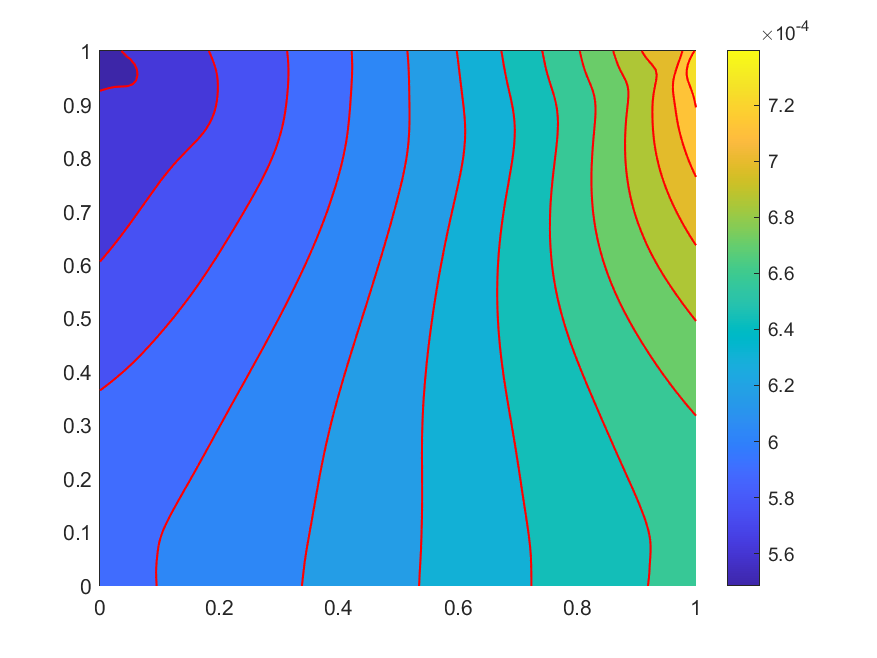}}\hfill
    \subfloat[fugacity $z|\theta_0|$, $\theta_0=-4$]
    {\includegraphics[width=0.33\textwidth, clip]{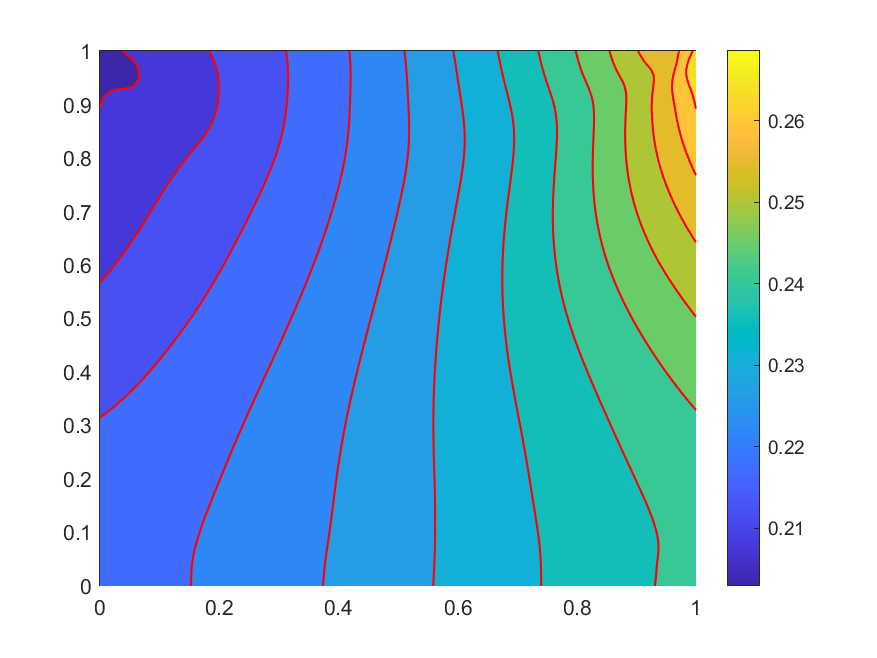}}\\
    \subfloat[Temperature $T$, $\theta_0=4$]
    {\includegraphics[width=0.33\textwidth, clip]{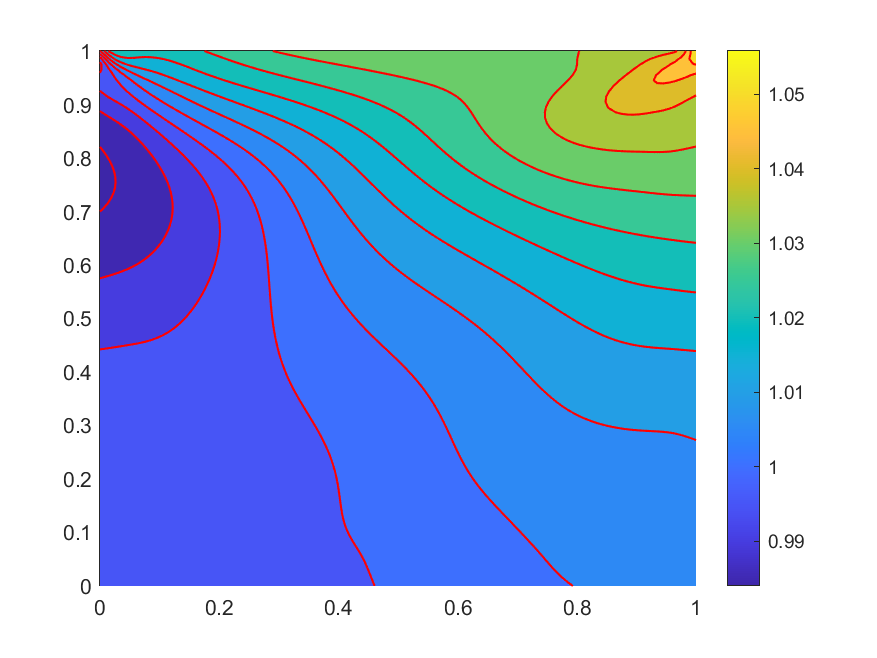}}\hfill
    \subfloat[Temperature $T$, $\theta_0=0.01$]
    {\includegraphics[width=0.33\textwidth, clip]{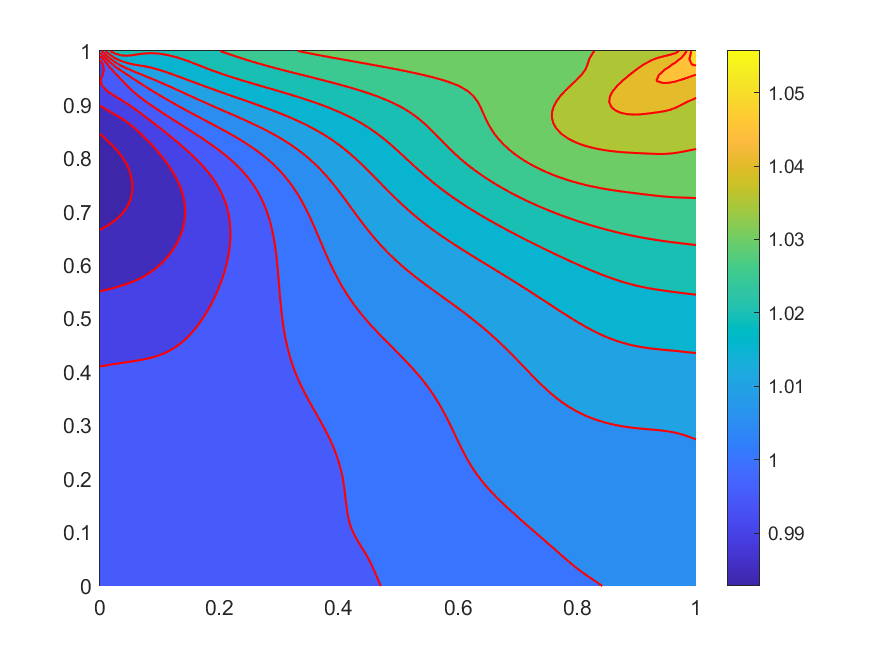}}\hfill
    \subfloat[Temperature $T$, $\theta_0=-4$]
    {\includegraphics[width=0.33\textwidth, clip]{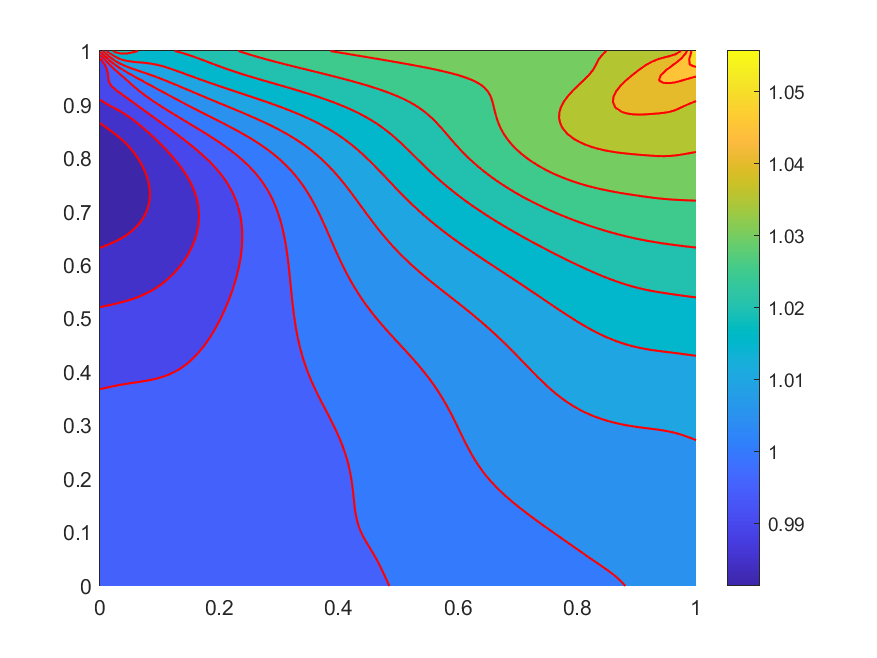}}\\
    \subfloat[shear stress $p_{12}$, $\theta_0=4$]
    {\includegraphics[width=0.33\textwidth, clip]{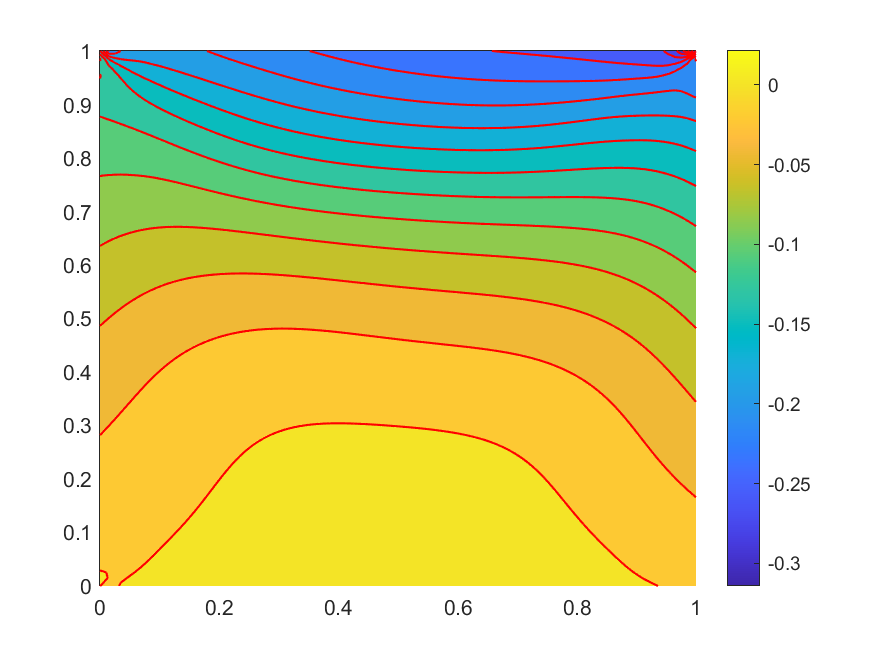}}\hfill
    \subfloat[shear stress $p_{12}$, $\theta_0=0.01$]
    {\includegraphics[width=0.33\textwidth, clip]{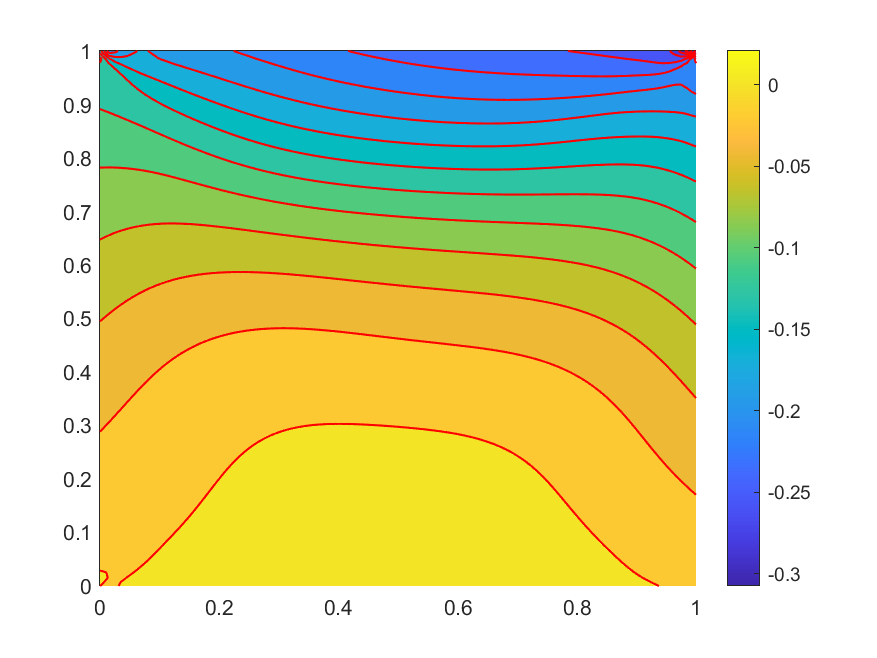}}\hfill
    \subfloat[shear stress $p_{12}$, $\theta_0=-4$]
    {\includegraphics[width=0.33\textwidth, clip]{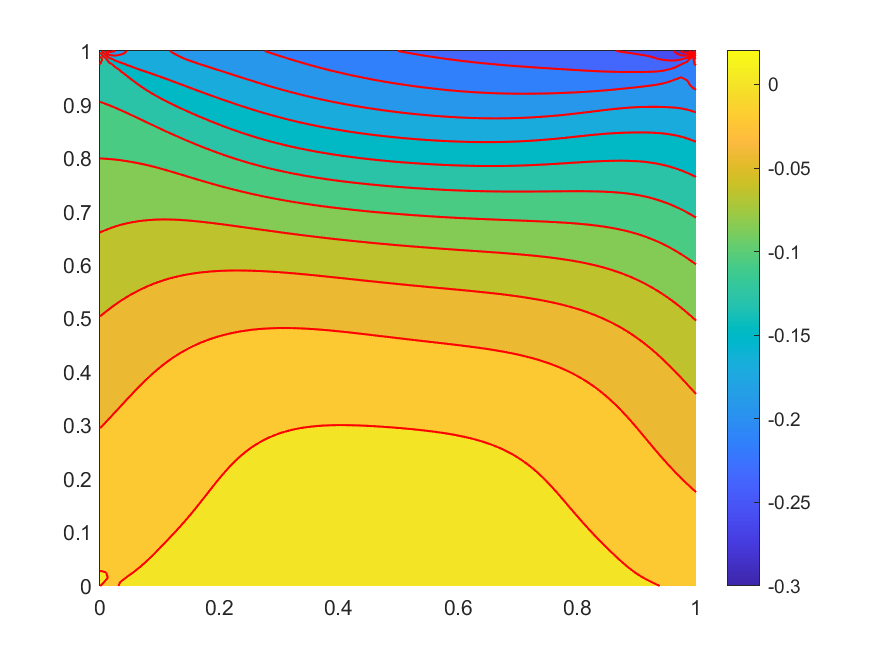}}\hfill
    \caption{(The lid-driven cavity flow in Sec. \ref{sec:cavity}) The numerical solutions of the fugacity $z|\theta_0|$, temperature $T$ and shear stress $p_{12}$ for $\epsilon = 1$. The left column presents for the Fermi gas in the quantum regime for $\theta_0 = 4$, the middle column presents the near classical regime for $\theta_0 = 0.01$, and the right column presents the Bose gas in the quantum regime for $\theta_0 = -4$. }
    \label{fig:Cavity2}
\end{figure}

\section{Conclusion}
\label{sec:conclusion}
We have proposed an asymptotic preserving IMEX Hermite spectral method for the quantum BGK equation. To enhance the overall efficiency of the numerical scheme, a fast algorithm for computing the polylogarithm is introduced to derive the expansion coefficients of the quantum Maxwellian. In the numerical experiments, the AP property has been successfully verified. Subsequently, the numerical scheme has been validated through simulations of the spatially 1D Sod and mixing regime problems. Finally, this Hermite spectral method is applied to a spatially 2D lid-driven cavity flow problem, which further demonstrates its outstanding efficiency.

\section*{Acknowledgments}
The work of Ruo Li is partially supported by the National Natural Science Foundation of China (Grant No. 12288101).
This work of Yanli Wang is partially supported by the National Natural Science Foundation of China (Grant No. 12171026, U2230402, and 12031013), and Foundation of President of China Academy of Engineering Physics (YZJJZQ2022017). 
This research is supported by High-performance Computing Platform of Peking University.


\section{Appendix}
\label{sec:app}
\subsection{Comparison to \textbf{polylog} in MATLAB} 
\label{app:compare}
To validate the algorithm for calculating the polylogarithm, as proposed in Sec. \ref{sec:integral}, we compare the results with the MATLAB function \textbf{polylog}. The error $e=\big|\Li_{s, \rm num}(y)-\Li_{s, \rm ref}(y)\big|$ is recorded for $s=1.5, 2.5, 3.5$ and $y\in[-10,1]$ with an interval of $0.01$. Here, $\Li_{\rm num} $ corresponds to the result obtained using the method proposed in Sec. \ref{sec:integral}, and $\Li_{\rm ref}$ is the reference result obtained using the MATLAB function \textbf{polylog}. 

\begin{figure}[!ht]
    \centering
      \subfloat[Error of $\Li_{1.5}(y)$.]
    {\includegraphics[width=0.3\textwidth, height=0.24\textwidth,
      clip]{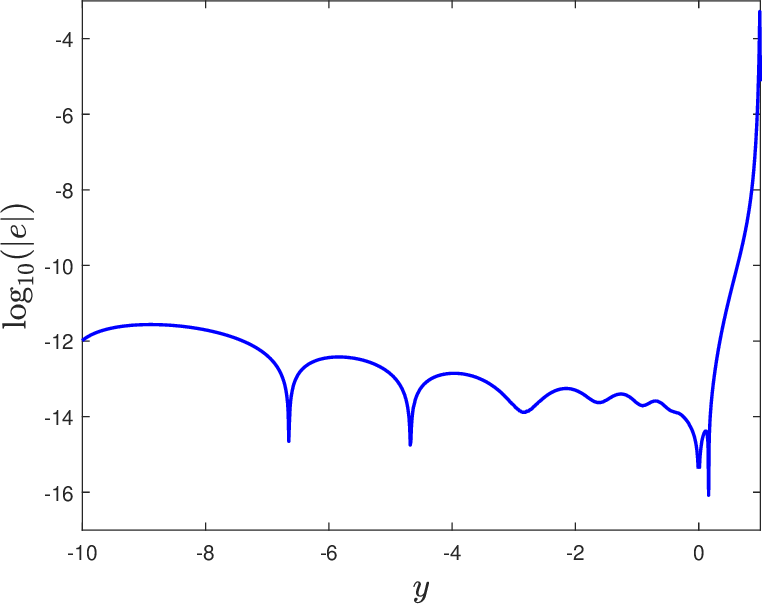}}\hfill
      \subfloat[Error of $\Li_{2.5}(y)$.]
    {\includegraphics[width=0.3\textwidth, height=0.24\textwidth,
      clip]{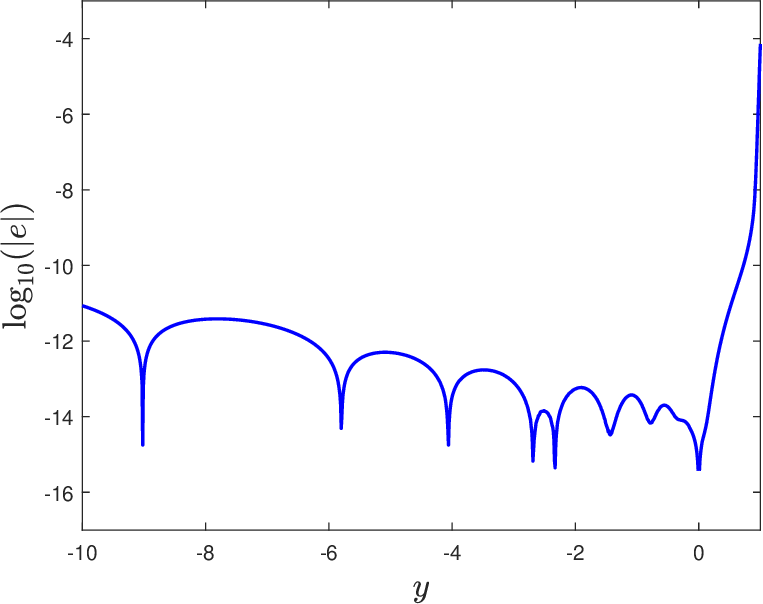}}\hfill
      \subfloat[Error of $\Li_{3.5}(y)$.]
    {\includegraphics[width=0.3\textwidth, height=0.24\textwidth,
      clip]{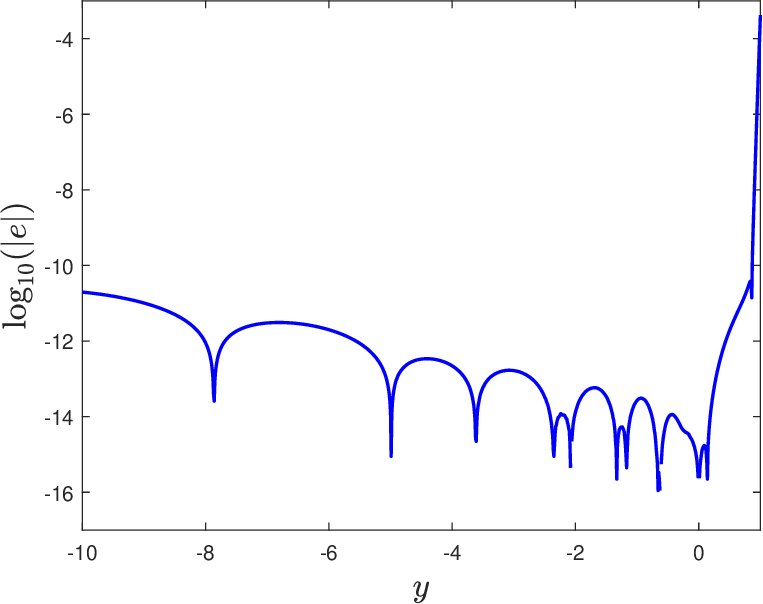}}\hfill
      \caption{(Comparison to \textbf{polylog} in MATLAB in App. \ref{app:compare}) The error $e = \big|\Li_{\rm num}(y) - \Li_{\rm ref}(y)\big|, \; s=1.5,2.5,3.5$. The reference result is obtained using the MATLAB function \textbf{polylog}.}    
      \label{fig:polylog}
\end{figure}

Fig. \ref{fig:polylog} shows that the error $e$ is quite small when $y$ is far from $1$, which is close to the machine's precision. The error increases as $y$ approaches $1$ due to the singularity of the integral \eqref{eq:algo-Li}. To illustrate the efficiency of this integral algorithm, we record the running time to obtain $\Li_{s, \rm num}$ and $\Li_{s, \rm ref}$ in Tab. \ref{table:polylog_time}. The results reveal that the integral algorithm is significantly faster compared to \textbf{polylog}, which demonstrates its high efficiency.

\begin{table}[!ht]
\centering
\def\arraystretch{1.5}
{\footnotesize
\begin{tabular}{llll}
\hline
    & $s=1.5$ & $s=2.5$ & $s=3.5$        \\
\hline
Integral method (s) & $1.97\times10^{-3}$s & $3.93\times10^{-3}$s & $4.16\times10^{-3}$s\\
\textbf{polylog} (s) & $9.67$s & $9.82$s & $11.82$s\\
\hline
\end{tabular}
}
\caption{(Comparison to \textbf{polylog} in MATLAB in App. \ref{app:compare}) Total running time (on MATLAB R2020b) to obtain $\Li_{s, \rm num}$ and $\Li_{s, \rm ref}$.}
\label{table:polylog_time}
\end{table}




\subsection{Newton iteration method}
\label{app:Newton}
In this section, a simple experiment will be conducted to verify the efficiency of the Newton iteration method introduced in Sec. \ref{sec:solve} for obtaining $z$ and $T$. We consider a spatial homogeneous problem with a source term, and the governing equation \eqref{eq:Boltz} is reduced to
\begin{equation}
    \label{eq:exp_Newton}
    \pd{f}{t}=(\mM_q-f)+ \mathcal{S}(t),
\end{equation}
where $\mathcal{S}(t)$ is the source term defined as
\begin{equation}
    \label{eq:source}
        \mathcal{S} = \frac{\rho_r(t)}{ (2 \pi T_r(t))^{3/2}}\exp\left(-\frac{|\bv|^2}{2T_r(t)}\right), 
\end{equation}
with $\rho_r(t)$ and $T(t)$ being random variables uniformly and independently distributed in the interval $[0.2, 1.8]$ for any $t$. The initial condition is given by the summation of two equilibrium states as
\begin{equation}
    \label{eq:ini_exp_Newton}
    f(0,\bv)=\frac12\left(\frac{1}{z^{-1}\exp\left(\frac{(\bv-\bu)^2}{2T}\right)+\theta_0}+\frac{1}{z^{-1}\exp\left(\frac{(\bv+\bu)^2}{2T}\right)+\theta_0}\right),
\end{equation}
where $T=1$, $\bu=(1,0,0)$, and $z$ is chosen such that $\rho(0)=1$. The time step size is set as $\Delta t=0.001$, which is similar in scale to the simulations in Sec. \ref{sec:experiment}. We set $\theta_0 = \pm 0.01$ and $\pm 9$, and the iteration counts for different $\theta_0$ at each time step are shown in Fig. \ref{fig:Newton} for the final time $t = 0.1$.  

\begin{figure}[!ht]
    \centering
      \subfloat[$\theta_0=\pm0.01$]
    {\includegraphics[width=0.45\textwidth, height=0.36\textwidth,
      clip]{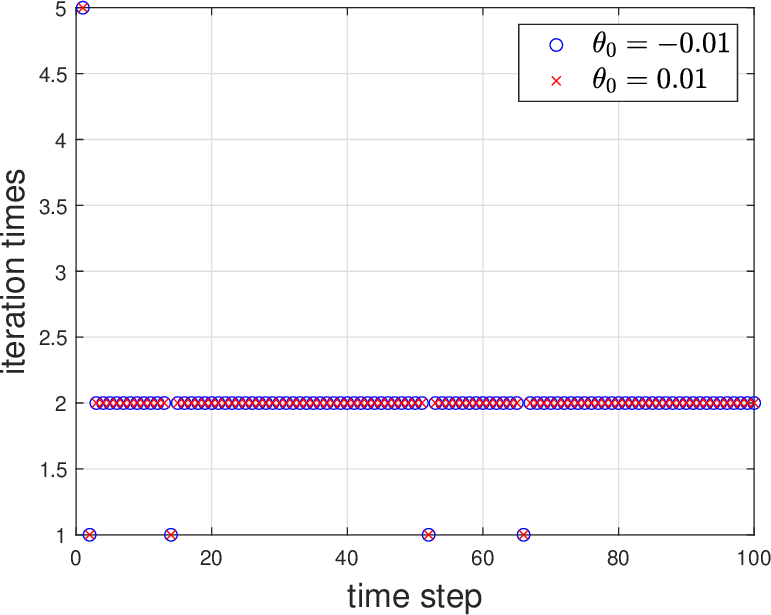}}\hfill
      \subfloat[$\theta_0=\pm9$]
    {\includegraphics[width=0.45\textwidth, height=0.36\textwidth,
      clip]{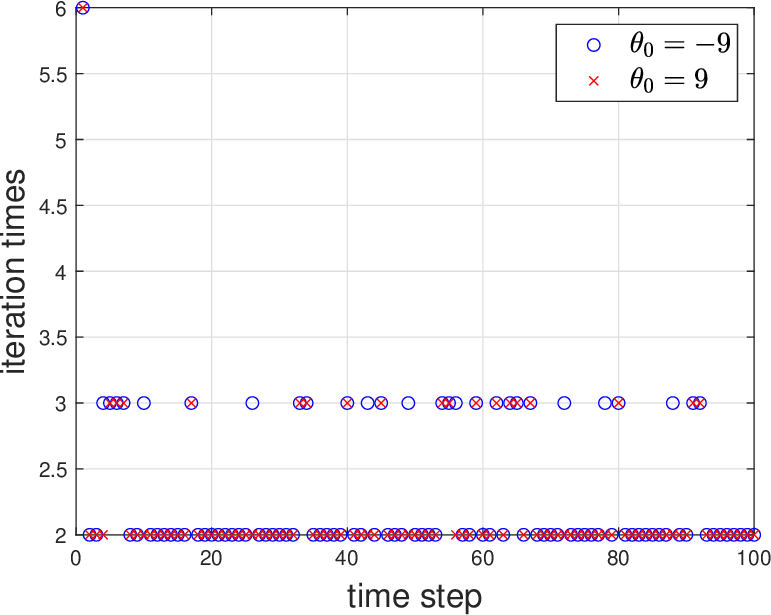}}\hfill
    \caption{(Efficiency of the Newton iteration method in App. \ref{app:Newton}) The Newton iteration times in each time step when solving \eqref{eq:exp_Newton}.}
    \label{fig:Newton}
\end{figure}

It is evident that for all $\theta_0$, the number of iterations is quite small. This validates the high efficiency of this Newton method in solving the nonlinear system \eqref{eq:newsystem}, regardless of whether the problem is in the near classical regime or the regime with a strong quantum effect.




\addcontentsline{toc}{section}{References}
\bibliographystyle{plain}
\bibliography{article}

\end{document}